\documentclass[11pt,reqno,a4paper]{amsart}
\usepackage[margin=1in]{geometry}
\usepackage{amsmath, amsthm, amssymb}
\usepackage{enumitem}
\usepackage[colorlinks=true, pdfstartview=FitV, linkcolor=blue,citecolor=blue, urlcolor=blue]{hyperref}
\usepackage[usenames,dvipsnames]{xcolor}
\usepackage{subcaption}
\usepackage{mathtools}
\usepackage{bbm}
\usepackage{upgreek}
\usepackage{comment}

\usepackage{dsfont}

\def\dd{{\rm d}}

\def\eps{\varepsilon}
\def\epsilon{\varepsilon}
\def\e{{\rm e}} 
\def\TS{\textstyle}

\def\RR {\mathbb{R}}

\def\NN {\mathbb{N}}

\def\Re{{\rm Re}}
\def\Im{{\rm Im}}
\newcommand{\norm}[1]{\left\lVert #1 \right\rVert}
\newcommand{\jap}[1]{\left\langle #1 \right\rangle}

\def\bu {\boldsymbol{u}}

\newcommand{\cA}{\mathcal{A}}

\newcommand{\cD}{\mathcal{D}}
\newcommand{\cE}{\mathcal{E}}
\newcommand{\cF}{\mathcal{F}}

\newcommand{\cH}{\mathcal{H}}
\newcommand{\cI}{\mathcal{I}}

\newcommand{\cK}{\mathcal{K}}
\newcommand{\cL}{\mathcal{L}}

\newcommand{\cN}{\mathcal{N}}
\newcommand{\cO}{\mathcal{O}}
\newcommand{\cP}{\mathcal{P}}

\newcommand{\cR}{\mathcal{R}}
\newcommand{\cS}{\mathcal{S}}
\newcommand{\cT}{\mathcal{T}}

\newcommand{\cX}{\mathcal{X}}
\newcommand{\cY}{\mathcal{Y}}
\newcommand{\cZ}{\mathcal{Z}}

\newcommand{\sfw}{\mathsf{w}}
\newcommand{\adv}{\mathsf{adv}}
\newcommand{\diff}{\mathsf{diff}}

\newcommand{\app}{\mathrm{app}}
\newcommand{\BS}{\mathrm{BS}}

\newcommand{\Ein}{\mathrm{Ein}}

\newcommand{\Ker}{\mathrm{Ker}}
\newcommand{\Ran}{\mathrm{Ran}}
\newcommand{\Rey}{\mathrm{Re}}
\newcommand{\mrm}{\mathrm{m}}
\newcommand{\mrM}{\mathrm{M}}
\newcommand{\mrI}{\mathrm{I}}

\newcommand{\err}{\mathrm{r}}

\newcommand{\1}{\mathbbm{1}}

\newcommand{\ww}{w}

\newtheorem{proposition}{Proposition}[section]
\newtheorem{theorem}[proposition]{Theorem}
\newtheorem{corollary}[proposition]{Corollary}
\newtheorem{lemma}[proposition]{Lemma}
\theoremstyle{definition}
\newtheorem{definition}[proposition]{Definition}
\newtheorem{remark}[proposition]{Remark}

\newtheorem{hypotheses}[proposition]{Hypotheses}
\theoremstyle{definition}

\numberwithin{equation}{section}

\title{The long way of a viscous vortex dipole}

\author{Michele Dolce}
\address{Institute of Mathematics, EPFL, Station 8, 1015 Lausanne, Switzerland}
\email{michele.dolce@epfl.ch}

\author{Thierry Gallay}
\address{Institut Fourier, Université Grenoble Alpes, CNRS, Institut Universitaire
de France, 38000 Grenoble, France}
\email{thierry.gallay@univ-grenoble-alpes.fr}

\begin{document}


\begin{abstract}
We consider the evolution of a viscous vortex dipole in $\RR^2$ originating from
a pair of point vortices with opposite circulations.  At high Reynolds number
$\Rey \gg 1$, the dipole can travel a very long way, compared to the distance
between the vortex centers, before being slowed down and eventually destroyed by
diffusion. In this regime we construct an accurate approximation of the solution
in the form of a two-parameter asymptotic expansion involving the aspect ratio
of the dipole and the inverse Reynolds number.  We then show that the exact solution 
of the Navier-Stokes equations remains close to the approximation on a time interval 
of length $\cO(\Re^\sigma)$, where $\sigma < 1$ is arbitrary. This improves 
upon previous results which were essentially restricted to $\sigma = 0$. As an 
application, we provide a rigorous justification of an existing formula which gives 
the leading order correction to the translation speed of the dipole due to finite 
size effects.
\end{abstract}

\maketitle


\section{Introduction}\label{sec:intro}

Numerical investigations of freely decaying two-dimensional turbulence show that 
the distribution of vorticity tends to concentrate in a relatively small fraction
of the spatial domain, so as to form a collection of coherent vortices which interact 
over a long period of time and grow in size due to diffusion and merging 
\cite{McW84,FoxDav2010}. In contrast, the small scales of the flow correspond to thin 
filaments of vorticity, which may be created during rare events such as vortex mergers.
A precise description of vortex interactions thus appears as a necessary step toward 
a better understanding of the dynamics of two-dimensional flows at high Reynolds 
number \cite{DiVer2002,Meunier2005}. 

As long as the distances between the vortex centers remain substantially
larger than the sizes of the vortex cores, the dynamics of a finite collection of
vortices in $\RR^2$ is well approximated by the {\em point vortex system}: 
\begin{equation}\label{eq:PV}
  z_i'(t) \,=\, \sum_{j \neq i} \frac{\Gamma_j}{2\pi}\,\frac{\bigl(z_i(t) - z_j(t)\bigr)^\perp
  }{|z_i(t) - z_j(t)|^2}\,, \qquad i,j = 1,\dots,N\,,
\end{equation}
where $z_1(t),\dots,z_N(t) \in \RR^2$ denote the positions of the vortex centers
and $\Gamma_1,\dots,\Gamma_N \in \RR$ are the circulations of the corresponding
vortices. The ODE system \eqref{eq:PV}, which was already studied by Helmholtz
and Kirchhoff, can be rigorously derived as an asymptotic model for the
evolution of sharply concentrated solutions of the two-dimensional Euler or
Navier-Stokes equations \cite{MarPul94,Mar98,Gallay2011,CeSeis24}. Note
that \eqref{eq:PV} only makes sense if $z_i(t) \neq z_j(t)$ when $i \neq j$,
but this condition is not always preserved under evolution. Indeed, if $N \ge 3$
and if the circulations $\Gamma_i$ do not have the same sign, there are
examples of solutions of \eqref{eq:PV} for which three vortices collide in
finite time \cite{Gr77}. However, the set of initial data leading to such
collisions is negligible in the sense of Lebesgue's measure \cite{MarPul94}. 

To formulate more precisely the relation between the point vortex system
\eqref{eq:PV} and the fundamental equations of fluid motion, we start from 
the two-dimensional vorticity equation
\begin{equation}\label{eq:NS}
  \partial_t\omega(x,t) + \bu(x,t)\cdot\nabla \omega(x,t) \,=\, \nu
  \Delta \omega(x,t)\,, \qquad x\in \RR^2\,,\quad t>0\,,
\end{equation}
where $\bu(x,t) = \bigl(u_1(x,t),u_2(x,t)\bigr)\in \RR^2$ denotes the velocity of the
fluid at point $x \in \RR^2$ and time $t > 0$, and $\omega := \partial_1 u_2 -
\partial_2 u_1$ is the associated vorticity. The constant parameter $\nu > 0$
represents the kinematic viscosity of the fluid. The velocity field $\bu$ is
divergence-free, and can be expressed in terms of the vorticity $\omega$ by
the {\em Biot-Savart formula} $\bu(\cdot,t) = \BS[\omega(\cdot,t)]$, where
\begin{equation}\label{def:BS}
  \BS[\omega](x) \,:=\, \frac{1}{2\pi}\int_{\RR^2} \frac{(x-y)^\perp}{
  |x-y|^2}\,\omega(y)\,\dd y\,, \qquad x \in \RR^2\,.
\end{equation}
As in \eqref{eq:PV}, if $x = (x_1,x_2) \in \RR^2$, we denote $x^\perp = (-x_2,x_1)$
and $|x|^2 = x_1^2 + x_2^2$. The most important conserved quantities for
the dynamics of \eqref{eq:NS} are the {\em mean} $\mrM[\omega]$ and the
{\em linear momentum} $\mrm[\omega] = (\mrm_1[\omega],\mrm_2[\omega])$ defined by
\begin{equation}\label{def:mean-mom}
  \mrM[\omega] \,=\, \int_{\RR^2}\omega(x)\,\dd x\,, \qquad 
  \mrm_i[\omega] \,=\, \int_{\RR^2} x_i\omega(x)\,\dd x, \qquad i=1,2\,.
\end{equation} 

It is known that the Cauchy problem for the evolution equation \eqref{eq:NS} is
globally well-posed if the initial vorticity $\omega^{in}$ is a finite Radon
measure on $\RR^2$ \cite{GaGa2005,BedMas2014}. In particular, given any integer 
$N \ge 1$, we can consider the discrete measure
\begin{equation}\label{def:initvort}
  \omega^{in} \,=\, \sum_{i = 1}^N \Gamma_i\,\delta(\cdot - x_i)\,,
\end{equation}
which represents a finite collection of point vortices of circulations
$\Gamma_1,\dots,\Gamma_N \in \RR$ located at the positions $x_1,\dots,x_N \in \RR^2$.
Without loss of generality, we assume that $\Gamma_i \neq 0$ and that $x_i \neq x_j$
when $i \neq j$. The point vortex system \eqref{eq:PV} with initial data $z_i(0) = x_i$ 
is then well posed on the time interval $[0,T]$ for some $T > 0$, and we denote
\begin{equation}\label{def:dGam}
  d \,:=\, \min_{t \in [0,T]}\min_{i \neq j}|z_i(t) - z_j(t)| \,>\, 0\,, \qquad
  |\Gamma| \,:=\, |\Gamma_1| + \dots + |\Gamma_N| \,>\, 0\,.
\end{equation}

The following statement is the starting point of our analysis.

\begin{theorem}\label{thm:gal11} {\bf \cite{Gallay2011}}
Assume that the point vortex system \eqref{eq:PV} with circulations
$\Gamma_i$ and initial positions $x_i$ is well posed on the time
interval $[0,T]$. There exists a constant $C_0 > 0$ such that the unique
solution $\omega^\nu$ of the vorticity equation \eqref{eq:NS} with initial
data \eqref{def:initvort} satisfies
\begin{equation}\label{eq:Nvortconv}
  \frac{1}{|\Gamma|}\int_{\RR^2} \Bigl|\,\omega^\nu(x,t) \,-\, \sum_{i=1}^N
  \frac{\Gamma_i}{4\pi\nu t}\,\e^{-\frac{|x - z_i(t)|^2}{4\nu t}}\Bigr|\,\dd x \,\le\,
  C_0\,\frac{\nu t}{d^2}\,, \qquad \text{for all}~\, t \in (0,T)\,,
\end{equation}
where $z_1(t)\,\dots,z_N(t)$ is the solution of \eqref{eq:PV} such that
$z_i(0) = x_i$ for $i = 1,\dots,N$.
\end{theorem}

This result shows that the solution $\omega^\nu$ of the vorticity equation
\eqref{eq:NS} with initial data \eqref{def:initvort} can be approximated by a
superposition of Lamb-Oseen vortices whose centers follow the evolution defined
by the point vortex system. The approximation is accurate as long as the radius
$\sqrt{\nu t}$ of the vortex cores is much smaller than the distance $d$ between
the centers. Importantly, the constant $C_0$ in \eqref{eq:Nvortconv} depends on
the normalized time $|\Gamma| T/d^2$, but not on the viscosity parameter $\nu > 0$. 
In particular, for any fixed $t \in (0,T]$, we can take the limit $\nu \to 0$ in
\eqref{eq:Nvortconv} so as to obtain the weak convergence result
\begin{equation}\label{eq:weakconv}
  \omega^\nu(\cdot,t) ~\xrightharpoonup[\nu \to 0]{\hbox to 8mm{}}~ 
  \sum_{i=1}^N \Gamma_i \,\delta\bigl(\cdot - z_i(t)\bigr)\,, \quad 
  \text{for all}~\, t \in [0,T]\,,
\end{equation}
which provides a natural link between the dynamics of the vorticity
equation \eqref{eq:NS} and the point vortex system \eqref{eq:PV}. 

Interesting questions arise when trying to extend the approximation result in
Theorem~\ref{thm:gal11} to longer time scales. The first one is related to the
collapse of point vortices. Assume for instance that the solution of system
\eqref{eq:PV} is defined on the maximal time interval $[0,T_*)$, with three
vortices collapsing at the origin as $t \to T_*$. In that case estimate
\eqref{eq:Nvortconv} is valid on the time interval $(0,T)$ for any $T < T_*$,
but the constant $C_0 > 0$ blows up as $T \to T_*$ because the distance $d$
converges to zero in this limit. So it is not clear at all if the weak
convergence \eqref{eq:weakconv} holds up to collision time, although this is
certainly a reasonable conjecture.  A fortiori, we do not know if a limiting
procedure as in \eqref{eq:weakconv} can be used to define a continuation of 
the point vortex dynamics after collapse.

In a different direction, one may consider global solutions of the point vortex
system, for which $|z_i(t) - z_j(t)| \ge d > 0$ for all $t \ge 0$ if $i \neq j$.
This is the case, for instance, if $N = 2$ or if the circulations $\Gamma_i$ are 
all positive \cite{MarPul94}.  Here again the approximation result
\eqref{eq:Nvortconv} can be applied on any time interval $(0,T)$, but the
constant $C_0$ depends on $T$ and increases rapidly as $T \to +\infty$. In
particular, it is not clear if the solution $\omega^\nu(x,t)$ is well
approximated by a superposition of Lamb-Oseen vortices as long as
$\nu t \ll d^2$, namely whenever the radius of the vortex cores is 
small compared to the distance $d$ between the centers.

The present paper is a contribution to the study of the latter question in the
particular case where $N = 2$ and $\Gamma_1 + \Gamma_2 = 0$, which corresponds
to a {\em vortex dipole}. To be specific, given $\Gamma > 0$ and $d > 0$, we
consider initial data of the form
\begin{equation}\label{eq:omnuin}
  \omega^{in} \,=\, \Gamma \delta(\cdot -x_\ell) - \Gamma\delta(\cdot-x_\err)\,,
  \qquad \text{where} \quad x_\ell\,=\, \Bigl(-\frac{d}{2}\,,0\Bigr)\,, \quad x_\err \,=\,
  \Bigl(\frac{d}{2}\,,0\Bigr)\,,
\end{equation}
so that $\omega^{in}$ represents a pair of point vortices of circulations
$\pm \Gamma$ separated by a distance $d$. In that situation, the point vortex
dynamics predicts a uniform translation at speed $V = \Gamma/(2\pi d)$ in a
direction normal to the line segment joining the vortex centers.  More
precisely, the positions $z_\ell(t), z_\err(t)$ of the point vortices with 
circulation $\Gamma$, $-\Gamma$ (respectively) are given by
\begin{equation}\label{eq:HKevol}
  z_\ell(t) \,=\, \Bigl(-\frac{d}{2}\,,Z_2(t)\Bigr)\,, \qquad
  z_\err(t) \,=\, \Bigl(\frac{d}{2}\,,Z_2(t)\Bigr)\,,
\end{equation}
where $Z_2(t) = Vt$. Theorem~\ref{thm:gal11} thus asserts that the solution
$\omega^\nu$ of \eqref{eq:NS} with initial data \eqref{eq:omnuin} is well
approximated, on any fixed time interval $(0,T)$, by the viscous vortex dipole
\begin{equation}\label{eq:omapp}
  \omega_\app^\nu(x,t) \,=\, \frac{\Gamma}{4\pi\nu t}\Bigl(\e^{-\frac{|x - z_\ell(t)|^2}{4\nu t}}
  - \e^{-\frac{|x - z_\err(t)|^2}{4\nu t}}\Bigr)\,, \qquad x \in \RR^2\,, \quad t > 0\,,
\end{equation}
provided $\nu T \ll C_0^{-1} d^2$ where the constant $C_0$ depends on $\Gamma T/d^2$. 

To discuss the ultimate validity of such an approximation, we introduce 
as in \cite{GaSring} the following quantities, which play an important role in 
our analysis:  
\begin{equation}\label{def:all}
  \eps(t) \,=\, \frac{\sqrt{\nu t}}{d}\,, \qquad 
  \delta \,=\, \frac{\nu}{\Gamma}\,, \qquad
  T_\adv \,=\, \frac{d^2}{\Gamma}\,, \qquad 
  T_\diff \,=\, \frac{d^2}{\nu}\,.
\end{equation}
Here $\eps(t)$ is the {\em aspect ratio} of the viscous vortex dipole
\eqref{eq:omapp} at time $t > 0$, namely the ratio of the size of the vortex
cores to the distance between the vortex centers. The dimensionless parameter
$\delta$, which measures the relative strength of the viscous and the advection
effects, is usually called the {\em inverse Reynolds number}. In view of
\eqref{eq:HKevol}, the {\em advection time} $T_\adv$ is the time needed for the
vortex dipole to cover the distance $V T_\adv = d/(2\pi)$. Finally, the {\em
diffusion time} $T_\diff$ is the time at which the aspect ratio $\eps(t)$
becomes equal to one, so that the support of both vortices strongly overlap. We
are interested in the high Reynolds number regime $\delta \ll 1$, for which
$T_\adv \ll T_\diff$. In that situation, the viscous vortex dipole can travel
over a very long distance in the vertical direction before being destroyed by
diffusion.

Our main result shows that the solution of \eqref{eq:NS} can be approximated by
a viscous vortex dipole over a long time interval of size $T_\adv  \delta^{-\sigma}$,
where $0 < \sigma < 1$, provided that the position $Z_2(t)$ of the vortex centers
\eqref{eq:HKevol} is chosen in a suitable way. 

\begin{theorem}\label{thm:main} Fix $\sigma \in [0,1)$. There exist positive constants
$C$ and $\delta_0$ such that, for any $\Gamma > 0$, any $d > 0$ and any $\nu > 0$
with $\delta := \nu/\Gamma \le \delta_0$, the solution $\omega^\nu$ of \eqref{eq:NS}
with initial data \eqref{eq:omnuin} satisfies
\begin{equation}\label{eq:2vortconv}
  \frac{1}{\Gamma}\int_{\RR^2} \bigl|\,\omega^\nu(x,t) - \omega_\app^\nu(x,t)\bigr|\,\dd x 
  \,\le\, C\,\frac{\nu t}{d^2}\,, \qquad \text{for all}~\, t \in \bigl(0,T_\adv
  \delta^{-\sigma}\bigr)\,.
\end{equation}
Here $\omega_\app^\nu$ is the viscous vortex dipole \eqref{eq:omapp}, \eqref{eq:HKevol}
and the vertical position $Z_2(t)$ of the vortex centers is a smooth function 
satisfying $Z_2(0) = 0$ and 
\begin{equation}\label{eq:Zpdef}
  Z_2'(t) \,=\, \frac{\Gamma}{2\pi d}\Bigl(1 - 2\pi \alpha\,\epsilon^4  + \cO\bigl(
  \epsilon^5  + \delta^2 \epsilon + \delta \epsilon^2\bigr)\Bigr)\,, 
  \qquad \text{with}~\, \alpha \,\approx\, 22.24\,.
\end{equation}
\end{theorem}

\begin{remark}\label{rem:Vspeed}
The formula \eqref{eq:Zpdef} shows that the leading order correction
to the constant velocity $V$ predicted by the point vortex system is negative
and proportional to $\epsilon^4$, where $\epsilon$ is as in \eqref{def:all}.
This fact is known in the literature, see \cite{HaNaFu2018} for a detailed
discussion. The constant $\alpha > 0$ has an explicit expression in terms 
of the solution of a linear differential equation on $\RR_+$, see
Section~\ref{ssec34} below. While Theorem~\ref{thm:main} asserts the existence
of {\em some} function $Z_2(t)$ satisfying \eqref{eq:Zpdef} such that estimate
\eqref{eq:2vortconv} holds, it is important to mention that $Z_2(t)$ is
not entirely characterized by \eqref{eq:Zpdef}, unless $\sigma$ is sufficiently
small. 
\end{remark}

If $\sigma = 0$, Theorem~\ref{thm:main} is merely a particular case of
Theorem~\ref{thm:gal11}, and instead of \eqref{eq:Zpdef} we can simply
take $Z_2'(t) = V = \Gamma/(2\pi d)$. The situation is different when
$\sigma \in (0,1)$, because we have to control the solution of \eqref{eq:NS}
on a much longer time scale. According to \eqref{eq:Zpdef}, the total
distance $D$ covered by the vortex dipole over the time interval 
$(0,T_\adv\delta^{-\sigma})$ can be estimated as
\[
  D \,=\, Z_2\bigl(T_\adv\delta^{-\sigma}\bigr) \,\approx\, \frac{\Gamma}{2\pi d}\,
  T_\adv\delta^{-\sigma} \,=\, \frac{d}{2\pi}\,\delta^{-\sigma}\,,
\]
so that $D \gg d$ if $\delta > 0$ is sufficiently small. To prove the validity
of \eqref{eq:2vortconv}, it is therefore necessary to have a very accurate
expression of the vertical velocity $Z_2'(t)$, as can be seen from the following
back-of-the-envelope calculation. As is easy to verify, changing the vertical
position $Z_2(t)$ of the vortex dipole \eqref{eq:omapp} by a small quantity $z(t)$
produces an error proportional to $|z(t)|/(d \epsilon)$ when computing the
left-hand side of \eqref{eq:2vortconv}. This means that we have to know $Z_2(t)$
with a precision of order $d \epsilon^3$ to ensure that the right-hand side
of \eqref{eq:2vortconv} remains $\cO(\epsilon^2)$. Since the length of the time
interval is $T_\adv\delta^{-\sigma}$, it is sufficient to control the vertical
velocity $Z_2'(t)$ up to an error of order $(\Gamma/d)\epsilon^3 \delta^\sigma$. 
To relate this to \eqref{eq:Zpdef}, we observe that
\begin{equation}\label{eq:epsdel}
  \epsilon(t)^2 \,=\, \frac{\nu t}{d^2} \,=\, \frac{\delta t}{T_\adv}\,,
  \qquad \text{so that}\quad \epsilon(t)^2 \,\le\, \epsilon\bigl(
  T_\adv \delta^{-\sigma}\bigr)^2 \,=\, \delta^{1-\sigma}\,.
\end{equation}
Given $M \in \NN$, we thus have $\epsilon^{3+M} \le \epsilon^3\,\delta^{M(1-\sigma)/2}
\le \epsilon^3\,\delta^\sigma$ if $\sigma \le M/(M+2)$. So, for instance, if
we know that $Z_2'(t) = \Gamma/(2\pi d) + \cO(\epsilon^4)$, that information
is sufficient to obtain estimate \eqref{eq:2vortconv} provided $\sigma \le 1/3$.
Similarly, the improved formula \eqref{eq:Zpdef} which includes the leading
order correction to the point vortex dynamics is good enough for our
purposes provided that $\sigma \le 1/2$, but for larger values of $\sigma$
we need to derive a higher order approximation of the vertical speed. 

As in the previous works \cite{Gallay2011,GaSring}, the proof of
Theorem~\ref{thm:main} relies on the construction of an approximate solution of
\eqref{eq:NS}, in the form of a power series expansion in the small parameters
$\epsilon$ and $\delta$. Our expansion involves the vorticity distribution in
suitable self-similar variables, as well as the vertical speed of the vortex
dipole. A linear approximation in $\delta$ turns out to be sufficient, but the
order $M$ of the expansion in $\epsilon$ must satisfy $M > (3+\sigma)/(1-\sigma)$.
As $\sigma$ can be chosen arbitrarily close to $1$, this means that we have to
construct an asymptotic expansion to {\em arbitrarily high order} in $\epsilon$,
which can be done by an iterative procedure that is described in
Section~\ref{sec:appsol}. Note that a fourth-order expansion was sufficient 
in the situations considered in \cite{Gallay2011,GaSring}, because it was 
assumed that $\sigma = 0$ in \cite{Gallay2011} and $0 < \sigma \ll 1$ 
in \cite{GaSring}. As an aside we mention that, in the inviscid case $\delta = 0$, 
our approximate solution shares many similarities with exact traveling wave 
solutions of the two-dimensional Euler equation describing vortex pairs, which 
were constructed by variational methods or perturbation arguments, see
\cite{Bur88,BNLL13,CLZ21,DdPMP23,HmMa17}.

The second step in the proof consists in showing that the exact solution of the
vorticity equation \eqref{eq:NS} remains close to the approximation constructed
in the first step. This is far from obvious because the linearized equation at
the approximate solution contains very large advection terms that could
potentially trigger violent instabilities, and also because we need a control
over the long time interval $(0,T_\adv\delta^{-\sigma})$. We follow here the
earlier study \cite{GaSring} which considers the related, but more complicated,
problem of an axisymmetric vortex ring in the low viscosity limit. The idea
is to control the difference between the solution and its approximation using
suitable energy functionals, which incorporate the necessary information about
the stability of the approximate solution. Unlike in \cite{Gallay2011}, it does
not seem possible to rely on weighted enstrophy estimates only, so we
follow the general approach introduced by Arnold to study the stability of
steady states for the 2D Euler equations, see \cite{GaSarnold}. This
method gives a clear roadmap to design a nonlocal energy-like functional
whose evolution is barely affected by the most dangerous terms in the
linearized equation at the approximate solution. 

There are numerous previous works devoted to the study of localized solutions
of the viscous vorticity equation \eqref{eq:NS}. Most of them consider initial
data that are contained in a disjoint union of balls of radius
$\epsilon_0 \ll 1$, and show that the corresponding solutions remain essentially
concentrated in small regions that evolve according to the point vortex
dynamics, see \cite{Mar98,CeSe18}. The effect of the viscosity is often treated
perturbatively, under the assumption that $\nu \le C \epsilon_0^\alpha$ for some
positive constants $C$ and $\alpha$. The solution is usually controlled on a
time interval $[0,T]$ that is fixed independently of $\epsilon_0$ \cite{Mar98},
but if the solution of the point vortex system is globally defined it is
possible to take $T \ge c\,|\log \epsilon_0|$ \cite{CeSe18}. To our knowledge,
the only previous work where the viscosity can be chosen independently of
$\epsilon_0$, so that initial data of the form \eqref{def:initvort} are allowed,
is \cite{CeSeis24}. Here also, under the assumption that the point vortices do
not collapse in finite time, the solution can be controlled on the time interval
$(0,T)$ with $T \ge c\,|\log(\epsilon_0^2 + \nu)|$. Taking $\epsilon_0 = 0$ we
obtain $T \ge c\,|\log \nu|$ for the initial data \eqref{def:initvort},
whereas Theorem~\ref{thm:main} describes the solution over a much longer
time scale $T = \cO (\nu^{-\sigma})$, but only in the simple example
of a vortex dipole. 

To conclude this introduction, we mention that the techniques 
developed in the present paper have recently been extended to study 
vortex pairs with general circulations \cite{ZZ25}, as well as more 
complicated configurations called vortex crystals \cite{DD25}.


\section{The rescaled vorticity equation}\label{sec:RVE}

In this section, we introduce the framework needed to prove
Theorem~\ref{thm:main}. Following \cite{Gallay2011} we first define
self-similar variables which allow us to desingularize the dynamics of
\eqref{eq:NS} at initial time, when point vortices are considered as initial
data. We thus obtain a rescaled system describing a ``zoomed-in" evolution for
the vorticity distribution, with the property that the initial data are smooth
in the new variables. We next make an appropriate choice for the positions of
the vortex cores, which ensures that the rescaled vorticity distribution has
vanishing first order moments. 
\subsection{Self-similar variables}\label{ssec21}

Our goal is to study the solution $\omega$ of the vorticity 
equation~\eqref{eq:NS} with initial data $\omega^{in} = \Gamma
\delta(\cdot -x_\ell) - \Gamma\delta(\cdot-x_\err)$, where
$\Gamma > 0$ is the circulation parameter and the initial positions 
$x_\err, x_\ell$ are given by \eqref{eq:omnuin}. For symmetry
reasons, the solution $\omega(x_1,x_2,t)$ is an odd function of $x_1$
for all times. It will be convenient to use the notation  
\begin{equation}\label{def:tilde}
  \widetilde{x} \,=\, (-x_1,x_2)\,, \qquad \text{for any } x \,=\, (x_1,x_2)
  \in \RR^2\,.
\end{equation}
We introduce self-similar variables $\xi_\err$ (for the right vortex) and
$\xi_\ell$ (for the left vortex) defined by
\begin{equation}\label{def:change}
  \xi_\err \,=\, \frac{x-Z(t)}{\sqrt{\nu t}}\,, \qquad \xi_\ell \,=\,
  \frac{\widetilde{x}-Z(t)}{\sqrt{\nu t}}\,,
\end{equation}
where $Z(t) \in \RR^2$ is the time-dependent location of the
right vortex, which is a free parameter at this stage. In view of
the Helmholtz-Kirchhoff dynamics, we anticipate that 
\begin{equation}\label{def:Zt}
  Z(t) \,=\, \left(\frac{d}{2}\,,\,Z_2(t)\right)\,, \qquad \text{where}\quad
  Z_2(t) \,\approx\, \frac{\Gamma t}{2\pi d}\,.
\end{equation}
Using the conserved quantities of \eqref{eq:NS}, we verify below that
the assumption $Z_1(t) = d/2$ is indeed a natural one, and we also make
an appropriate choice for the vertical position $Z_2(t)$. 

For the time being, we look for a solution of equation~\eqref{eq:NS} in
the form
\begin{equation}\label{eq:ansom}
  \omega(x,t) \,=\, -\frac{\Gamma}{\nu t}\,\Omega(\xi_\err,t) +
  \frac{\Gamma}{\nu t}\,\Omega(\xi_\ell,t)\,,
\end{equation}
where the self-similar variables $\xi_\err, \xi_\ell$ are defined in 
\eqref{def:change}, and the rescaled vorticity $\Omega(\xi,t)$ is a new
dependent variable to be determined. It is important to realize that
the same function $\Omega$ is used to describe the vorticity distribution
of both the right and the left vortex, which reflects the fact that
the solution $\omega$ of \eqref{eq:NS} is an odd function of the
variable $x_1$. The corresponding decompositions for the velocity field
$\bu = (u_1,u_2)$ and the stream function $\psi$ are found to be 
\begin{equation}\label{eq:anspsi}
  \bu(x,t) \,=\, -\frac{\Gamma}{\sqrt{\nu t}}\,U(\xi_\err,t) - \frac{\Gamma}{
  \sqrt{\nu t}}\,\widetilde{U}(\xi_\ell,t), \qquad 
 \psi(x,t) = -\Gamma\,\Psi(\xi_\err,t) + \Gamma\,\Psi(\xi_\ell,t)\,,
\end{equation}
where it is understood that
\begin{equation}\label{UPsirel}
  U \,=\, \nabla^\perp \Psi ~\text{ and }~ \Psi = \Delta^{-1}\Omega\,,
  \quad\text{with }~
  (\Delta^{-1}\Omega)(\xi) \,:=\,
  \frac{1}{2\pi}\int_{\RR^2}\log|\xi-\eta|\,\Omega(\eta)\,\dd \eta\,.
\end{equation}
In what follows we write $U = \BS[\Omega]$ and $\Psi = \Delta^{-1}\Omega$
when \eqref{UPsirel} holds. Using \eqref{eq:anspsi} it is easy to verify
that $u_1, \psi$ are odd functions of the variable $x_1$ whereas
$u_2$ is an even function of $x_1$.

To formulate the equation satisfied by $\Omega(\xi,t)$, we introduce the
rescaled diffusion operator 
\begin{align}\label{def:cL}
  \cL \,:=\, \Delta_{\xi}+ \frac12\xi\cdot\nabla_\xi +1\,.
\end{align}
If $f,g$ are $C^1$ functions on $\RR^2$, we also define the Poisson bracket 
\begin{equation}\label{def:PB}
  \{f,g\} \,=\, \nabla^\perp f\cdot \nabla g \,=\, (\partial_1 f)(\partial_2 g)
  - (\partial_2 f)(\partial_1 g)\,. 
\end{equation}

\begin{lemma}
The vorticity distribution $\omega$ defined by \eqref{eq:ansom} is a solution
of \eqref{eq:NS} provided the function $\Omega$ satisfies the evolution
equation
\begin{equation}\label{eq:Omss}
  t\partial_t\Omega \,=\, \cL\Omega + \frac{1}{\delta}\Bigl\{\Psi - \cT_\eps\Psi
  +\frac{\eps d\xi_1}{\Gamma}\,Z_2'\,,\,\Omega\Bigr\}\,, \quad t > 0\,,
\end{equation}
where $\eps = \sqrt{\nu t}/d$, $\delta = \nu/\Gamma$, and $\cT_\eps$
is the translation/reflection operator defined by
\begin{equation}\label{def:Teps}
  (\cT_\eps \Psi)(\xi) \,=\, \Psi\bigl(\widetilde\xi - \eps^{-1}e_1 \bigr)\,,
  \quad e_1 = (1,0)\,.
\end{equation}
\end{lemma}

\begin{proof}
We write equation \eqref{eq:NS} in the form $\partial_t\omega + \{\psi,\omega\}
= \nu\Delta\omega$, where $\psi = \Delta^{-1}\omega$ is the stream function.
Using \eqref{eq:ansom} it is easy to verify that
\begin{align*}
  \partial_t \omega -\nu\Delta\omega \,=\, \frac{\Gamma}{\nu t^2}
  \biggl[ \Bigl(-t\partial_t\Omega + \cL \Omega + \sqrt{\frac{t}{\nu}}\,Z'\cdot
  \nabla\Omega\Bigr)(\xi_\err,t) + \Bigl(t\partial_t\Omega - \cL \Omega
  - \sqrt{\frac{t}{\nu}}\,Z'\cdot\nabla\Omega\Bigr)(\xi_\ell,t)\biggr]\,.
\end{align*}
On the other hand, in view of \eqref{def:tilde}--\eqref{def:Zt}, the self-similar
variables $\xi_\err, \xi_\ell$ satisfy the relations
\[
  \xi_\ell \,=\, \widetilde{\xi}_\err - \eps^{-1}e_1\,, \qquad
  \xi_\err \,=\,\widetilde{\xi}_\ell - \eps^{-1}e_1\,,
\]
so that $\Psi(\xi_\ell,t) = (\cT_\eps \Psi)(\xi_\err,t)$ and
$\Psi(\xi_\err,t) = (\cT_\eps \Psi)(\xi_\ell,t)$ by \eqref{def:Teps}. 
In view of \eqref{eq:anspsi}, \eqref{eq:ansom}, we thus have
\begin{align*}
  \bigl\{\psi,\omega\bigr\}_x \,&=\, \frac{\Gamma^2}{\nu t}\biggl[\Bigl\{\Psi(\xi_\err,t)
  - \cT_\eps\Psi(\xi_\err,t)\,,\,\Omega(\xi_\err,t)\Bigr\}_x
  + \Bigl\{\Psi(\xi_\ell,t) - \cT_\eps\Psi(\xi_\ell,t)\,,\,\Omega(\xi_\ell,t)
  \Bigr\}_x\biggr]\\
  \,&=\, \frac{\Gamma^2}{\nu^2 t^2}\biggl[\Bigl\{\Psi - \cT_\eps\Psi
  \,,\,\Omega\Bigr\}_\xi(\xi_\err,t) - \Bigl\{\Psi - \cT_\eps\Psi
  \,,\,\Omega\Bigr\}_\xi(\xi_\ell,t)\biggr]\,.
\end{align*}
Here we write $\{\cdot,\cdot\}_x$ or $\{\cdot,\cdot\}_\xi$ according
to whether the Poisson bracket is computed with respect to the
original variable $x$ or the rescaled variable $\xi$. Recalling that
$\delta = \nu/\Gamma$, we deduce from the equalities above that 
\eqref{eq:NS} is satisfied provided
\begin{equation}\label{eq:Omprelim}
  t\partial_t\Omega \,=\, \cL\Omega + \frac{1}{\delta}\Bigl\{\Psi-
  \cT_\eps\Psi\,,\,\Omega\Bigr\} + \sqrt{\frac{t}{\nu}}\,Z'\cdot
  \nabla\Omega\,, \quad t > 0\,.
\end{equation}
Since $\sqrt{t/\nu} = (\eps d)/(\Gamma\delta)$ and $Z'\cdot \nabla\Omega
= Z_2' \partial_2\Omega = Z_2' \{\xi_1,\Omega\}$, we deduce \eqref{eq:Omss}
from \eqref{eq:Omprelim}. 
\end{proof}

\begin{remark}\label{rem:RDE}
The evolution equation \eqref{eq:Omss} can be written in the equivalent
form
\begin{equation}\label{eq:OmssU}
  t\partial_t\Omega \,=\, \cL\Omega + \frac{1}{\delta}\Bigl(
  U + \mathbf{T}_\eps U + \frac{\eps d}{\Gamma}\,Z'(t)\Bigr)\cdot
  \nabla\Omega\,, \quad t > 0\,,
\end{equation}
where $U = \nabla^\perp \Psi$ is the velocity field obtained from $\Omega$
via the usual Biot-Savart law in $\RR^2$, and $\mathbf{T}_\eps U = -\nabla^\perp
\cT_\eps\Psi$ denotes the velocity generated by the mirror vortex, namely
\begin{equation}\label{def:Uelleps}
  (\mathbf{T}_\eps U)(\xi,t) \,=\, \widetilde{U}\bigl(\widetilde{\xi}-\eps^{-1}
  e_1,t\bigr)\,.
\end{equation}
Note that both velocity fields $U$ and $\mathbf{T}_\eps U$ are divergence
free. 
\end{remark}

It is essential to realize that the evolution equation \eqref{eq:Omss} or
\eqref{eq:OmssU} is singular at initial time $t = 0$, because the term
$t\partial_t\Omega$ involving the time derivative vanishes. In particular, the
initial data $\Omega_0$ cannot be chosen arbitrarily, as can be seen by taking
the limit $t \to 0$ (hence also $\eps \to 0$) in \eqref{eq:Omss}. Using the fact
that $\nabla\cT_\eps \Psi \to 0$ in that limit, which is established in 
Lemma~\ref{lem:Teps} below, we obtain the relation
\begin{equation}\label{eq:Om0}
  0 \,=\, \cL \Omega_0 + \frac{1}{\delta}\bigl\{\Psi_0,\Omega_0\bigr\}\,,
  \quad \text{where} \quad \Delta\Psi_0 \,=\, \Omega_0\,.
\end{equation}
This is exactly the equation satisfied by the profile of a self-similar
solution of the Navier-Stokes equation \eqref{eq:NS}. It thus follows
from \cite[Prop.~1.3]{GalWay2005} that any solution of \eqref{eq:Om0} that
is integrable over $\RR^2$ is necessarily a multiple of the {\em Lamb-Oseen vortex}
\begin{equation}\label{def:G}
  \Omega_0(\xi) \,=\, G(\xi), \qquad \text{where} \quad G(\xi) \,=\, \frac{1}{4\pi}
  \,\e^{-|\xi|^2/4}\,. 
\end{equation}
A direct calculation shows that $\cL \Omega_0 = 0$, and since $\Omega_0, \Psi_0$
are both radially symmetric one has $\{\Psi_0,\Omega_0\} = 0$ too. Note that
$\Omega_0$ is normalized so that $\int_{\RR^2}\Omega_0\,\dd\xi = 1$, which
ensures that the vorticity distribution \eqref{eq:ansom} satisfies 
$\omega(\cdot,t) \rightharpoonup \omega^{in}$ as $t \to 0$. For later use,
we recall that the velocity field $U_0$ and the stream function $\Psi_0$
associated with $\Omega_0$ have the following expressions
\begin{equation}\label{def:UG}
  U_0(\xi) \,=\, U^G(\xi) \,:=\, \frac{1}{2\pi}
  \frac{\xi^\perp}{|\xi|^2}\,\Bigl(1-{\rm e}^{-|\xi|^2/4}\Bigr)\,, \qquad
  \Psi_0(\xi) \,=\, \frac{1}{4\pi}\Ein\bigl(|\xi|^2/4\bigr) -
  \frac{\gamma_E}{4\pi}\,,
\end{equation}
where $\Ein(s) = \int_0^s (1-\e^{-\tau})/\tau\,\dd \tau$ is the exponential
integral, and $\gamma_E= 0,5772\dots$ is the Euler-Mascheroni constant.

It follows from the previous works \cite{GaGa2005,Gallay2011,BedMas2014} that
the rescaled vorticity equation \eqref{eq:Omss} has a unique global solution
with initial data $\Omega_0$ given by \eqref{def:G}.  This solution is smooth for
positive times, and is continuous up to $t = 0$ in the weighted $L^2$ space
defined by \eqref{def:cY} below. Our purpose here is to give an accurate
description of that solution in the regime where the viscosity parameter
is small. 

\subsection{Formula for the vertical speed}\label{ssec22}

There is no canonical definition of the center of a viscous vortex. Reasonable
possibilities are the center of vorticity, the maximum point of the vorticity
distribution, or the stagnation point of the flow, but these notions do not
coincide in general. In our approach the vortex center is defined as the 
point $Z(t) \in \RR^2$ where the self-similar variable $\xi_\err$ is centered, 
see \eqref{def:change}. We choose it to be the center of vorticity, which means 
that the function $\Omega(\xi,t)$ satisfies the moment conditions
\begin{equation}\label{eq:momcond}
  \int_{\RR^2} \xi_1 \Omega(\xi,t)\,\dd\xi \,=\, 0\,, \quad \text{and}\quad
  \int_{\RR^2} \xi_2 \Omega(\xi,t)\,\dd\xi \,=\, 0\,. 
\end{equation}
In fact the first condition in \eqref{eq:momcond} is fulfilled if
$Z_1(t) = d/2$, which we already assumed in \eqref{def:Zt}. The second condition
requires a specific choice of the vertical position $Z_2(t)$, which we now
describe. Here and in what follows we use the notation \eqref{def:mean-mom}
for the moments of $\Omega$.

\begin{lemma}\label{lem:momZ2}
Let $\Omega$ be the solution of \eqref{eq:Omss} with initial data \eqref{def:G},
and $U = \nabla^\perp\Psi$ be the associated velocity field. If for all $t > 0$ the
vertical velocity is chosen so that
\begin{equation}\label{def:Z2}
  Z_2'(t) \,=\, -\frac{\Gamma}{\eps d}\int_{\RR^2}U_2\bigl(\widetilde{\xi}-\eps^{-1}
  e_1,t\bigr)\Omega(\xi,t)\,\dd\xi \,\equiv\, \frac{\Gamma}{\eps d}\int_{\RR^2}
  \bigl(\partial_1 \cT_\eps \Psi\bigr)(\xi,t)\Omega(\xi,t)\,\dd\xi\,,
\end{equation} 
then the following moment identities hold true
\begin{equation}\label{eq:momid}
  \mathrm{M}[\Omega(t)] \,=\,1\,, \qquad \mathrm{m}_1[\Omega(t)] \,=\,
  \mathrm{m}_2[\Omega(t)] \,=\, 0\,, \qquad t > 0\,.
\end{equation}
\end{lemma}

\begin{proof}
It is easy to verify that $t\partial_t \int_{\RR^2}\Omega\,\dd\xi = 0$, because the
diffusion operator \eqref{def:cL} is mass preserving and the velocity field
$U +  \mathbf{T}_\eps U + (\eps d/\Gamma)Z'(t)$ in \eqref{eq:OmssU} is divergence
free. Since $\Omega_0$ is normalized so that $\mathrm{M}[\Omega_0] = 1$, it follows
that $\mathrm{M}[\Omega(t)] = 1$ for all $t \ge 0$. 

We next consider the evolution of the moment $\mathrm{m}_j[\Omega]$ for $j = 1,2$. 
We observe that
\[
  \int_{\RR^2}\xi_j\cL\Omega\,\dd \xi \,=\ -\frac12\int_{\RR^2}\xi_j\Omega\,\dd\xi\,,
  \quad \text{and}\quad \int_{\RR^2}\xi_j U\cdot\nabla\Omega\,\dd\xi \,=\,
  - \int_{\RR^2}U_j \Omega\,\dd\xi \,=\, 0\,.
\]
Indeed the first equality is easily obtained using the definition \eqref{def:cL}
and integrating by parts. The second one is a well-known property of the
Biot-Savart kernel in $\RR^2$, related to the conservation of the moment
$\mathrm{m}_j[\omega]$ for the original Navier-Stokes equation \eqref{eq:NS}.
It follows that
\begin{equation}\label{eq:mjdot}
  t\partial_t \mathrm{m}_j[\Omega] \,=\, -\frac12\mathrm{m}_j[\Omega]
  - \frac{1}{\delta}\int_{\RR^2} \bigl(\mathbf{T}_\eps U\bigr)_j \,\Omega\,\dd\xi
  - \frac{\eps d}{\delta\Gamma}\,Z_j'(t)\,, \quad t > 0\,.
\end{equation}
At this point, the idea is to choose $Z_j'(t)$ so that the last two terms
of \eqref{eq:mjdot} cancel, namely
\begin{equation}\label{eq:Zjdef}
  Z_j'(t) \,=\, -\frac{\Gamma}{\eps d}\int_{\RR^2} \bigl(\mathbf{T}_\eps U\bigr)_j
  (\xi,t)\,\Omega(\xi,t)\,\dd\xi\,, \quad t > 0\,.
\end{equation}
Equation \eqref{eq:mjdot} then reduces to $t\partial_t\mathrm{m}_j[\Omega]
= -\frac12\mathrm{m}_j[\Omega]$, which means that $t^{1/2}\mathrm{m}_j[\Omega(t)]$
is independent of time. That quantity obviously converges to zero as $t \to 0$,
and we conclude that $\mathrm{m}_j[\Omega(t)] = 0$ for all times, which gives
\eqref{eq:momid}. 

To understand the precise meaning of \eqref{eq:Zjdef}, we write \eqref{def:Uelleps}
in the more explicit form
\[
  \bigl(\mathbf{T}_\eps U\bigr)_1(\xi_1,\xi_2,t) \,=\, -U_1\bigl(-\xi_1-\eps^{-1},\xi_2,t
  \bigr)\,, \qquad
  \bigl(\mathbf{T}_\eps U\bigr)_2(\xi_1,\xi_2,t) \,=\, U_2\bigl(-\xi_1-\eps^{-1},\xi_2,t
  \bigr)\,. 
\]
In particular, using the Biot-Savart formula, we find for $j = 1$: 
\begin{align*}
  \int_{\RR^2} \bigl(\mathbf{T}_\eps U\bigr)_1\,\Omega\,\dd\xi 
  \,&=\, -\int_{\RR^2} U_1\bigl(-\xi_1-\eps^{-1},\xi_2\bigr)\Omega(\xi)\,\dd\xi\\
  \,&=\, \frac{1}{2\pi}\iint_{\RR^4} \frac{\xi_2-\eta_2}{|\xi_1+\eps^{-1}+\eta_1|^2
  +|\xi_2-\eta_2|^2}\,\Omega(\xi)\,\Omega(\eta)\,\dd\xi\,\dd \eta \,=\,0\,,
\end{align*}
where in the last equality we used the fact that the integrand is odd with respect to
the change of variables $\xi\leftrightarrow\eta$. According to \eqref{eq:Zjdef} we thus
have $Z_1'(t) = 0$, hence $Z_1(t) = d/2$ for all $t \ge 0$, in agreement with
\eqref{def:Zt}. In the case where $j = 2$, the right-hand side of \eqref{eq:Zjdef}
does not vanish in general, and gives the formula \eqref{def:Z2} for the
vertical speed.
\end{proof}

In the particular case where the vorticity $\Omega$ is given by \eqref{def:G}
and the velocity field $U$ by \eqref{def:UG}, the vertical speed $Z_2'$ defined
by \eqref{def:Z2} only depends on the small parameter $\eps$, except for the
prefactor $\Gamma/d$. It is then instructive to note that the speed $Z_2'$
agrees {\em to all orders} with the prediction of the point vortex system.

\begin{lemma}\label{lem:vertspeed}
If $\Omega = G$ and $U = U^G$, the vertical velocity \eqref{def:Z2} satisfies
\begin{equation}\label{eq:Z20}
  Z_2' \,=\, -\frac{\Gamma}{\eps d}\int_{\RR^2}U^G_2\bigl(\widetilde{\xi}-\eps^{-1}
  e_1\bigr)G(\xi)\,\dd\xi \,=\, \frac{\Gamma}{2\pi d} + O(\eps^\infty)\,.
\end{equation}
\end{lemma}

\begin{proof}
Since the Gaussian function $G$ decays very rapidly at infinity whereas the velocity
field $U^G$ is bounded, it is clear that integrating over the region $|\xi| \ge 1/(2\eps)$
gives a contribution to \eqref{eq:Z20} that is smaller than any power of $\eps$ as $\eps \to 0$.
We can thus write
\begin{align}\nonumber
  \frac{d}{\Gamma}\,Z_2' \,&=\, -\frac{1}{\eps}\int_{\{|\xi|<1/(2\eps)\}}
  U^G_2\bigl(\widetilde{\xi}-\eps^{-1}e_1\bigr)G(\xi)\,\dd\xi + O(\eps^\infty)\\ \label{dZgam}
  \,&=\, \frac{1}{2\pi\eps}\int_{\{|\xi|<1/(2\eps)\}}\frac{\xi_1+1/\eps}{|\xi_1+1/\eps|^2+|\xi_2|^2}
  \,G(\xi)\,\dd\xi + O(\eps^\infty)\,,
\end{align}
where in the second line we used the explicit expression \eqref{def:UG} of the
velocity field $U^G$, and the fact that $|\xi_1+1/\eps|^2+|\xi_2|^2 \ge 1/(4\eps^2)$
on the domain of integration. We next expand the fraction in \eqref{dZgam}
as follows:
\begin{equation}\label{expansion}
  \frac{1}{\eps}\,\frac{\xi_1+1/\eps}{|\xi_1+1/\eps|^2+|\xi_2|^2} \,=\, 
  \frac{1+\eps\xi_1}{|1+\eps\xi_1|^2+|\eps \xi_2|^2} \,=\,
  \sum_{n=0}^\infty(-1)^n \eps^n Q_n^c(\xi)\,,
\end{equation}
where $Q_n^c$ is the homogeneous polynomial on $\RR^2$ defined by
$Q_n^c\bigl(r\cos\theta,r\sin\theta\bigr) = r^n\cos(n\theta)$, see
Section~\ref{ssecA1}. Note that the series converges uniformly for $|\xi| \le
1/(2\eps)$, so that we can exchange the sum with the integral in
\eqref{dZgam}. Since $G$ is a radial function, only the term corresponding to $n
= 0$ gives a nonzero contribution, and we conclude that
\[
  \frac{d}{\Gamma}\, Z_2' \,=\, \frac{1}{2\pi}\int_{\{|\xi|<1/(2\eps)\}}
  G(\xi)\,\dd\xi + O(\eps^\infty) \,=\, \frac{1}{2\pi} + O(\eps^\infty)\,,
\]
which is the desired result. 
\end{proof}


\section{The approximate solution}\label{sec:appsol}

The purpose of this section is to construct an approximate solution of the
rescaled vorticity equation \eqref{eq:Omss} by performing an asymptotic
expansion in the small parameter $\eps = \sqrt{\nu t}/d$. Given an
integer $M \ge 2$, our approximate solution takes the form
\begin{equation}\label{appOP}
  \Omega_\app(\xi,t) \,=\, \Omega_0(\xi) + \sum_{k=2}^M \epsilon(t)^k\,\Omega_k(\xi)\,,
  \qquad \Psi_\app(\xi,t) \,=\, \Psi_0(\xi) + \sum_{k=2}^M \epsilon(t)^k\,\Psi_k(\xi)\,, 
\end{equation}
where the vorticity profiles $\Omega_k$ have to be determined, and the stream function
profiles are given by $\Psi_k = \Delta^{-1}\Omega_k$ as in \eqref{UPsirel}. The
corresponding expansion for the vertical speed is
\begin{equation}\label{appzeta}
  Z_2'(t) \,=\, \frac{\Gamma}{2\pi d}\,\zeta_\app(t)\,, \quad \text{where}\quad
  \zeta_\app(t) \,=\, 1 + \sum_{k=1}^{M-1} \epsilon(t)^k\,\zeta_k\,,
\end{equation}
for some $\zeta_k \in \RR$. It is important to note that $\Omega_\app$,
$\Psi_\app$ and $\zeta_\app$ depend on time only through the aspect ratio
$\eps = \sqrt{\nu t}/d$. Since $\eps \to 0$ as $t \to 0$, the leading order
terms $\Omega_0,\Psi_0$ in \eqref{appOP} are also the initial data of the
approximate solution $\Omega_\app, \Psi_\app$. For consistency we must choose
$\Omega_0$ to be the Lamb-Oseen vortex \eqref{def:G}, and $\Psi_0 =
\Delta^{-1}\Omega_0$ is the associated stream function given by
\eqref{def:UG}. Similarly, Lemma~\ref{lem:vertspeed} implies that
$\zeta_\app(0) = 1$. 

To measure by how much our approximate solution fails to satisfy \eqref{eq:Omss},
we introduce the remainder
\begin{equation}\label{def:resid}
  \cR_M \,:=\, \delta\bigl(\cL \Omega_\app - t\partial_t \Omega_\app\bigr)
  + \Bigl\{\Psi_\app - \cT_\eps\Psi_\app +\frac{\eps\xi_1}{2\pi}\,\zeta_\app\,,\,
  \Omega_\app\Bigr\}\,.
\end{equation}
Using Lemma~\ref{lem:Teps} below and the important fact that
$t\partial_t \eps = \eps/2$, we deduce from \eqref{appOP}, \eqref{appzeta} that
the remainder \eqref{def:resid} can be expanded into a power series in $\eps$. 
As we shall see in Section~\ref{ssec32} below, the remainder of the trivial 
approximation $(\Omega_\app,\Psi_\app) = (\Omega_0,\Psi_0)$ already 
satisfies $\cR_0 = \cO(\eps^2)$, and this is the reason for which 
the expansions \eqref{appOP} start at $k = 2$ instead of $k = 1$.  
Our goal is to choose the profiles $\Omega_k, \Psi_k, \zeta_k$ in such a way
that $\cR_M = \cO(\eps^{M+1})$ in an appropriate topology.  In fact we can
require a little bit less if we observe that the quantity \eqref{def:resid}
involves another small parameter $\delta = \nu/\Gamma$, which (unlike $\eps$)
does not depend on time.  A priori all profiles $\Omega_k, \Psi_k, \zeta_k$
depend on $\delta$ when $k > 0$, but it turns out that contributions of order
$\cO(\delta^2)$ are negligible for our purposes. So we can assume that
\begin{equation}\label{def:ENS}
  \Omega_k \,=\, \Omega_k^E + \delta\Omega_k^{NS}\,, \qquad
  \Psi_k \,=\, \Psi_k^E + \delta\Psi_k^{NS}\,, \qquad
  \zeta_k \,=\, \zeta_k^E + \delta\zeta_k^{NS}\,,
\end{equation}
where the Euler profiles $\Omega_k^E, \Psi_k^E, \zeta_k^E$ and the viscous
corrections $\Omega_k^{NS}, \Psi_k^{NS},\zeta_k^{NS}$ are now independent of
$\delta$, and can be chosen so that $\cR_M = \cO(\eps^{M+1} + \delta^2\eps^2)$.

A final observation is that the profiles \eqref{def:ENS} are not uniquely
determined unless additional conditions are imposed. For instance, as was already
discussed in Section~\ref{ssec22}, the vorticity $\Omega$ has a vanishing linear
moment with respect to the $\xi_2$ variable only if an appropriate choice
is made for the vertical speed $Z_2'$. Using the notation \eqref{def:mean-mom}
the hypotheses we make on the vorticity profiles can be formulated as follows: 

\begin{hypotheses}\label{Hyp:Om} The vorticity profiles $\Omega_k$ in
\eqref{def:ENS} satisfy:  

\medskip
H1) $\,\mrM[\Omega_0] = 1$ and $\mrM[\Omega_k] = 0$ for all $k \ge 1$; 

\medskip
H2) $\,\mrm_1[\Omega_k] = \mrm_2[\Omega_k] = 0$ for all $k \ge 0$; 

\medskip
H3) $\,\Omega_k^E$ is an {\em even} function of $\xi_2$ for all $k \ge 0$. 
\end{hypotheses}

It is important to note that hypotheses H1, H2 apply to both $\Omega_k^E$
and $\Omega_k^{NS}$, whereas the third assumption H3 only concerns the
Euler profiles $\Omega_k^E$. As a matter of fact, experimental observations
and numerical simulations of counter-rotating vortex pairs in viscous fluids
clearly show that the full vorticity distribution is not symmetric with
respect to the line joining the vortex centers, see \cite{DelRos2009}. 

We are now in position to state the main result of this section.

\begin{proposition}\label{prop:appsol}
Given any integer $M \ge 2$ there exists an approximate solution
of the form \eqref{appOP}, \eqref{appzeta}, \eqref{def:ENS} satisfying
Hypotheses~\ref{Hyp:Om} such that the remainder \eqref{def:resid} satisfies
the estimate
\begin{equation}\label{eq:remest}
  \cR_M \,=\, \cO_\cZ\bigl(\epsilon^{M+1}+\delta^2\epsilon^2\bigr)\,,
\end{equation}
where the notation $\cO_\cZ$ is introduced in Definition~\ref{def:topo}
below. 
\end{proposition}

\begin{remark}\label{rem:M}
We do not claim that the approximate solution is uniquely determined
by the properties listed in Proposition~\ref{prop:appsol}, but there
is a canonical choice that makes it unique, see Remark~\ref{rem:unique}
below for a discussion of this question. 
\end{remark}

The choice of the integer $M$, which determines the accuracy of the approximate
solution, depends on the intended purpose. The leading order deformation of the
stream lines and of the level sets of vorticity is already obtained for
$M = 2$, whereas the first correction to the vertical speed $Z_2'$ only
appears when $M = 5$, see Section~\ref{ssec34}. In particular, this means
that $\zeta_k = 0$ for $k = 1,2,3$. As we shall see in Section~\ref{sec:nonlinear},
we need to take $M > 3$ if we want to control the solution of \eqref{eq:NS}
over a time interval $[0,T_\adv]$ that is independent of the viscosity parameter. 
More generally, if $0 \le \sigma < 1$, we need $M > (3+\sigma)/(1-\sigma)$ to reach 
the time $T_\adv \delta^{-\sigma}$. 

\subsection{Function spaces and operators}\label{ssec31}

We first define the function spaces in which our approximate solution will
be constructed. Following \cite{GalWay2005,Gallay2011}, we introduce the
weighted $L^2$ space 
\begin{equation}\label{def:cY}
  \cY \,=\, \biggl\{f\in L^2(\RR^2)\,:\, \int_{\RR^2} |f(\xi)|^2
  \,\e^{|\xi|^2/4}\,\dd \xi < \infty\biggr\}\,,
\end{equation}
which is a Hilbert space equipped with the inner product 
\begin{equation}
  \jap{f,g}_{\cY} \,=\, \int_{\RR^2}f(\xi)\,g(\xi)\,\e^{|\xi|^2/4}\,\dd \xi\,, \qquad
  \forall\,f,g \in \cY\,.
\end{equation}
Using polar coordinates $(r,\theta)$ defined by $\xi  = \bigl(r\cos\theta,r\sin\theta\bigr)$,
we can expand any $f \in \cY$ in a Fourier series with respect to the angular
variable $\theta$. This leads to the direct sum decomposition
\begin{equation}\label{Ydecomp}
  \cY \,= \,\bigoplus_{n=0}^{\infty}\,\cY_n\,,
\end{equation}
where $\cY_n = \bigl\{f\in \cY\,:\, f = a(r)\cos(n\theta) + b(r)\sin(n\theta) \text{ with } a,b:
\RR_+\to\RR\bigr\}$. Note that $\cY_n \perp \cY_{n'}$ if $n \neq n'$, so that the 
decomposition \eqref{Ydecomp} is orthogonal. We also consider the dense subset 
$\cZ \subset \cY$ defined by
\begin{equation}\label{def:calZ}
  \cZ \,=\, \bigl\{f:\RR^2\to \RR \,:\, \xi \mapsto \e^{|\xi|^2/4}f(\xi) \in
  \cS_*(\RR^2)\bigr\}\,,
\end{equation}
where $\cS_*(\RR^2)$ denotes the space of smooth functions with at most
polynomial growth at infinity. More precisely, a function $g : \RR^2 \to \RR$
belongs to $\cS_*(\RR^2)$ if for any multi-index $\alpha=(\alpha_1,\alpha_2) \in \NN^2$
there exists $C>0$ and $N\in \NN$ such that $|\partial^{\alpha}g(\xi)| \le C(1+|\xi|)^N$
for all $\xi\in \RR^2$. As an aside we note that $\cS_*(\RR^2)$ is the multiplier
space of the Schwartz space $\cS(\RR^2)$ and of the space $\cS'(\RR^2)$ of
tempered distributions. Although we do not need to equip $\cS_*(\RR^2)$ with a
precise topology, the following notation will be useful.

\begin{definition}\label{def:topo}
Let $M \in \NN$ be a positive integer. 

\smallskip\noindent 1)
If $g_\eps \in \cS_*(\RR^2)$ depends on a small parameter $\eps > 0$,
we say that $g_\eps \,=\, \cO_{\cS_*}\bigl(\eps^M\bigr)$ if, 
for all $\alpha \in \NN^2$ and all $m \in \{0,\dots,M\}$, there exist 
$C > 0$ such that
\[
  |\partial_\xi^{\alpha} \partial_\eps^{m}g_\eps(\xi)| \,\le\, C(1+|\xi|)^N
  \eps^{M-m}\quad\forall\,\xi \in \RR^2\,.
\]
\noindent 2) Similarly, if $f_\eps \in \cZ$, we write $f_\eps = \cO_{\cZ}
\bigl(\eps^M\bigr)$ if $\e^{|\xi|^2/4}f_\eps = \cO_{\cS_*}\bigl(\eps^M\bigr)$. 
\end{definition}

We next study three linear operators which play an important role in the construction 
of the approximate solution. 

\subsubsection*{{\rm A)} The diffusion operator}

We consider the rescaled diffusion operator $\cL$ defined by \eqref{def:cL} as
a linear operator in $\cY$ with (maximal) domain
\begin{equation}\label{def:cLdom}
  D(\cL) \,=\, \bigl\{f\in \cY \, : \,\Delta f\in \cY, \, 
  \xi\cdot \nabla f \in \cY\bigr\}\,.
\end{equation}
It is well known that $\cL$ is {\em self-adjoint} in the Hilbert space
$\cY$ with compact resolvent and purely discrete spectrum:  
\begin{equation}\label{specL}
  \sigma(\cL) \,=\, \Bigl\{-\frac{n}{2} \, : \, n\in \NN\Bigr\}\,,
\end{equation}
see for instance \cite[Appendix~A]{GalWay2002}. The one-dimensional kernel
of $\cL$ is spanned by the Gaussian function $G$ defined in \eqref{def:G},
whereas the first-order derivatives $\partial_1G, \partial_2 G$ span the
two-dimensional eigenspace corresponding to the eigenvalue $-1/2$. More
generally, the eigenvalue $-n/2$ has multiplicity $n+1$ and the
eigenfunctions are Hermite functions of degree $n$. 

It is easy to verify that the operator $\cL$ is invariant under rotations about
the origin in $\RR^2$, so that it commutes with the direct sum decomposition
\eqref{Ydecomp}: if $f \in \cY_n \cap D(\cL)$, then $\cL f \in \cY_n$.
It is also clear that $\cL \cZ \subset \cZ$, where $\cZ$ is defined by
\eqref{def:calZ}. In the same spirit, the following result will be
established in the Appendix:

\begin{lemma}\label{lem:Linvert}
For any $\kappa > 0$ and any $f \in \cZ$ one has $(\kappa - \cL)^{-1}f
\in \cZ$.   
\end{lemma}

\subsubsection*{{\rm B)} The advection operator}

Another important operator, denoted by $\Lambda:D(\Lambda)\to \cY$, arises when
linearizing the quadratic term in \eqref{eq:NS} at the Lamb-Oseen vortex.
It is defined by
\begin{equation}\label{def:Lambda}
  \Lambda f \,=\, U^G\cdot \nabla f + \BS[f]\cdot \nabla G\,, \qquad f \in D(\Lambda)\,, 
\end{equation}
where the functions $G,U^G$ are given by \eqref{def:G}, \eqref{def:UG} and
the Biot-Savart operator by \eqref{def:BS}. Equivalently, we have
\begin{equation}\label{def:Lambda2}
  \Lambda f \,=\, \bigl\{\Psi_0\,,\,f\bigr\} + \bigl\{\Delta^{-1}f\,,\,\Omega_0\}\,,
  \qquad f \in D(\Lambda)\,, 
\end{equation}
where $\Omega_0,\Psi_0$ are also defined in \eqref{def:G}, \eqref{def:UG}
and $\{\cdot,\cdot\}$ is the Poisson bracket \eqref{def:PB}. 
The operator $\Lambda$ is considered as acting on the maximal domain
\begin{equation}\label{def:Dlambda}
  D(\Lambda) \,=\, \bigl\{f\in \cY \, : \, U^G\cdot \nabla f \in \cY\bigr\}\,.
\end{equation}
It is not difficult to verify that $\Lambda$ is also invariant under rotations
about the origin, so that it commutes with the the direct sum decomposition
\eqref{Ydecomp}. Moreover, it is clear that $\Lambda f \in \cZ$ if $f \in \cZ$,
because the velocity fields $U^G$ and $\BS[f]$ belong to the multiplier space
$\cS_*(\RR^2)^2$. Finally we recall the following properties established
in \cite{GalWay2005,Maekawa2011,Gallay2011,GaSring}:

\begin{proposition}\label{prop:Lambda}
The operator $\Lambda$ is skew-adjoint in the Hilbert space $\cY$ with
kernel 
\begin{equation}\label{kerLam}
  \Ker(\Lambda) \,=\, \cY_0 \,\oplus\, \bigl\{\beta_1\partial_1 G +
  \beta_2\partial_2 G  \,:\, \beta_1, \beta_2 \in \RR\bigr\}\,.
\end{equation}
In addition, if $g \in \Ker(\Lambda)^\perp \cap \cZ$, the equation $\Lambda f = g$
has a unique solution $f \in \Ker(\Lambda)^\perp \cap \cZ$, and $f$ is an
even function of the variable $\xi_2$ if $g$ is an odd function of $\xi_2$. 
\end{proposition}

\begin{remark}\label{rem:Lambda}
Since $\Lambda$ is skew-adjoint in $\cY$ we have $\Ker(\Lambda)^\perp =
\overline{\Ran(\Lambda)}$, where $\Ran(\Lambda)$ is the range of $\Lambda$.
So a necessary condition for the equation $\Lambda f = g$ to have a
solution is that $g \perp \Ker(\Lambda)$, which according to \eqref{kerLam} is
equivalent to
\begin{equation}\label{eq:solvcond}
  \cP_0\,g \,=\, 0\,, \quad \text{and}\quad \mrm_1[g] \,=\, \mrm_2[g]
  \,=\, 0\,,
\end{equation}
where $\cP_0$ is the orthogonal projection in $\cY$ onto the subspace $\cY_0$ of
radially symmetric functions, and $\mrm_1, \mrm_2$ are the first-order moments
defined in \eqref{def:mean-mom}. Note that the solvability conditions
\eqref{eq:solvcond} are not sufficient in general to ensure that
$g \in \Ran(\Lambda)$, but under the additional assumption that $g \in \cZ$
Proposition~\ref{prop:Lambda} shows that $g = \Lambda f$ for some $f \in \cZ$.
\end{remark}

\subsubsection*{{\rm C)} The translation/reflection operator}

Finally we study the action of the operator $\cT_\eps$ defined by
\eqref{def:Teps} on functions $\Psi \in \cS_*(\RR^2)$ such that
$\Delta \Psi \in \cZ$. 

\begin{lemma}\label{lem:Teps}
Assume that $\Omega \in \cZ$ and let $\Psi = \Delta^{-1}\Omega$ as in
\eqref{UPsirel}. For each integer $n \ge 1$, let $P_n$ be the polynomial of
degree $n$ given by 
\begin{equation}\label{def:Pn}
  P_n(\xi) \,=\, \frac{(-1)^{n-1}}{n} \frac{1}{2\pi }\int_{\RR^2}Q_n^c
  \bigl(\xi_1 + \eta_1,\xi_2 - \eta_2\bigr)\,\Omega(\eta)\,\dd\eta\,,
\end{equation}
where $Q_n^c$ is the homogeneous polynomial on $\RR^2$ defined by $Q_n^c(r\cos\theta,
r\sin\theta) = r^n\cos(n\theta)$, see Section~\ref{ssecA1}. Then for all
$N \in \NN$ one has the expansion
\begin{equation}\label{Texpansion}
  \bigl(\cT_\eps\Psi\bigr)(\xi) \,=\, C\,\log\frac{1}{\eps}
  + \sum_{n=1}^N \epsilon^n P_n(\xi) + \cO_{\cS_*}\bigl(\epsilon^{N+1}\bigr)\,,
\end{equation}
where $C = \mrM[\Omega]/(2\pi)$ with $\mrM[\Omega]$ given by \eqref{def:mean-mom}. 
\end{lemma}

\begin{proof}
Using the representation \eqref{UPsirel} and the definition \eqref{def:Teps} of the
operator $\cT_\eps$, we find
\begin{align*}
  \bigl(\cT_\eps\Psi\bigr)(\xi) \,&=\, \Psi(-\xi_1-1/\eps,\xi_2) \,=\,
  \frac{1}{2\pi}\int_{\RR^2}\log\sqrt{|\xi_1{+}\eta_1{+}1/\eps|^2 +
  |\xi_2{-}\eta_2|^2}\,\Omega(\eta)\,\dd\eta \\
  \,&=\, \frac{1}{2\pi}\,\log\frac{1}{\eps}\int_{\RR^2}\Omega(\eta)\,\dd\eta
  + \frac{1}{4\pi}\int_{\RR^2}\log\Bigl(|1{+}\eps(\xi_1{+}\eta_1)|^2 + \eps^2
  |\xi_2{-}\eta_2|^2\Bigr)\,\Omega(\eta)\,\dd\eta\,.                              
\end{align*}
The first term in the right-hand side is $C \log(1/\eps)$, so we need only
consider the last integral. We can also assume that $\xi \in B_{R_\eps}$, where
$B_{R_{\eps}}$ is the ball of radius $R_\eps:=1/(4\eps)$ centered at the origin. Indeed, 
any function in $S_*(\RR^2)$ restricted to the complementary region $B_{R_\eps}^c$ is 
already of order $\eps^N$ for all $N$, and we clearly have $\cT_\eps\Psi\in S_*(\RR^2)$. 
Similarly we can restrict the domain of integration so that $\eta \in B_{R_\eps}$, 
up to negligible errors. Now, if $\xi,\eta \in B_{R_\eps}$, we denote 
$x = \eps(\xi_1+\eta_1,\xi_2-\eta_2) \in \RR^2$ and we use the expansion
\begin{equation}\label{eq:Qnlog}
  \frac12 \log\bigl(1 + 2x_1 + |x|^2\bigr) \,=\, \sum_{n=1}^\infty \frac{(-1)^{n-1}}{n}
  \,Q_n^c(x)\,, \qquad |x| < 1\,,
\end{equation}
which is justified in Section~\ref{ssecA1}. Truncating the Taylor series at
order $N$ and estimating the integral remainder together with its derivatives
with respect to $\xi$ and $\eps$, we obtain the formula
\[
  \bigl(\cT_\eps\Psi\bigr)(\xi) \,=\, C\,\log\frac{1}{\eps} + \sum_{n=1}^N
  \frac{(-1)^{n-1}}{n} \frac{\eps^n}{2\pi}\int_{B_{R_\eps}}Q_n^c\bigl(\xi_1+\eta_1,\xi_2-\eta_2)
  \,\Omega(\eta)\,\dd\eta + \cO_{\cS_*}\bigl(\epsilon^{N+1}\bigr)\,.
\]
Since $\Omega\in \cZ$, we can replace the integral over $B_{R_\eps}$ with the
integral over $\RR^2$ up to an error of size $\mathcal{O}_{\cS_*}(\eps^{N+1})$.
Thus \eqref{Texpansion} is proved in view of \eqref{def:Pn}.
\end{proof}

\begin{remark}\label{rem:OO}
In the sequel, when controlling the size of error terms, we shall use the
notation $\cO_{\cS_*}(\eps^M)$ without explicitly checking that all conditions
required by Definition~\ref{def:topo} are satisfied. Indeed, derivatives with
respect to $\xi$ are straightforward to estimate, and the dependence of our
quantities with respect to $\eps$ is either polynomial or occurs through the
operator $\cT_\eps$.  In the latter case, one can argue as in the proof of Lemma
\ref{lem:Teps}.
\end{remark}

\begin{remark}\label{rem:firstPn}
For $n = 1$ and $2$, it follows from \eqref{def:Pn} that
\[
  P_1(\xi) \,=\, \frac{1}{2\pi}\int_{\RR^2}(\xi_1+\eta_1)\,\Omega(\eta)\,\dd\eta\,, \quad
  P_2(\xi) \,=\, -\frac{1}{4\pi}\int_{\RR^2}\Bigl[(\xi_1{+}\eta_1)^2 - (\xi_2{-}\eta_2)^2\Bigr]
  \,\Omega(\eta)\,\dd\eta\,. 
\]
In particular, if $\mrM[\Omega] = 0$, we have $\nabla P_1 = 0$ so that $\nabla\cT_\eps
\Psi = \cO_{\cS_*}(\epsilon^2)$. If in addition $\mrm_1[\Omega] = \mrm_2[\Omega] = 0$, then
$\nabla P_2 = 0$ so that $\nabla\cT_\eps\Psi = \cO_{\cS_*}(\epsilon^3)$. 
\end{remark}

\subsection{The second order approximation}\label{ssec32}

Before starting the construction of our approximate solution, we compute
the error generated by the naive approximation $\Omega_\app = \Omega_0$,
$\zeta_\app = 1$, which corresponds to setting $M = 0$ in \eqref{appOP},
\eqref{appzeta}. In that case, since $\cL\Omega_0 = t\partial_t \Omega_0 = 0$
and $\{\Psi_0,\Omega_0\} = 0$, the remainder \eqref{def:resid} reduces to 
\begin{equation}\label{def:resid0}
  \cR_0 \,:=\, \Bigl\{-\cT_\eps \Psi_0 +\frac{\eps\xi_1}{2\pi}\,,\,
  \Omega_0\Bigr\} \,=\, \Bigl(\mathbf{T}_\eps U^G + \frac{\eps}{2\pi}\,e_2
  \Bigr)\cdot \nabla G\,,
\end{equation}
where the operators $\cT_\eps$ and $\mathbf{T}_\eps$ are defined in \eqref{def:Teps}
and \eqref{def:Uelleps}, respectively. The following statement is a particular
case of the results established in \cite[Section~3.1]{Gallay2011}. We
give a short proof for the reader's convenience. 

\begin{proposition}\label{prop:naive}
For any integer $N \ge 2$ the remainder \eqref{def:resid0} satisfies
\begin{equation}\label{Rem0exp}
  \cR_0(\xi) \,=\, \frac{1}{4\pi} \sum_{n=2}^N (-1)^{n-1} \eps^n Q_n^s(\xi)\,G(\xi)
   + \cO_\cZ\bigl(\eps^{N+1}\bigr)\,,
\end{equation}
where $Q_n^s$ is the homogeneous polynomial on $\RR^2$ defined by $Q_n^s(r\cos\theta,
r\sin\theta) = r^n\sin(n\theta)$, see Section~\ref{ssecA1}. 
\end{proposition}

\begin{proof}
We introduce the notation $\eta=\eps^{-1}e_1$. Using the definition of $U^G$
in \eqref{def:UG}, we find
\begin{equation}\label{id:U0ell}
  \bigl(\mathbf{T}_\eps\,U^G\bigr)(\xi) \,=\, \widetilde{U^G}\bigl(
  \widetilde{\xi}-\eta\bigr) \,=\, -U^G(\xi+\eta)\,, \quad \text{and}\quad
  \frac{\eps}{2\pi}e_2 \,=\, \frac{1}{2\pi}\frac{\eta^\perp}{|\eta|^2}\,,
\end{equation}
hence
\begin{equation}\label{def:naiveR}
  \cR_0(\xi) \,=\, \frac{1}{2\pi}\biggl(\frac{\eta^\perp}{|\eta|^2} -
  \frac{(\xi+\eta)^\perp}{|\xi+\eta|^2} + \frac{(\xi+\eta)^\perp}{|\xi+\eta|^2}
  \,\e^{-|\xi+\eta|^2/4}\biggr) \cdot \nabla G(\xi)\,.
\end{equation}
To prove \eqref{Rem0exp}, it is sufficient to estimate \eqref{def:naiveR}
for $|\xi| \le 1/(2\epsilon)$, because in the complementary region
the right-hand side of \eqref{def:naiveR} is of order $\epsilon^\infty$ in $\cZ$.
If $|\xi| \le 1/(2\epsilon)$, then $|\xi+\eta|^2 \ge 1/(4\eps^2)$ so that
the exponential factor in \eqref{def:naiveR} is $\cO(\eps^\infty)$. Therefore,
since $\nabla G(\xi) = -(\xi/2) G(\xi)$, we obtain
\begin{equation}\label{Rem0exp2}
  \cR_0(\xi) \,=\, \frac{1}{4\pi}\,\xi\cdot\biggl(\frac{(\xi+\eta)^\perp}{|\xi+\eta|^2}
  - \frac{\eta^\perp}{|\eta|^2}\biggr) G(\xi) \,\chi\bigl(\eps|\xi|\bigr) + 
  \cO_\cZ\bigl(\eps^\infty\bigr)\,,
\end{equation}
where $\chi$ is a smooth function on $\RR_+$ such that $\chi(r) = 1$ for
$r \le 1/4$ and $\chi(r) = 0$ for $r \ge 1/2$. To conclude the proof we observe
that, for $|\xi| \le 1/(2\epsilon)$, 
\[
  \xi\cdot\biggl(\frac{(\xi+\eta)^\perp}{|\xi+\eta|^2}-\frac{\eta^\perp}{|\eta|^2}\biggr)
  \,=\, \eps\xi_2 \biggl(\frac{1}{|1+\eps\xi_1|^2+|\eps \xi_2|^2} - 1\biggr)
  \,=\, \sum_{n=2}^\infty (-1)^{n-1}\epsilon^n Q_n^s(\xi)\,,
\]
where the last equality follows from \eqref{eq:Qnid}. Thus \eqref{Rem0exp2}
implies \eqref{Rem0exp}. 
\end{proof}

\begin{remark}\label{rem:naive}
It is important to notice that, according to \eqref{Rem0exp}, the remainder
$\cR_0$ is already of order $\eps^2$. As is clear from the above proof, 
the cancellation of the first order in $\eps$ is due to the choice $\zeta_\app = 1$,
which in turn is equivalent to $Z_2' = \Gamma/(2\pi d)$. In other words,
in our approach the translation speed given by the Helmholtz-Kirchhoff system
can be recovered by minimizing the error of the naive approximation.
This is a general feature of interacting vortices in the plane, see 
\cite{Gallay2011}. In the case of axisymmetric vortex rings, which
share many similarities with counter-rotating vortex pairs, the
remainder of the naive approximation is $\cO(\eps)$ and not better,
even if the translation speed is chosen appropriately, see \cite{GaSring}. 
\end{remark}

Since $\cR_0 = \cO_{\cZ}(\eps^2)$, the first nontrivial step in our construction is 
the second order approximation, corresponding to $M = 2$, for which 
$\Omega_\app = \Omega_0 + \epsilon^2 \Omega_2$ and $\Psi_\app = \Psi_0 
+ \epsilon^2\Psi_2$. Although this is not obvious, we anticipate 
that $\zeta_1 = 0$, which means that there is no correction to the translation 
speed at this level of approximation. To determine the unknown profile 
$\Omega_2$, the strategy is again to minimize the remainder \eqref{def:resid}, 
which takes the form
\begin{equation}\label{def:Res2}
\begin{split}
  \cR_2 \,&=\, \delta\epsilon^2 (\cL-1)\Omega_2 + \Bigl\{\Psi_0 -
  \cT_\eps\Psi_0 + \frac{\eps\xi_1}{2\pi} + \epsilon^2 \Psi_2 -
  \epsilon^2 \cT_\eps\Psi_2\,,\, \Omega_0 + \epsilon^2\Omega_2\Bigr\} \\
  \,&=\, \delta\epsilon^2 (\cL-1)\Omega_2 + \cR_0 + \epsilon^2
  \Lambda \Omega_2 + \epsilon^2 \cN_2\,,
\end{split}  
\end{equation}
where $\Lambda$ is the linear operator \eqref{def:Lambda2} and
\begin{equation}\label{def:R2}
  \cN_2 \,=\, \Bigl\{-\cT_\eps\Psi_0 + \frac{\eps\xi_1}{2\pi}
  + \epsilon^2 \Psi_2 - \epsilon^2 \cT_\eps\Psi_2\,,\,\Omega_2\Bigr\}
  - \Bigl\{\cT_\eps \Psi_2 \,,\, \Omega_0\Bigl\}\,.
\end{equation}
Invoking Proposition~\ref{prop:naive} with $N = 2$, we obtain
\begin{equation}\label{def:cH2}
  \cR_0 \,=\, \epsilon^2 \cH_2 + \cO_\cZ\bigl(\eps^3\bigr)\,, \qquad \text{where}
  \quad \cH_2(\xi) \,=\, -\frac{1}{2\pi}\,\xi_1\xi_2 G(\xi)\,,
\end{equation}
and we can thus write the remainder \eqref{def:Res21} in the form
\begin{equation}\label{def:Res21}
  \cR_2 \,=\, \epsilon^2\Bigl[\delta(\cL-1)\Omega_2 + \Lambda \Omega_2
  + \cH_2\Bigr] + \epsilon^2 \cN_2 + \cO_\cZ\bigl(\eps^3\bigr)\,.
\end{equation}

The properties recalled in Section~\ref{ssec31} imply that the linear operator
$\delta(\cL-1) + \Lambda$ is maximally dissipative in the Hilbert space $\cY$,
hence invertible for any $\delta > 0$. Since $\cH_2 \in \cZ \subset \cY$, there
exists a unique profile $\Omega_2 \in \cY$ such that the quantity inside
brackets in \eqref{def:Res21} vanishes exactly. However, we need to verify that
$\Omega_2 \in \cZ$ and that $\Omega_2$ does not blow up in the limit
$\delta \to 0$, which is not immediately obvious. For these reasons, we find it
simpler to solve the problem approximately, using only the information given by
Proposition~\ref{prop:Lambda}.  We set
$\Omega_2 = \Omega_2^E + \delta\Omega_2^{NS}$, where

\medskip i) $\Omega_2^E \in \cY_2 \cap \cZ$ is the unique solution
of $\Lambda \Omega_2^E + \cH_2 = 0$;

\medskip ii) $\Omega_2^{NS} \in \cY_2 \cap \cZ$ is the unique
solution of $\Lambda \Omega_2^{NS} + (\cL-1)\Omega_2^E = 0$.

\medskip
We recall that $\cY_2$ is the subspace of $\cY$ corresponding to the angular
Fourier mode $n = 2$, see \eqref{Ydecomp}. The explicit expression \eqref{def:cH2}
shows that $\cH_2 \in \cY_2 \cap \cZ$, hence $\cH_2 \in \Ker(\Lambda)^\perp$
in view of \eqref{kerLam}. Applying Proposition~\ref{prop:Lambda} we conclude
that the equation $\Lambda \Omega_2^E + \cH_2 = 0$ has indeed a unique solution
$\Omega_2^E \in \cY_2\cap \cZ$, which is an even function of $\xi_2$ because
$\cH_2$ is obviously odd with respect to $\xi_2$. Similarly, since 
$(\cL-1)\Omega_2^E \in  \cY_2 \cap \cZ$, it follows from Proposition~\ref{prop:Lambda}
that the equation $\Lambda \Omega_2^{NS} + (\cL-1)\Omega_2^E = 0$ has a unique
solution $\Omega_2^{NS} \in \cY_2 \cap \cZ$. We conclude that the full profile
$\Omega_2$ belongs to $\cY_2 \cap \cZ$, which implies in particular that the moment
conditions in Hypotheses~\ref{Hyp:Om} are automatically satisfied. 

\begin{remark}\label{rem:explicit2}
It is possible to obtain a more explicit expression of the profile $\Omega_2$,
in terms of solutions of linear ODEs on $\RR_+$, see Section~\ref{ssec34}.  In
particular $\Omega_2^E(\xi) = (\xi_2^2 - \xi_1^2)w_2(|\xi|)$ for some smooth and
{\em nonnegative} function $w_2$ with Gaussian decay at infinity.  The profile
$\Omega_2^E$, which represents the leading order correction to the radially
symmetric vortex $\Omega_0$ in the expansion \eqref{appOP}, is responsible for
the deformation of the stream lines and of the level lines of the vorticity,
which are nearly elliptical at this order of approximation, see
Fig.~\ref{fig1}. As an aside, since $\cH_2 = \{h_2, G\}$ for a suitable function
$h_2$, it is observed in \cite{DD25} that $\Omega_2^E$ can be represented by a
Neumann series involving the operator $A^{-1/2}(-\Delta)^{-1}A^{-1/2}$, where
$A$ is defined in \eqref{def:A} below. This is, in fact, a consequence of the
analysis carried out in \cite{GaSarnold}.
\end{remark}

With the above choice of $\Omega_2$ the remainder \eqref{def:Res2} takes the form
\begin{equation}\label{def:Res22}
  \cR_2 \,=\, \epsilon^2 \delta^2(\cL-1)\Omega_2^{NS} + \epsilon^2 \cN_2
  + \cO_\cZ\bigl(\eps^3\bigr) \,=\, \cO_\cZ\bigl(\eps^3 + \delta^2\epsilon^2\bigr)\,,
\end{equation}
because it is easy to verify using \eqref{def:R2}, Lemma~\ref{lem:Teps} and
Remark~\ref{rem:firstPn} that $\cN_2 = \cO_\cZ(\eps^2)$. This concludes the proof
of Proposition~\ref{prop:appsol} in the particular case where $M = 2$. 

\begin{remark}\label{rem:unique}
A minor drawback of solving the linear equation $\delta(\cL-1)\Omega_2 + \Lambda
\Omega_2 + \cH_2 = 0$ perturbatively in $\delta$ is that the solution is not
unique. Indeed, since the subspace $\cY_0$ of radially symmetric functions
is contained in $\Ker(\Lambda)$ by \eqref{kerLam}, we can add to $\Omega_2^{NS}$ any
element of $\cY_0 \cap \cZ$ without affecting the remainder estimate \eqref{def:Res22},
and Hypotheses~\ref{Hyp:Om} are still satisfied as well. Uniqueness is restored
if one assumes that $\cP_0 \Omega_2^{NS} = 0$, where $\cP_0$ is the orthogonal projection
in $\cY$ onto $\cY_0$. Note that we must always impose $\cP_0 \Omega_2^E = 0$, otherwise 
equation ii) above has no solution. 
\end{remark}

\subsection{The induction step}\label{ssec33}

We now use an induction argument to complete the proof of Proposition~\ref{prop:appsol}.
Assume that the conclusion holds for some integer $M \ge 2$ (the case $M = 2$
being settled in Section~\ref{ssec32}.) We consider a refined approximate
solution of the form
\begin{equation}\label{appsol2}
  \tilde\Omega_\app = \Omega_\app + \epsilon^{M+1} \Omega_{M+1}\,, \quad
  \tilde\Psi_\app = \Psi_\app + \epsilon^{M+1} \Psi_{M+1}\,, \quad
  \tilde\zeta_\app = \zeta_\app + \epsilon^M \zeta_M\,,
\end{equation}
where $\Omega_\app, \Psi_\app, \zeta_\app$ are as in \eqref{appOP}, \eqref{appzeta},
and where $\Omega_{M+1} \in \cZ$, $\Psi_{M+1} \in \cS_*(\RR^2)$, $\zeta_M \in \RR$ have
to be determined so that $\cR_{M+1} = \cO_\cZ(\epsilon^{M+2}+\delta^2\epsilon^2)$. 

We first study the remainder $\cR_M$ given by \eqref{def:resid}, which is a quadratic
polynomial in the parameter $\delta$ in view of \eqref{def:ENS}. Using in particular
Lemma~\ref{lem:Teps}, we can expand the right-hand side of \eqref{def:resid}
in powers of $\eps$ and, by induction hypothesis, the expansion starts at order
$\cO_\cZ(\eps^{M+1})$ for the terms that are proportional to $\delta^0$ or $\delta^1$.
In other words, there exist $\cH_0,\cH_1 \in \cZ$ such that
\begin{equation}\label{Remdec}
  \cR_M \,=\, \epsilon^{M+1} \cH_0 + \delta \epsilon^{M+1} \cH_1 +
  \cO_\cZ\bigl(\epsilon^{M+2}+\delta^2\epsilon^2\bigr)\,.
\end{equation}
The idea is of course to choose $\Omega_{M+1},\Psi_{M+1},\zeta_M$ so as to
cancel the terms $\cH_0, \cH_1$ in \eqref{Remdec}. To do that, we need some
information on the first order moments.  Using \eqref{def:resid}, one can check
by a direct calculation that $\mrM[\cR_M] = \mrm_1[\cR_M] = 0$, so that
$\mrM[\cH_j] = \mrm_1[\cH_j] = 0$ for $j = 0,1$. However, we have
$\mrm_2[\cR_M] \neq 0$ in general. In addition, it follows from
Hypotheses~\ref{Hyp:Om} that $\cR_M$ is an odd function of $\xi_2$ when
$\delta = 0$, which implies that $\cH_0$ is an odd function of $\xi_2$.

We next consider the remainder of the refined approximation \eqref{appsol2},
which reads 
\begin{align}\nonumber
  \cR_{M+1} \,&:=\, \delta\bigl(\cL\tilde\Omega_\app - t\partial_t\tilde
  \Omega_\app\bigr) + \Bigl\{\tilde\Psi_\app - \cT_\eps\tilde\Psi_\app
  +\frac{\eps\xi_1}{2\pi}\,\tilde\zeta_\app\,,\, \tilde\Omega_\app\Bigr\} \\ \label{RMa}
  \,&=\, \cR_M + \delta\epsilon^{M+1}\bigl(\cL- {\TS\frac{M+1}{2}}\bigr)\,\Omega_{M+1}
  + \epsilon^{M+1} \Bigl\{\Psi_\app - \cT_\eps\Psi_\app +\frac{\eps\xi_1}{2\pi}\,\zeta_\app\,,\,
  \Omega_{M+1}\Bigr\}\\ \nonumber
  & \quad\, + \,\epsilon^{M+1}\Bigl\{\Psi_{M+1} - \cT_\eps\Psi_{M+1} +\frac{\xi_1}{2\pi}\,\zeta_M
    \,,\, \Omega_\app + \epsilon^{M+1}\Omega_{M+1}\Bigr\}\,.
\end{align}
Using the expansion \eqref{Remdec} and the identity
\[
  \bigl\{\Psi_0\,,\,\Omega_{M+1}\bigr\} + \Bigl\{\Psi_{M+1} +\frac{\xi_1}{2\pi}\,\zeta_M
  \,,\, \Omega_0\Bigr\} \,=\, \Lambda\Omega_{M+1} + \frac{\zeta_M}{2\pi}\,\partial_2
  \Omega_0\,,
\]
where $\Lambda$ is the differential operator \eqref{def:Lambda}, we can write the
quantity $\cR_{M+1}$ in the form
\begin{equation}\label{RMb}
  \cR_{M+1} \,=\, \eps^{M+1} \cA_{M+1} + \eps^{M+1} \cN_{M+1} + \cO_\cZ\bigl(\epsilon^{M+2}
  +\delta^2\epsilon^2\bigr)\,,
\end{equation}
where $\cA_{M+1}$ is the collection of the principal terms:  
\begin{equation}\label{Pidef}
   \cA_{M+1} \,=\,  \delta\bigl(\cL - {\TS\frac{M+1}{2}}
  \bigr)\Omega_{M+1} + \Lambda\Omega_{M+1} + \cH_0 + \delta \cH_1 + \frac{\zeta_M}{2\pi}
  \,\partial_2 \Omega_0\,,
\end{equation}
whereas $\cN_{M+1}$ gathers higher order corrections: 
\begin{align*}
  \cN_{M+1} \,&=\, \Bigl\{\Psi_\app - \Psi_0 - \cT_\eps\Psi_\app +\frac{\eps\xi_1}{2\pi}
  \,\zeta_\app\,,\,\Omega_{M+1}\Bigr\} \\
  & + \Bigl\{\Psi_{M+1} +\frac{\xi_1}{2\pi}\,\zeta_M
  \,,\, \Omega_\app - \Omega_0 + \epsilon^{M+1}\Omega_{M+1}\Bigr\} 
  -\Bigl\{\cT_\eps\Psi_{M+1}\,,\, \Omega_\app + \epsilon^{M+1}\Omega_{M+1}\Bigr\}\,.
\end{align*}

We now determine $\Omega_{M+1},\Psi_{M+1},\zeta_M$ so as to minimize the quantity
$\cA_{M+1}$. We first define the correction to the vertical speed: 
\begin{equation}\label{zMdef}
  \zeta_M \,=\, \zeta_M^E + \delta\zeta_M^{NS}\,, \quad\text{where}\quad  
  \frac{\zeta_M^E}{2\pi} \,=\, \int_{\RR^2}\xi_2 \cH_0(\xi)\,\dd\xi\,, \quad
  \frac{\zeta_M^{NS}}{2\pi} \,=\, \int_{\RR^2}\xi_2 \cH_1(\xi)\,\dd\xi\,.
\end{equation}
We thus have $\cA_{M+1} =  \delta\bigl(\cL - {\TS\frac{M+1}{2}}\bigr)\Omega_{M+1} +
\Lambda\Omega_{M+1} + \tilde\cH_0 + \delta \tilde\cH_1$, where
\begin{equation}\label{def:tildeH}
  \tilde \cH_0 \,:=\, \cH_0 + \frac{\zeta_M^E}{2\pi}\,\partial_2 \Omega_0\,, \qquad
  \tilde \cH_1 \,:=\, \cH_1 + \frac{\zeta_M^{NS}}{2\pi}\,\partial_2 \Omega_0\,,
\end{equation}
and the choice \eqref{zMdef} ensures that $\mrm_2[\tilde\cH_0] = \mrm_2[\tilde\cH_1] = 0$. 
We next define
\begin{equation}\label{OMdef}
  \Omega_{M+1} \,=\, \Omega_{M+1}^{E,0} + \Omega_{M+1}^{E,1} + \delta
  \Omega_{M+1}^{NS}\,, 
\end{equation}
where the vorticity profiles $\Omega_{M+1}^{E,0}, \Omega_{M+1}^{E,1}, \Omega_{M+1}^{NS}$
are determined in the following way: 

\begin{enumerate}[leftmargin=15pt,itemsep=2pt]

\item The radially symmetric function $\Omega_{M+1}^{E,0} \in \cY_0 \cap \cZ$ 
is the unique solution, given by Lemma~\ref{lem:Linvert}, of the elliptic equation
\begin{equation}\label{def:E,0}
  \bigl(\cL - {\TS\frac{M+1}{2}}\bigr)\Omega_{M+1}^{E,0} +
  \cP_0 \tilde \cH_1 \,=\, 0\,,
\end{equation}
where $\cP_0$ is the orthogonal projection in $\cY$ onto the radial subspace $\cY_0$.  

\item The function $\Omega_{M+1}^{E,1} \in \Ker(\Lambda)^\perp \cap \cZ$ is the unique
solution, given by Proposition~\ref{prop:Lambda}, of
\begin{equation}\label{def:E,1}
  \Lambda \Omega_{M+1}^{E,1} + \tilde \cH_0 \,=\, 0\,.
\end{equation}
Remark that $\tilde \cH_0 \in \Ker(\Lambda)^\perp$ because $\tilde \cH_0(\xi)$ is an
odd function of $\xi_2$, which implies that $\cP_0\tilde\cH_0 = 0$ and $\mrm_1[\tilde\cH_0]
= 0$, and because $\mrm_2[\tilde\cH_0] = 0$ by our choice of $\zeta_M^E$.
Note also that $\Omega_{M+1}^{E,1}$ is an even function of $\xi_2$, as asserted
in Proposition~\ref{prop:Lambda}. 

\item The function $\Omega_{M+1}^{NS} \in \Ker(\Lambda)^\perp \cap \cZ$ is the unique
solution, given by Proposition~\ref{prop:Lambda}, of
\begin{equation}\label{def:E,NS}
  \Lambda\Omega_{M+1}^{NS} + (1-P_0)\tilde \cH_1 + \bigl(\cL - {\TS\frac{M+1}{2}}
  \bigr)\Omega_{M+1}^{E,1} \,=\,0\,,
\end{equation}
where the last two terms belong to $\Ker(\Lambda)^\perp \cap \cZ$ by construction.
\end{enumerate}

In view of \eqref{zMdef}--\eqref{def:E,NS} we have $\cA_{M+1} =  \delta^2\bigl(\cL
- {\TS\frac{M+1}{2}}\bigr)\Omega_{M+1}^{NS}$, and the profile $\Omega_{M+1}$ satisfies
Hypotheses~\ref{Hyp:Om}. Returning to \eqref{RMb} we thus find
\begin{align*}
  \cR_{M+1} \,&=\, \epsilon^{M+1}\delta^2\bigl(\cL- {\TS\frac{M+1}{2}}\bigr)
  \Omega_{M+1}^{NS}+ \epsilon^{M+1} \cN_{M+1} + \cO_\cZ\bigl(\epsilon^{M+2}
  +\delta^2\epsilon^2\bigr)\\
  \,&=\, \cO_\cZ\bigl(\epsilon^{M+2}+\delta^2\epsilon^2\bigr)\,,
\end{align*}
because using Lemma~\ref{lem:Teps} it is easy to verify that $\cN_{M+1} =
\cO_\cZ(\eps^2)$. This concludes the induction step, and the proof of
Proposition~\ref{prop:appsol} is now complete. \qed

\subsection{Leading order correction to the vertical speed}\label{ssec34}

The goal of this section is to compute the leading order correction to the
vertical speed $Z_2'$ in the approximate solution \eqref{appOP},
\eqref{appzeta}. It turns out that this correction occurs for $M = 5$, which
means that $\zeta_k = 0$ for $k = 1,2,3$, see \cite{HaNaFu2018}. As is explained
in Section~\ref{ssec33}, the coefficient $\zeta_4$ is chosen so as to ensure the
solvability of the ``elliptic'' equation for the vorticity profile $\Omega_5$,
as in \eqref{zMdef}. Fortunately, it turns out that the expression of $\zeta_4$
only involves the leading order correction $\Omega_2$ to the vorticity
distribution. No information on $\Omega_3$ and $\Omega_4$ is needed at this
stage. 

\begin{lemma}\label{lem:Om23}
Using polar coordinates $\xi = (r\cos\theta,r\sin\theta)$, the leading order
correction $\Omega_2$ in the approximate solution \eqref{appOP} takes the form
\begin{equation}\label{Om2exp}
  \Omega_2(\xi) \,=\, -\sfw_2(r)\cos(2\theta)  + \delta\,\hat\sfw_2(r)
  \sin(2\theta)\,,
\end{equation}
for some $\sfw_2, \hat\sfw_2 : \RR_+ \to \RR$ with $\sfw_2 > 0$. 
\end{lemma}

\begin{proof}
We already know that $\Omega_2 = \Omega_2^E + \delta \Omega_2^{NS}$ where
$\Lambda\Omega_2^E + \cH_2 = 0$ and $\Lambda \Omega_2^{NS} + (\cL-1)\Omega_2^E = 0$,
see Section~\ref{ssec32}. According to \eqref{def:cH2} we have
$-\cH_2 = b(r)\sin(2\theta)$ where $b(r) = r^2g(r)/(2\pi)$ with $g$ as in
\eqref{def:v0gh}.  In particular $\cH_2 \in \cY_2 \cap \cZ$, so that we
can apply Lemma~\ref{lem:Lambda} in Section~\ref{ssecA3}. We thus obtain the
formulas $\Omega_2^E = w(r)\cos(2\theta)$, $\Psi_2^E = \varphi(r)\cos(2\theta)$, 
where
\begin{equation}\label{eq:wphi2}
  w(r) \,=\, -\varphi(r)h(r) - \frac{b(r)}{2v_0(r)} \,=\, -h(r)\Bigl(
  \varphi(r) + \frac{r^2}{4\pi}\Bigr)\,,
\end{equation}
and $\varphi$ is the unique solution of the ODE \eqref{eq:varphidef2} with $n=2$ such that
$\varphi(r) = \cO(r^2)$ as $r \to 0$ and $\varphi(r) = \cO(r^{-2})$ as $r \to +\infty$. 
It follows from the the maximum principle that $\varphi$ is a positive function,
so that $w(r) < 0$ for all $r > 0$ in view of \eqref{eq:wphi2}. We thus have
$\Omega_2^E = -\sfw_2(r)\cos(2\theta)$ with $\sfw_2(r) = -w(r) > 0$. We deduce that
$(\cL-1)\Omega_2^E = a(r)\cos(2\theta)$ for some $a : \RR_+ \to \RR$, and
applying Lemma~\ref{lem:Lambda} again we conclude that $\Omega_2^{NS} =
\hat\sfw_2(r)\sin(2\theta)$ for some $\hat\sfw_2 : \RR_+ \to \RR$.
\end{proof}

\begin{remark}\label{rem:nextorder}
Similarly the next correction $\Omega_3 = \Omega_3^E + \delta \Omega_3^{NS}$ is determined
by the relations $\Lambda\Omega_3^E + \cH_3 = 0$ and $\Lambda \Omega_3^{NS} +
(\cL-\frac32)\Omega_3^E = 0$, where $\cH_3$ is the third order term in the expansion
\eqref{Rem0exp}, namely
\[
  \cH_3(\xi) \,=\, \frac{1}{4\pi}\,Q_3^s(\xi)G(\xi) \,=\, 
  \frac{1}{4\pi}\,\bigl(3\xi_1^2\xi_2 - \xi_2^3 \bigr)G(\xi)\,.
\]
Thus $\cH_3 \in \cY_3 \cap \cZ$ and $\cH_3 = b(r)\sin(3\theta)$ for some
$b : \RR_+ \to \RR$. Proceeding exactly as before, we thus find that
$\Omega_3^E = \sfw_3(r)\cos(3\theta)$ and $\Omega_3^{NS} = \hat\sfw_3(r)\sin(3\theta)$
for some $\sfw_3, \hat\sfw_3 : \RR_+ \to \RR$.
\end{remark}

We are now able to formulate the main result of this section.

\begin{proposition}\label{prop:speed}
If $M \ge 5$ the advection speed \eqref{appzeta} satisfies
\begin{equation}\label{Z2exp}
  Z_2'(t) \,=\, \frac{\Gamma}{2\pi d}\Bigl(1 - 2\pi \alpha\,\epsilon^4  + \cO\bigl(
  \epsilon^5 + \delta^2 \epsilon\bigr)\Bigr)\,,
\end{equation}
where
\begin{equation}\label{def:alpha}
  \alpha \,=\, \frac{1}{\pi}\int_{\RR^2} \bigl(\xi_2^2-\xi_1^2\bigr)\Omega_2(\xi)\,\dd\xi
  \,=\, \int_0^\infty r^3 \sfw_2(r)\,\dd r \,\approx\, 22.24\,.
\end{equation}
\end{proposition}

\begin{remark}\label{rem:speed}
The error term $\cO(\epsilon^5 + \delta^2 \epsilon)$ in \eqref{Z2exp} is
probably not optimal, and may be improved using the techniques developed
in \cite{DD25}. We just recall here that $\epsilon^2 = \delta t/T_\adv$,
where $T_\adv$ is defined in \eqref{def:all}. Thus, if $t \gtrsim T_\adv$, then
$\delta \lesssim \epsilon^2$ and the error term in \eqref{Z2exp} can be
replaced by $\cO(\epsilon^5)$. 
\end{remark}

\begin{proof}
We consider the approximate solution \eqref{appOP}, \eqref{appzeta} for some
$M \ge 5$. According to Proposition~\ref{prop:appsol}, the remainder
\eqref{def:resid} satisfies $\cR_M = \cO_\cZ\bigl(\epsilon^{M+1}+\delta^2\epsilon^2\bigr)$.
To obtain a formula for the vertical speed, we multiply both members of \eqref{def:resid}
by $\xi_2$ and we integrate over $\xi \in \RR^2$. Proceeding as in the proof of
Lemma~\ref{lem:vertspeed}, and recalling that $-U_{2,\app}(\tilde{\xi}-\eps^{-1}e_1,t)
=\partial_1\cT_\eps\Psi_{\app}$, we find
\[
  \int_{\RR^2} \Bigl(\partial_1 \cT_\epsilon \Psi_\app - \frac{\epsilon}{2\pi}\,
  \zeta_\app\Bigr)\Omega_\app\,\dd\xi \,=\, \cO\bigl(\epsilon^{M+1} + \delta^2\epsilon^2
  \bigr)\,.
\]
Since $\mrM[\Omega_\app] = 1$ by Hypotheses~\ref{Hyp:Om}, we deduce the
representation formula
\begin{equation}\label{zetaapp}
  \zeta_\app \,=\, \frac{2\pi}{\epsilon} \int_{\RR^2} \bigl(\partial_1 \cT_\epsilon
  \Psi_\app\bigr)\Omega_\app\,\dd\xi +  \cO\bigl(\epsilon^M + \delta^2\epsilon)\,,
\end{equation}
which is the analogue of \eqref{def:Z2}. 

In the rest of the proof, we assume that $M = 5$. To compute the integral in
\eqref{zetaapp}, we apply Lemma~\ref{lem:Teps} with $\Omega = \Omega_\app$ and
$\Psi = \Psi_\app$. This gives the expansion
\begin{equation}\label{cTepsPsi}
  \partial_1\bigl(\cT_\eps\Psi_\app\bigr)(\xi) \,=\, \sum_{n=1}^5 \epsilon^n
  \partial_1 P_n(\xi) + \cO_{\cS_*}\bigl(\epsilon^6\bigr)\,,
\end{equation}
where the polynomials $P_n$ are given by \eqref{def:Pn} with $\Omega =
\Omega_\app$. Using Remark~\ref{rem:firstPn} as well as the information
on the moments of $\Omega_\app$ contained in Hypotheses~\ref{Hyp:Om}, 
we find
\[
  \partial_1 P_1(\xi) \,=\, \frac{1}{2\pi}\,\mrM[\Omega_\app] \,=\,
  \frac{1}{2\pi}\,, \qquad
  \partial_1 P_2(\xi) \,=\, -\frac{1}{2\pi}\,\Bigl(\xi_1 \mrM[\Omega_\app]
  + \mrm_1[\Omega_\app]\Bigr) \,=\, -\frac{\xi_1}{2\pi}\,. 
\]
Similarly, a direct calculation shows that
\begin{align*}
  \partial_1 P_3(\xi) \,&=\, \frac{1}{2\pi} \int_{\RR^2}\Bigl(
 (\xi_1+\eta_1)^2 - (\xi_2-\eta_2)^2\Bigr)\Omega_\app(\eta)\,\dd\eta \\
 \,&=\, \frac{1}{2\pi}\,(\xi_1^2 - \xi_2^2) + \frac{1}{2\pi} \int_{\RR^2}
 (\eta_1^2 - \eta_2^2)\Omega_\app(\eta)\,\dd\eta 
  \,=\, \frac{1}{2\pi}\,Q_2^c(\xi) - \frac{\alpha \epsilon^2}{2} +
  \cO_{\cS_*}\bigl(\epsilon^3\bigr)\,,
\end{align*}
where in the last equality we used the fact that $\Omega_\app = \Omega_0 +
\epsilon^2\Omega_2 + \cO_\cZ(\epsilon^3)$ with $\Omega_2$ as in \eqref{Om2exp}, 
together with the definition of $\alpha$ in \eqref{def:alpha}. Finally, since
$\Omega_\app = \Omega_0 + \cO_\cZ(\epsilon^2)$, we also have
\[
  \partial_1 P_4(\xi) \,=\, -\frac{1}{2\pi}\,Q_3^c(\xi) + \cO_{\cS_*}\bigl(\epsilon^2\bigr)\,,
  \qquad
  \partial_1 P_5(\xi) \,=\, \frac{1}{2\pi}\,Q_4^c(\xi) + \cO_{\cS_*}\bigl(\epsilon^2\bigr)\,.
\]
Note that the homogeneous polynomials $Q_n^c$ already appear in the expansion \eqref{expansion}. 
Summarizing, we have shown that
\begin{equation}\label{cTepsPsi2}
  \frac{2\pi}{\epsilon}\,\partial_1\bigl(\cT_\eps\Psi_\app\bigr)(\xi) \,=\, 
  \sum_{n=0}^4 (-1)^n\epsilon^n Q_n^c(\xi) -\pi \alpha \epsilon^4 + 
  \cO_{\cS_*}\bigl(\epsilon^5\bigr)\,.
\end{equation}

To conclude the proof, we multiply \eqref{cTepsPsi2} by $\Omega_\app(\xi)$ and we integrate
over $\xi \in \RR^2$. The contribution of the leading order term $\Omega_0$
in $\Omega_\app$ is $1 - \pi\alpha\epsilon^4 + \cO(\epsilon^5)$, and using again
\eqref{Om2exp} and Hypotheses~\ref{Hyp:Om} we obtain
\[
  \frac{2\pi}{\epsilon} \int_{\RR^2} \bigl(\partial_1 \cT_\epsilon
  \Psi_\app\bigr)\bigl(\Omega_\app - \Omega_0\bigr)\,\dd\xi \,=\,
  \epsilon^4 \int_{\RR^2}Q_2^c(\xi)\Omega_2(\xi)\,\dd\xi + \cO\bigl(\epsilon^5\bigr)
  \,=\, -\pi\alpha\epsilon^4 + \cO\bigl(\epsilon^5\bigr)\,.
\]
Altogether we thus find $\zeta_\app = 1 - 2\pi\alpha \epsilon^4 + \cO\bigl(
\epsilon^5 + \delta^2 \epsilon\bigr)$. 
\end{proof}

\subsection{Functional relationship in the inviscid case}\label{ssec35}

In this section, we investigate the properties of our approximate solution
\eqref{appOP}, \eqref{appzeta} in the limiting situation where $\delta = 0$,
which corresponds to the inviscid case. In view of \eqref{def:resid} and
Proposition~\ref{prop:appsol}, we have
\begin{equation}\label{def:residE}
  \cR_M^E \,:=\, \Bigl\{\Psi_\app^E - \cT_\eps\Psi_\app^E  + \frac{\eps\xi_1}{2\pi}
  \,\zeta_\app^E\,,\,\Omega_\app^E\Bigr\} \,=\, \cO_{\cZ}\bigl(\epsilon^{M+1}\bigr)\,,
\end{equation}
where, as in \eqref{def:ENS}, the letter $E$ in superscript refers to the Euler
equation. We consider the stream function in the uniformly translating frame
attached to the vortex center, defined as
\begin{equation}\label{def:PhiappE}
  \Phi_\app^E \,:=\, \Psi_\app^E - \cT_\eps\Psi_\app^E  + \frac{\eps\xi_1}{2\pi}
  \,\zeta_\app^E + \frac{1}{2\pi}\log\frac{1}{\epsilon}\,,
\end{equation}
where the last term on the right-hand side is a constant included in order to cancel
an irrelevant $\log(1/\varepsilon)$-order in the expansion \eqref{eq:PhiappE} below.
The remainder $\cR_M^E = \bigl\{\Phi_\app^E\,, \, \Omega_\app^E\bigr\}$ would vanish
identically if we had a functional relationship of the form
$\Phi_\app^E + F(\Omega_\app^E) = 0$ for some (smooth) function
$F : \RR \to \RR$. In reality, since $\cR_M^E = \cO(\eps^{M+1})$, the best we
can hope for is an approximate functional relationship, which holds up to
corrections of order $\cO(\epsilon^{M+1})$.

At leading order ($M = 0$), our approximate solution is $\Omega_\app = \Omega_0$,
$\Phi_\app = \Psi_\app = \Psi_0$, where $\Omega_0$ and $\Psi_0$ are defined in
\eqref{def:G}, \eqref{def:UG}. It follows that $\Phi_0 + F_0(\Omega_0) = 0$ if
we define
\begin{equation}\label{def:F0}
  F_0(s) \,=\, \frac{1}{4\pi}\Bigl(\gamma_E - \Ein\bigl(\log\frac{1}{4\pi s}\bigr)
  \Bigr)\,, \qquad 0 < s \le \frac{1}{4\pi}\,.
\end{equation}
Note that $F_0$ is smooth, strictly increasing, and satisfies $F_0(s) \sim 
-(4\pi)^{-1}\log\log\frac{1}{s}$ as $s \to 0$. For later use we define
\begin{equation}\label{def:A}
  A(\xi) \,:=\, F_0'\bigl(\Omega_0(\xi)\bigr) \,=\, -\frac{\nabla\Psi_0(\xi)}{
  \nabla\Omega_0(\xi)} \,=\, \frac{4}{|\xi|^2}\Bigl(\e^{|\xi|^2/4} - 1\Bigr)\,.
\end{equation}

We next investigate the functional relationship for the second order 
approximation ($M = 2$). As is explained in Section~\ref{ssec32} we 
have $\Omega_\app^E = \Omega_0 + \epsilon^2 \Omega_2^E$ and $\Psi_\app^E = \Psi_0 
+ \epsilon^2 \Psi_2^E$, where the corrections $\Omega_2^E, \Psi_2^E$ are computed in 
Lemma~\ref{lem:Om23}. Using Lemma~\ref{lem:Teps} with $N = 2$, we obtain the following
expansion of the stream function in the moving frame: 
\begin{equation}\label{eq:PhiappE}
  \Phi_\app^E \,=\, \Psi_0 + \epsilon^2 \Psi_2^E + \frac{\epsilon^2}{4\pi}
  \,\bigl(\xi_1^2 - \xi_2^2\bigr) + \cO_{\cS_*}\bigl(\epsilon^3\bigr)\,.
\end{equation}

\begin{figure}[t]
  \begin{center}
  \begin{picture}(450,150)
  \put(7,0){\includegraphics[width=0.45\textwidth]{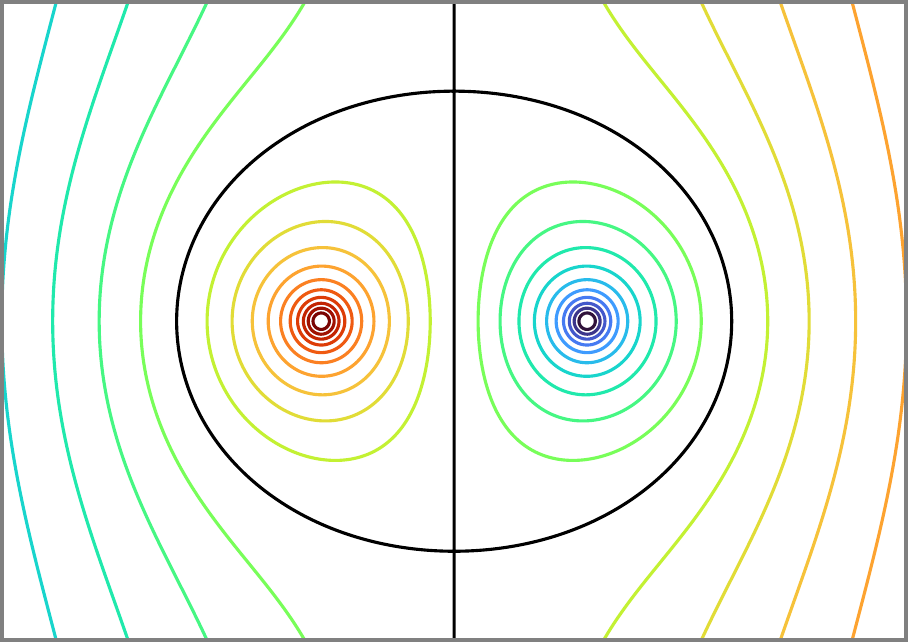}}
  \put(235,0){\includegraphics[width=0.45\textwidth]{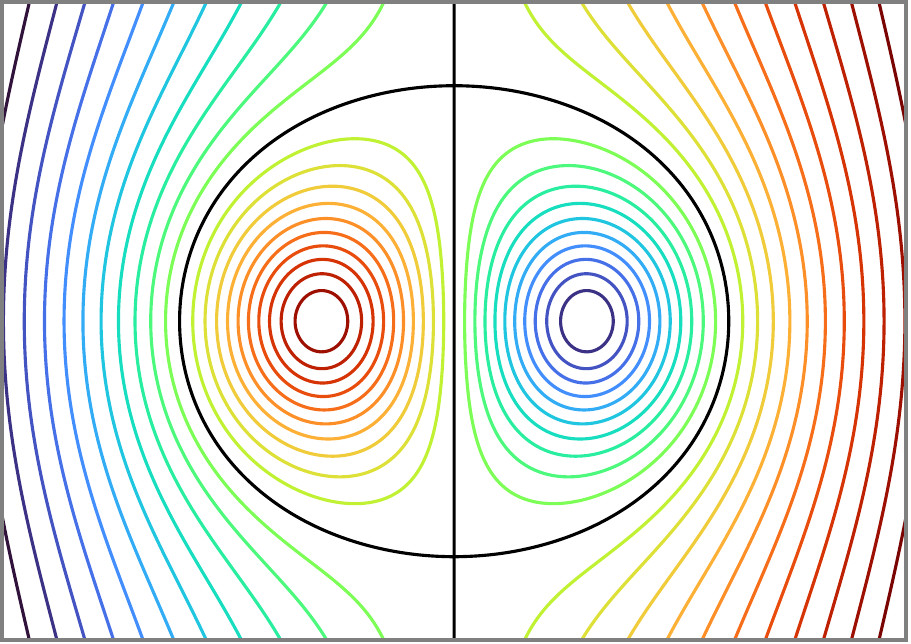}}
  \end{picture}
  \captionsetup{width=.9\linewidth}
  \caption{{\small
The level lines of the function $\Phi_\app^E$ defined in \eqref{def:PhiappE},
which correspond to the stream lines of the inviscid approximate solution
$\Omega_\app^E$ in the co-moving frame, are represented for $M = 2$ and
$\eps = 1/50$ (left) or $\eps=1/8$ (right). Large positive values
of $\Phi_\app^E$ are depicted in red, and large negative values in blue.
The flow has two elliptic stagnation points located at $\xi = 0$ and
$\xi = (-1/\eps,0)$ in our coordinates, as well as two hyperbolic points
on the black line which separates the vortex dipole from the exterior flow.
Near the vortex centers, the stream lines are nearly elliptical with a major 
axis in the $\xi_2$-direction, which reflects the fact that $\sfw_2 > 0$
in \eqref{Om2exp}.}}
  \label{fig1}
  \end{center}
\end{figure}\noindent

The level lines of the second order approximation \eqref{eq:PhiappE} are shown in Fig.~\ref{fig1} 
(ignoring the $\eps^3$ corrections). We now look for a relationship of the form 
$\Phi_\app^E + F(\Omega_\app^E) = \cO_{\cS_*}(\eps^3)$, where $F = F_0 + \epsilon^2 F_2$ for 
some $F_2 : (0,+\infty) \to \RR$. This is problematic, however, because it is not true that 
$\Omega_\app^E(\xi) > 0$ for all $\xi \in \RR^2$. As is explained in Remark~\ref{rem:germ} 
below, this difficulty is avoided if we only require that the second order Taylor
polynomial in $\epsilon$ of the quantity $\Phi_\app^E + F(\Omega_\app^E)$ vanishes. 
This gives the relation
\begin{equation}\label{eq:F2rel}
  \Psi_2^E + \frac{1}{4\pi}\,\bigl(\xi_1^2 - \xi_2^2\bigr) + F_0'(\Omega_0)\Omega_2^E
  + F_2(\Omega_0) \,=\, 0\,,
\end{equation}
which serves as a definition of $F_2$. It turns out that the first three terms
in \eqref{eq:F2rel} sum up to zero, so that we can actually take $F_2 \equiv 0$. Indeed,
this is a consequence of equation \eqref{eq:wphi2} in Lemma~\ref{lem:Om23}, because
$\Psi_2^E = \varphi(r)\cos(2\theta)$, $\xi_1^2 - \xi_2^2 = r^2\cos(2\theta)$,
$\Omega_2^E = w(r)\cos(2\theta)$, and $F_0'(\Omega_0) = A = 1/h$ by \eqref{def:A}.

\smallskip
To establish the functional relationship at any order, we use an
induction argument on the integer $M$, as in the proof of
Proposition~\ref{prop:appsol}. The following definition will be useful:

\begin{definition}\label{def:Kclass}
We say that a smooth function $F : (0,+\infty) \to \RR$ belongs to the class
$\cK$ if $F(\Omega_0) \in \cS_*(\RR^2)$, where $\Omega_0$ is defined in
\eqref{def:G}. 
\end{definition}

Note that, if $\cH \in \cS_*(\RR^2)$, the equation $F(\Omega_0) = \cH$ has a
solution $F \in \cK$ if and only if $\{\cH,\Omega_0\} = 0$, namely if $\cH$
is radially symmetric. In that case the function $F$ is uniquely determined
on the range of $\Omega_0$, which is the interval $(0,(4\pi)^{-1}]$. Since
$\cH$ may grow polynomially at infinity, the function $F$ may have a
logarithmic singularity at $s = 0$. The following observation is also useful. 

\begin{lemma}\label{lem:Fk}
Let $F:(0,+\infty)\to \RR$ be in the class $\cK$. Then, for all $k \in \NN$, the $k$-th 
order derivative of $F$ has the property that $F^{(k)}(\Omega_0)\Omega_0^k \in \cS_*(\RR^2)$.
\end{lemma}

\begin{proof}
We verify by induction over $k$ that the function $\cF_k := F^{(k)}(\Omega_0)\Omega_0^k$
belongs to $\cS_*(\RR^2)$. This is the case for $k = 0$ because $F \in \cK$. 
Assume now that $\cF_k \in \cS_*(\RR^2)$ for some $k \in \NN$. Using the identity
$\nabla\Omega_0(\xi) = -(\xi/2)\Omega_0(\xi)$, we obtain by a direct calculation
\begin{equation}\label{eq:k+1}
  2 \nabla\cF_k(\xi) + \xi\bigl(\cF_{k+1}(\xi) + k \cF_k(\xi)\bigr) \,=\, 0\,, 
  \qquad \xi \in \RR^2\,.
\end{equation}
Since both $\nabla\cF_k$ and $\xi \cF_k$ belong to $\cS_*(\RR^2)^2$, so does 
$\xi \cF_{k+1}$ by \eqref{eq:k+1}. This means that $\cF_{k+1}$ and its derivatives 
have at most a polynomial growth at infinity, so that $\cF_{k+1} \in \cS_*(\RR^2)$.  
\end{proof}

\begin{proposition}\label{prop:germ}
Given any integer $M \ge 2$, let $\Omega_\app^E,\Psi_\app^E$ be the approximate
solution \eqref{appOP} given by Proposition~\ref{prop:appsol} with $\delta = 0$.
There exists $F \in \cK$ of the form $F = F_0 + \epsilon^2 F_2 + \dots + \epsilon^M F_M$ 
such that the stream function $\Phi_\app^E$ defined by \eqref{def:PhiappE} satisfies
\begin{equation}\label{eq:Phirel}
  \Pi_M \Bigl(\Phi_\app^E + F\bigl(\Omega_\app^E\bigr)\Bigr) \,=\,0\,, 
  \quad \hbox{where}\quad \Pi_M f \,=\, \sum_{k=0}^M \frac{\epsilon^k}{k!}
  \,\frac{\dd^k f}{\dd \epsilon^k}\biggl|_{\epsilon = 0}\,.
\end{equation}
\end{proposition}

\begin{remark}\label{rem:germ}
Unfortunately, we cannot write the conclusion of Proposition~\ref{prop:germ}
in the seemingly more natural form
\[
  \Phi_\app^E + F\bigl(\Omega_\app^E\bigr) \,=\, \cO_{\cS_*}\bigl(\epsilon^{M+1}\bigr)\,,
\]
because the quantity $F\bigl(\Omega_\app^E\bigr)$ is not properly defined.
Indeed, our assumptions on the approximate solution do not ensure that
$\Omega_\app^E(\xi) > 0$ for all $\xi \in \RR^2$, whereas a function $F \in \cK$
is only defined on $(0,+\infty)$. However the Taylor polynomial
$\Pi_M F\bigl(\Omega_\app^E\bigr)$ is well defined, because it only involves
derivatives of $F$ evaluated at $\Omega_0$, and its coefficients belong to the
space $\cS_*(\RR^2)$ by Lemma~\ref{lem:Fk}. 
\end{remark}
  
\begin{proof} To simplify the notation, we drop the superscript ``$E$'' and 
the subscript ``$\app$'' everywhere. We proceed by induction over the integer $M$.
Assume that, for some $M \ge 2$, we have constructed $\Phi \in \cS_*(\RR^2)$,
$\Omega \in \cZ$, and $F \in \cK$ of the form
\[
  \Phi \,=\, \Phi_0 + \sum_{k=2}^M \epsilon^k \Phi_k\,, \qquad
  \Omega \,=\, \Omega_0 + \sum_{k=2}^M \epsilon^k \Omega_k\,, \qquad
  F \,=\, F_0 + \sum_{k=2}^M \epsilon^k F_k\,,
\]
such that $\bigl\{\Phi,\Omega\bigr\} = \cO_{\cZ}\bigl(\epsilon^{M+1}\bigr)$ and
$\Pi_M\bigl(\Phi + F(\Omega)\bigr) = 0$. (We have just checked that this holds
for $M = 2$.) Suppose now that we have a refined expansion of the form
\begin{equation}\label{eq:refexp}
  \tilde \Phi \,=\, \Phi + \epsilon^{M+1}\Phi_{M+1}\,, \quad
  \tilde \Omega \,=\, \Omega + \epsilon^{M+1}\Omega_{M+1}\,, \quad
  \text{where} \quad \bigl\{\tilde\Phi,\tilde\Omega\bigr\} = \cO_{\cZ}
  \bigl(\epsilon^{M+2}\bigr)\,.
\end{equation}
We want to find $F_{M+1} \in \cK$ such that $\Pi_{M+1}\bigl(\tilde\Phi
+ \tilde F(\tilde\Omega)\bigr) = 0$ with $\tilde F = F + \epsilon^{M+1}F_{M+1}$. 

\smallskip
First, using the induction hypothesis, we observe that
\begin{equation}\label{eq:F+Phi}
  \Pi_{M+1}\bigl(\Phi + F(\Omega)\bigr) \,=\, \bigl(\Pi_{M+1} - \Pi_M\bigr)
  \bigl(\Phi + F(\Omega)\bigr) \,=\, \epsilon^{M+1} \cH_{M+1}\,,
\end{equation}
for some $\cH_{M+1} \in \cS_*(\RR^2)$. We deduce that $\bigl\{\Phi\,,\,\Omega\bigr\}
= \epsilon^{M+1}\bigl\{\cH_{M+1}\,,\,\Omega_0\bigr\} + \cO_{\cZ}\bigl(\epsilon^{M+2}\bigr)$,
because
\[
  \Pi_{M+1}\bigl\{\Phi\,,\,\Omega\bigr\} \,=\, \Pi_{M+1}\bigl\{\Phi + F(\Omega)\,,
  \,\Omega\bigr\} \,=\, \Pi_{M+1}\bigl\{\Pi_{M+1}(\Phi + F(\Omega))\,,
  \,\Omega\bigr\} \,=\, \epsilon^{M+1}\bigl\{\cH_{M+1}\,,\,\Omega_0\bigr\}\,. 
\]
On the other hand, it follows from the definition of $\tilde\Phi, \tilde\Omega$
in \eqref{eq:refexp} that
\[
  \bigl\{\tilde\Phi\,,\,\tilde\Omega\bigr\} \,=\, \bigl\{\Phi\,,\,\Omega\bigr\}
  + \epsilon^{M+1}\bigl\{\Phi_{M+1}\,,\,\Omega_0\bigr\} + \epsilon^{M+1} \bigl\{\Phi_0
  \,,\,\Omega_{M+1}\bigr\} +  \cO_{\cZ}\bigl(\epsilon^{M+2}\bigr)\,,
\]
so the assumption that $\bigl\{\tilde\Phi,\tilde\Omega\bigr\} = \cO_{\cZ}\bigl(\epsilon^{M+2}\bigr)$
implies
\begin{equation}\label{eq:refcond}
  \bigl\{\cH_{M+1}\,,\,\Omega_0\bigr\} + \bigl\{\Phi_{M+1}\,,\,\Omega_0\bigr\} +
  \bigl\{\Phi_0\,,\,\Omega_{M+1}\bigr\} \,=\, 0\,.
\end{equation}
Finally using \eqref{eq:refexp}, \eqref{eq:F+Phi} we find
\begin{align}\nonumber
  \Pi_{M+1}\bigl(\tilde \Phi + \tilde F(\tilde\Omega)\bigr) \,&=\,
  \Pi_{M+1}\bigl(\Phi + F(\Omega)\bigr) +   \Pi_{M+1}\bigl(\tilde \Phi - \Phi\bigr) +
  \Pi_{M+1}\bigl(\tilde F(\tilde\Omega) - F(\Omega)\bigr) \\  \label{eq:tPhi+F}       
  \,&=\, \epsilon^{M+1}\Bigl(\cH_{M+1} + \Phi_{M+1} + F_0'(\Omega_0)\Omega_{M+1}
  + F_{M+1}(\Omega_0)\Bigr)\,,
\end{align}
because at the points where $\tilde\Omega$ and $\Omega$ are positive we have the
identity
\[
  \tilde F\bigl(\tilde\Omega) - F(\Omega) \,=\, F\bigl(\Omega + \epsilon^{M+1}\Omega_{M+1}\bigr)
 - F(\Omega) \,+\, \epsilon^{M+1}F_{M+1}\bigl(\Omega + \epsilon^{M+1}\Omega_{M+1}\bigr)\,.
\]

\smallskip
In view of \eqref{eq:tPhi+F}, we must choose $F_{M+1} \in \cK$ so that
$\cA_{M+1} + F_{M+1}(\Omega_0) = 0$, where
\[
  \cA_{M+1} \,:=\, \cH_{M+1} + \Phi_{M+1} + F'(\Omega_0)\Omega_{M+1} \,\in\, \cS_*(\RR^2)\,.
\]
As was mentioned after Definition~\ref{def:Kclass}, this is possible if and only
if $\{\cA_{M+1},\Omega_0\} = 0$, but this solvability condition is implied by 
\eqref{eq:refcond} because
\[
  \bigl\{F_0'(\Omega_0)\Omega_{M+1}\,,\,\Omega_0\bigr\} \,=\, \bigl\{\Omega_{M+1}\,,\,F_0(\Omega_0)\bigr\}
  \,=\, \bigl\{\Omega_{M+1}\,,\,-\Phi_0\bigr\} \,=\, \bigl\{\Phi_0\,,\,\Omega_{M+1}\bigr\}\,.
\]
So there exists $F_{M+1} \in \cK$ such that $\Pi_{M+1}\bigl(\tilde\Phi
+ \tilde F(\tilde\Omega)\bigr) = 0$ with $\tilde F = F + \epsilon^{M+1}F_{M+1}$.
This concludes the proof. 
\end{proof}

\begin{corollary}\label{cor:frel}
With the same notation as in Proposition~\ref{prop:germ}, there exist $\sigma_1 > 0$,
$C > 0$, and $N \in \NN$ (depending on $M$) such that, for $\epsilon > 0$ small enough, 
\begin{equation}\label{eq:frelbd}
  \bigl|\nabla \bigl(\Phi_\app^E + F(\Omega_\app^E)\bigr)(\xi)\bigr| \,\le\,
  C\,\epsilon^{M+1}\bigl(1 + |\xi|\bigr)^N\,, \quad \text{for}~\,|\xi| \,\le\, 2\epsilon^{-\sigma_1}\,.
\end{equation}
\end{corollary}

\begin{proof}
We recall that $\Omega_\app^E = \Omega_0 + \epsilon^2\Omega_2 + \dots + \epsilon^M \Omega_M$.
For any $k \in \{2,\dots,M\}$, we have $\Omega_k \in \cZ$, so that $\epsilon^k |\Omega_k(\xi)|
\le C_k\,\epsilon^k\,\Omega_0(\xi) (1+|\xi|)^{N_k}$ for some constants $C_k > 0$ and $N_k \in \NN$.
This means that $\epsilon^k \Omega_k/\Omega_0$ is very small in the region $\{\xi \in \RR^2\,:\,
|\xi| \le 2\epsilon^{-\sigma_1}\}$ if $\sigma_1 < k/N_k$ and $\epsilon > 0$ is small enough.
As a consequence, if $\sigma_1 > 0$ is small enough, we have the estimate
\begin{equation}\label{eq:Ompos}
  \frac12\,\Omega_0(\xi) \,\le\, \Omega_\app^E(\xi) \,\le\, 2\Omega_0(\xi)\,,
  \qquad |\xi| \,\le\, 2\epsilon^{-\sigma_1}\,,
\end{equation}
which implies in particular that $\Omega_\app^E(\xi)$ is positive when $|\xi| \le 2\epsilon^{-\sigma_1}$. 

We now consider the function $\Theta(\epsilon,\xi) \,:=\, \Phi_\app^E(\xi) + 
F(\Omega_\app^E(\xi))$, which is well defined for $\epsilon > 0$ sufficiently small if 
$|\xi| \le 2\epsilon^{-\sigma_1}$. Using a Taylor expansion of order $M$ at $\epsilon = 0$, 
and taking into account the fact that $\Pi_M \Theta = 0$ by Proposition~\ref{prop:germ}, we 
obtain the representation formula
\begin{equation}\label{eq:Thetarep}
  \Theta(\epsilon,\xi) \,=\, \frac{\epsilon^{M+1}}{M!} \int_0^1 (1-\tau)^M
  \partial_\epsilon^{M+1}\Theta(\tau \epsilon,\xi)\,\dd\tau\,, \qquad
  |\xi| \,\le\, 2\epsilon^{-\sigma_1}\,.
\end{equation}
The integrand in \eqref{eq:Thetarep} is estimated by straightforward calculations,
using the bound \eqref{eq:Ompos} and the fact that $\Phi_\app^E \in \cS_*(\RR^2)$,
$\Omega_\app^E \in \cZ$, and $F \in \cK$. We find that $|\partial_\epsilon^{M+1}
\Theta(\tau \epsilon,\xi)| \le C(1+|\xi|)^N$ for some integer $N \in \NN$, and a similar
estimate holds for the derivatives with respect to $\xi$. This gives the bound \eqref{eq:frelbd}
after integrating over $\tau \in [0,1]$. 
\end{proof}


\section{Correction to the approximate solution}\label{sec:nonlinear} 

In this section, our goal is to show that the exact solution $\Omega(\xi,t)$ of \eqref{eq:Omss}
with initial data \eqref{def:G} remains close to the approximate solution $\Omega_\app(\xi,t)$
constructed in Section~\ref{sec:appsol}. The accuracy of our approximation depends on
the integer $M$ in \eqref{appOP}, which is chosen large enough, and on the inverse
Reynolds number $\delta = \nu/\Gamma$, which is taken sufficiently small. We make
the decomposition
\begin{equation}\label{def:per}
  \Omega(\xi,t) \,=\, \Omega_\app(\xi,t) + \delta \ww(\xi,t)\,, \quad
  \Psi(\xi,t) \,=\, \Psi_\app(\xi,t) + \delta\varphi(\xi,t)\,, 
\end{equation}
where $\ww$ is the vorticity perturbation and $\varphi = \Delta^{-1}\ww$ is the
associated correction of the stream function. We also decompose the vertical
speed $Z_2'$ as
\begin{equation}\label{def:perZ2}
    Z_2'(t) \,=\, \frac{\Gamma}{2\pi d}\bigl(\zeta_\app(t)+\zeta(t)\bigr)\,, \qquad
\end{equation}
where $\zeta_\app$ is the approximation \eqref{appzeta} and, in agreement with
\eqref{def:Z2}, the correction $\zeta$ is given by the formula
\begin{equation}\label{def:zeta}
  \zeta(t) \,=\, \frac{2\pi}{\eps}\int_{\RR^2}\Bigl((\partial_1\cT_\eps \Psi)\Omega\Bigr)
  (\xi,t)\,\dd \xi \,-\, \zeta_{\app}(t)\,.
\end{equation}
Inserting \eqref{def:per}, \eqref{def:perZ2} into \eqref{eq:Omss} and using the definition
\eqref{def:resid} of the remainder $\mathcal{R}_M$, we find that the vorticity
perturbation $\ww$ satisfies
\begin{align}\label{linear0}
  t\partial_t \ww-\cL\ww \,=\, &\frac{1}{\delta}\Bigl\{\Psi_\app-\cT_{\eps}\Psi_\app
  +\frac{\eps \xi_1}{2\pi}\zeta_{\app}\,,\ww\Bigr\} + \frac{1}{\delta}
  \Bigl\{\varphi-\cT_\eps \varphi\,,\Omega_\app\Bigr\}\\ \label{force0}
  & +\Bigl\{\varphi-\cT_\eps \varphi\,,\ww\Bigr\} +\frac{1}{\delta^2}\,\cR_M
  +\frac{\eps\,\zeta}{2\pi\delta^2}\Bigl\{\xi_1,\Omega_{\app}+\delta \ww\Bigr\}\,.
\end{align}

This evolution equation has to be solved with zero initial data at time $t = 0$,
because both $\Omega_{\app}(\cdot,t)$ and $\Omega(\cdot,t)$ converge as
$t \to 0$ to the same limit $\Omega_0$ given by \eqref{def:G}, see the
discussion after Remark~\ref{rem:RDE}. Moreover, using the moment conditions
\eqref{eq:momid} in Lemma~\ref{lem:momZ2}, which are also fulfilled by the
approximate solution $\Omega_\app$ in view of Hypotheses~\ref{Hyp:Om}, we see
that the solution of \eqref{linear0}--\eqref{force0} satisfies
\begin{equation}\label{eq:zeromoments}
  \mathrm{M}[\ww(\cdot,t)] \,=\, \mathrm{m}_1[\ww(\cdot,t)] \,=\,
  \mathrm{m}_2[\ww(\cdot,t)] \,=\, 0\,, \qquad \text{for all } t > 0\,.
\end{equation}
Note that the last condition $\mathrm{m}_2[\ww] = 0$ is ensured by our choice
\eqref{def:zeta} of the correction $\zeta$ to the approximate vertical speed
$\zeta_\app$. 

The structure of the evolution equation \eqref{linear0}--\eqref{force0} is quite
transparent. In the first line we find the linearization of \eqref{eq:Omss} at
the approximate solution $\Omega_{\app}$, in a frame moving with the
approximate vertical velocity \eqref{appzeta}. The nonlinear interaction between
the vorticity perturbation $\ww$ and the associated stream function $\varphi$ is
described by the first term in \eqref{force0}, whereas the last term takes into
account the small correction $\zeta$ to the vertical speed. Since $w(\cdot,0) = 0$,
the solution of \eqref{linear0}--\eqref{force0} is actually driven by the source term
$\delta^{-2}\cR_M$ in \eqref{force0}, which measures by how much $\Omega_\app$ fails 
to be an exact solution of \eqref{eq:Omss}. According to Proposition~\ref{prop:appsol},
this term is small if $M$ is large enough and $\delta$ small enough. 

We are now in a position to state the main result of this section, which is a refined
version of Theorem \ref{thm:main}. 

\begin{theorem}\label{th:mainNL} Fix $\sigma \in [0,1)$, take $M \in \NN$ such that
$M > (3+\sigma)/(1-\sigma)$, and let $\ww$ be the solution of \eqref{linear0}--\eqref{eq:zeromoments}
with initial data $\ww|_{t=0}=0$. There exist positive constants $C$ and $\delta_0$ such that,
for any $\Gamma > 0$, any $d > 0$ and any $\nu > 0$ with $\delta := \nu/\Gamma \le \delta_0$,
the following holds:
\begin{equation}\label{bd:wmain}
  \|(\Omega-\Omega_{\app})(t)\|_{\cX_\eps} \,=\, \delta \|\ww(t)\|_{\cX_\eps} \,\le\,
  C\bigl(\delta^{-1}\eps^{M+1}+\delta\eps^2\bigr), \quad \text{for all}~\, t \in
  \bigl(0,T_\adv\delta^{-\sigma}\bigr)\,,
\end{equation}
where $\eps$ and $T_\adv$ are as in \eqref{def:all}, and the function space $\cX_\eps\hookrightarrow
L^1(\mathbb{R}^2)$ is defined in \eqref{def:Xeps} below. Moreover, the  vertical speed satisfies
\begin{equation}\label{eq:ZpdefNL}
  \frac{d}{\Gamma}\,Z_2'(t) \,=\, \frac{1}{\eps}\int_{\RR^2}\bigl(\partial_1\cT_{\eps}\Psi_\app
  \bigr)\Omega_\app\,\dd \xi + \cO\bigl(\eps^M+\delta^{-1}\eps^{M+1}+\delta^2\eps
 +\delta \eps^2\bigr)\,.
\end{equation}
\end{theorem}

\begin{remark}
As is explained in Sections~\ref{sec:Arnold} and \ref{sec:weight}, we need to introduce
a carefully designed weighted space $\cX_\eps$ to fully exploit the stability properties
of our approximate solution $\Omega_\app$ of \eqref{eq:Omss}. At this stage, however,
it is enough to know that $\cX_\eps\hookrightarrow L^1(\RR^2)$ uniformly in $\epsilon$. 
Observing that $\delta^{-1}\eps^{M+1}\le \eps^2$ when $t \le T_\adv\delta^{-\sigma}$ and 
$\delta$ is small enough, and recalling that $\Omega_\app = \Omega_0 + \cO(\eps^2)$ by 
\eqref{appOP}, we see that \eqref{bd:wmain} readily implies the estimate \eqref{eq:2vortconv}
in Theorem~\ref{thm:main}, in view of \eqref{eq:omapp} and \eqref{eq:ansom}. 
Similarly, assuming that $M \ge 5$ is sufficiently large, the expression \eqref{eq:Zpdef}
of the vertical speed follows from \eqref{eq:ZpdefNL} if we use the relation \eqref{zetaapp} 
and the expression of $\zeta_\app$ computed in Proposition~\ref{prop:speed}. 
\end{remark}

\smallskip
The rest of this section is devoted to the proof of Theorem \ref{th:mainNL},
which is organized as follows. In Section~\ref{sec:Arnold} we isolate the most
dangerous terms in the equation \eqref{linear0}--\eqref{force0}, and we explain
why they are difficult to control. We then discuss how to overcome these issues
in a simplified situation where the geometric ideas underlying Arnold's
variational principle can be presented without too many technicalities.  The
functional framework needed to prove Theorem~\ref{th:mainNL} is introduced in
Section~\ref{sec:weight}, where the $\eps$-dependent weighted space $\cX_\eps$
appearing in \eqref{bd:wmain} is precisely defined. The short Section~\ref{sec:velocity}
is entirely devoted to the control of the correction $\zeta$ to the vertical speed.
In Section~\ref{sec:energy} we introduce our main energy functional, inspired
from Arnold's theory, and we study its coercivity properties. We also 
state the key Proposition~\ref{prop:key}, which is the core of the proof
of Theorem~\ref{th:mainNL}. The time derivative of our energy functional is
computed in Section~\ref{sec:energyId}, and consists of various terms
that are estimated in the subsequent sections~\ref{sec:diffusion}--\ref{sec:NL}. Once
this is done, a simple Gr\"onwall-type argument allows us to complete the proof
of Proposition~\ref{prop:key}, hence also of Theorem~\ref{th:mainNL}.

\subsection{Main difficulties and Arnold's variational principle}\label{sec:Arnold}

Before proceeding with the analysis of the evolution equation \eqref{linear0}--\eqref{force0},
it is convenient to identify the terms that produce the main contributions. We recall that,
according to \eqref{def:ENS}, our approximate solution can be decomposed as
\[
  \Omega_\app \,=\, \Omega_\app^E + \delta \Omega_\app^{NS}\,, \qquad
  \Psi_\app \,=\, \Psi_\app^E + \delta \Psi_\app^{NS}\,, \qquad
  \zeta_\app \,=\, \zeta_\app^E + \delta \zeta_\app^{NS}\,,
\]  
where $\Omega_\app^E$ is the Eulerian approximation already considered in
Section~\ref{ssec35}, and $\Omega_\app^{NS}$ is a viscous correction. In analogy with 
the definition of $\Phi_\app^E$ in \eqref{def:PhiappE}, we denote
\begin{equation}\label{def:PhiappNS}
  \Phi_\app^{NS} \,:=\, \Psi_\app^{NS} - \cT_\eps\Psi_\app^{NS}  + \frac{\eps\xi_1}{2\pi}
  \,\zeta_\app^{NS}\,.
\end{equation}
The equation \eqref{linear0}--\eqref{force0} can then be written in the equivalent form
\begin{align}\label{linear}
  t\partial_t \ww-\cL\ww \,=\, \,&\frac{1}{\delta}\left(\left\{\Phi_\app^E,\ww\right\}
  +\left\{\varphi-\cT_\eps \varphi,\Omega_\app^{E}\right\}\right)+\left\{\Phi_\app^{NS},\ww\right\}
  +\left\{\varphi-\cT_\eps \varphi,\Omega_\app^{NS}\right\}\\ \label{force}
  & +\bigl\{\varphi-\cT_\eps \varphi\,,\ww\bigr\} +\frac{1}{\delta^2}\,\cR_M
  +\frac{\eps\,\zeta}{2\pi\delta^2}\bigl\{\xi_1,\Omega_{\app}+\delta \ww\bigr\}\,.
\end{align}

As is explained in the introduction, the main challenge in the proof of Theorem
\ref{th:mainNL} is to control the vorticity perturbation $\ww$ over the long
time interval $(0,T_\adv \delta^{-\sigma})$, where $\sigma\in(0,1)$. The size of
$\ww$ is measured in a function space $\cX_\eps\hookrightarrow L^1(\RR^2)$, keeping
in mind that $L^1(\RR^2)$ is the Lebesgue space whose norm is invariant under
the self-similar scaling \eqref{eq:ansom}.  The best we can hope for is to
propagate the bounds we have on the forcing term, and thanks to
Proposition~\ref{prop:appsol} we know that $\delta^{-2}\cR_M = \cO_\cZ (\delta^{-2}
\eps^{M+1}+\eps^2)$. This explains the right-hand side of estimate
\eqref{bd:wmain} in Theorem~\ref{th:mainNL}.

To control the nonlinear terms in \eqref{force}, we have to make sure that $\ww$
remains small, and in the light of the above we must therefore assume that
$\delta^{-2}\eps^{M+1}\ll 1$. Recalling that $\eps \le \delta^{(1-\sigma)/2}$
when $t \le T_\adv \delta^{-\sigma}$, we see that this condition is met when
$\delta > 0$ is small if we suppose that $M$ is large enough so that 
$M>(3+\sigma)/(1-\sigma)$. The term involving $\delta^{-2}\eps\zeta \{\xi_1,
\Omega_{\app}\}$ is also potentially problematic, because the bound on $\zeta$
that will be established in Section~\ref{sec:velocity} below does not compensate
for the large prefactor $\delta^{-2}$, but due to a subtle cancellation (related 
to translation invariance in the vertical direction) this term will not seriously
affect our energy estimates.

More importantly, we have to handle carefully the linear terms in
\eqref{linear}, especially the nonlocal ones involving the stream function
perturbation $\varphi$. The most dangerous terms are multiplied by $\delta^{-1}$
and correspond to the linearization of \eqref{eq:Omss} at the Eulerian
approximate solution $\Omega_\app^E$. In contrast, the contributions due to the
viscous corrections are of size $\cO(\eps^2)$ and will be easy to control. As
was already observed in \cite{Gallay2011}, the linearized operator at the naive
approximation $\Omega_0$ is skew-symmetric in a Gaussian weighted space, so that
the leading order terms in \eqref{linear} do not contribute to the energy
estimates in that space. Moreover, the corrections due to the difference
$\Omega_\app - \Omega_0$ are proportional to $\delta^{-1}\epsilon^2$, a quantity that
remains small as long as $t \ll T_\adv$.

The real difficulties begin when one tries to control the solution $\ww$ of
\eqref{linear}--\eqref{force} on a time interval $(0,T)$ with $T \gg T_\adv$.
If $T$ is independent of $\delta = \nu/\Gamma$, this can still be done using an
appropriate modification of the Gaussian weight in the energy estimates, see
\cite{Gallay2011}. However, to reach longer time scales corresponding to
$T = T_\adv\delta^{-\sigma}$ with $\sigma > 0$, we need a different and more
robust approach that fully exploits the stability properties of our approximate
solution $\Omega_{\app}$. In the particular case under consideration, we know
that $\Omega_{\app}^E$ is very close to a traveling wave solution of the 2D
Euler equation, and we even have an approximate functional relationship
between $\Omega_{\app}^E$ and $\Phi^E_\app$, as shown in Section~\ref{ssec35}.
For steady states (or traveling waves) of the 2D Euler equations with a global
functional relationship between vorticity and stream function, Arnold
\cite{Arnold66} introduced a beautiful and general variational principle that
can be used to investigate stability. This approach was recently revisited in
\cite{GaSarnold} and successfully applied to the study of the vanishing viscosity
limit for axisymmetric vortex rings \cite{GaSring}, a problem that has many
similarities with the case of vortex dipoles. In the rest of this paragraph, 
we briefly present Arnold's general strategy and we explain how it can be
implemented in our situation. 

\medskip
Assume that we are given a steady state of the two-dimensional Euler equation,
with the property that the vorticity $\omega_*$ and the associated stream function 
$\psi_* = \Delta^{-1}\omega_*$ satisfy a global functional relationship of the form
$\psi_*= F(\omega_*)$, where $F$ is a smooth function. Following \cite{Arnold66}
we consider the energy functional 
\begin{equation}\label{def:toyEE}
  \cE[\omega] \,:=\, \int \cF(\omega)\,\dd x - \frac12\int \psi \omega\,\dd x\,,
\end{equation}
where $\cF$ is a primitive of $F$. Here and in what follows we are vague about
the spatial domain under consideration, and we do not perform rigorous
calculations.  The right-hand side of \eqref{def:toyEE} is the sum of the
kinetic energy of the fluid and a Casimir functional, so that
$\cE[\omega]$ is conserved under the evolution defined by the Euler
equation. The first variation of $\cE$ at $\omega_*$ vanishes by construction,
and taking the second variation we obtain the quadratic form
\begin{equation}\label{def:ArnQ}
  \cE''(\omega_*)[\omega,\omega] \,=\, \int F'(\omega_*)\omega^2\,\dd x -
  \int \psi\omega\,\dd x\,.
\end{equation}
A key observation is that this quadratic function of $\omega$ is invariant
under the evolution defined by the linearized Euler equation at $\omega_*$. Indeed,
if $\partial_t \omega = -\{\psi_*,\omega\}-\{\psi,\omega_*\}$, a direct calculation 
shows that $\cE''(\omega_*)[\omega,\omega]$ does not vary in time. As a consequence, 
if the quadratic form \eqref{def:ArnQ} has a definite sign, it can be used to prove
the stability of the steady state $\omega_*$, not only for the Euler equation but
for related systems as well, see \cite{GaSarnold} for an application to the
stability of vortices in the 2D Navier-Stokes equations. 

We now explain how to implement Arnold's approach in our situation. We consider
a simplified version of the evolution equation \eqref{linear}--\eqref{force},
where we rescale time and only keep the problematic terms that we have just identified.
Given a small parameter $\eps > 0$ and an arbitrary real number $\zeta$,
our model system reads
\begin{equation}\label{eq:toyArnold}
  \partial_\tau \ww \,=\, \bigl\{\Phi_\app^E\,,\ww\bigr\} + \bigl\{\varphi-\cT_\eps \varphi
  + \zeta \xi_1\,,\Omega_\app^E\bigr\}\,, \qquad \Delta\varphi \,=\, \ww\,,
\end{equation}
where $\Omega_{\app}^E$ is the Eulerian approximate solution and $\Phi_\app^E$
is given by \eqref{def:PhiappE}. Extrapolating the conclusions of
Proposition~\ref{prop:germ}, we assume for simplicity that we have
an exact, global relationship of the form $\Phi_\app^E = -F(\Omega_\app^E)$,
for some $F : \RR \to \RR$ (the minus sign is introduced for convenience,
just to ensure that $F$ is an increasing function.) Inspired by \eqref{def:ArnQ},
we define the quadratic functional
\begin{equation}\label{def:Arnold_energy}
  E(\ww) \,:=\, \frac12 \int F'(\Omega_{\app}^E)\ww^2\,\dd\xi \,+\, \frac12 \int
  \bigl(\varphi-\cT_\eps \varphi\bigr) \ww\,\dd\xi\,,
\end{equation}
where the last term is, up to a sign, the total energy of the fluid, taking into
account the mirror vortex located at $\xi = (-1/\eps,0)$. 

As before, we claim that the functional $E(\ww)$ is invariant under the linear
evolution defined by \eqref{eq:toyArnold}. To prove this, we first observe that
\begin{equation}\label{def:Arn1}
  \partial_\tau E[\ww] \,=\, \int F'(\Omega_{\app}^E)\ww \partial_\tau \ww\,\dd\xi \,+\,
  \int \bigl(\varphi-\cT_\eps \varphi\bigr) \partial_\tau \ww\,\dd\xi\,,
\end{equation}
where, to obtain the second term, we used the identities
\begin{equation}\label{def:Arn2}
  \int (\partial_\tau\varphi)\ww\,\dd\xi \,=\, \int \varphi\partial_\tau\ww\,\dd\xi\,, \qquad
  \int (\cT_\eps\partial_\tau\varphi)\ww\,\dd\xi \,=\, \int (\cT_\eps\varphi)\partial_\tau\ww\,\dd\xi\,.
\end{equation}
The first one is established by writing $w = \Delta\varphi$ and integrating by parts,
the second one using in addition the fact that the translation-reflection operator
$\cT_\eps$ defined in \eqref{def:Teps} is self-ajdoint in $L^2(\RR^2)$ and commutes with
the Laplacian. 

It remains to evaluate the right-hand side of \eqref{def:Arn1}. Starting from 
\eqref{eq:toyArnold} and using the functional relation $\Phi_\app^E + F(\Omega_\app^E) = 0$,
we find
\begin{align*}
  \int F'(\Omega_{\app}^E)\ww \partial_\tau \ww\,\dd\xi \,&=\,
  \int F'(\Omega_{\app}^E)\ww \Bigl(\bigl\{\ww\,,F(\Omega_\app^E)\bigr\} + \bigl\{\varphi-\cT_\eps \varphi
  + \zeta \xi_1\,,\Omega_\app^E\bigr\}\Bigr)\,\dd\xi \\
  \,&=\, 0 + \int \ww \bigl\{\varphi-\cT_\eps \varphi + \zeta \xi_1\,,F(\Omega_\app^E)\bigr\}
  \,\dd\xi\,,
\end{align*}
where we invoke familiar identities involving Poisson brackets, such as
$F'(a)\{a,b\} = \{F(a),b\}$, $\{F(a),G(a)\} = 0$, and $\int \{a,b\}c\,\dd\xi =
\int a\{b,c\}\,\dd\xi$. Similarly we obtain
\begin{align*}
   \int \bigl(\varphi-\cT_\eps \varphi\bigr) \partial_\tau \ww\,\dd\xi \,&=\, 
   \int \bigl(\varphi-\cT_\eps \varphi\bigr)\Bigl(\bigl\{\ww\,,F(\Omega_\app^E)\bigr\}
   + \bigl\{\varphi-\cT_\eps \varphi + \zeta \xi_1\,,\Omega_\app^E\bigr\}\Bigr)\,\dd\xi \\
  \,&=\, 0 +  
   \int \bigl(\varphi-\cT_\eps \varphi\bigr)\Bigl(\bigl\{\ww\,,F(\Omega_\app^E)\bigr\}
   + \zeta \bigl\{\xi_1\,,\Omega_\app^E\bigr\}\Bigr)\,\dd\xi\,.
\end{align*}
So we deduce from \eqref{def:Arn1} that
\begin{equation}\label{def:Arn3}
\begin{split}
  \partial_\tau E[\ww] \,&=\, \zeta \int \Bigl(w\bigl\{\xi_1\,,F(\Omega_\app^E)\bigr\}
  + \bigl(\varphi-\cT_\eps \varphi\bigr) \bigl\{\xi_1\,,\Omega_\app^E\bigr\}\Bigr)\,\dd\xi \\
  \,&=\, \zeta \int w \bigl\{\xi_1\,,F(\Omega_\app^E) + \Phi_\app^E\bigr\}\,\dd\xi \,=\, 0\,,
\end{split}
\end{equation}
which is the desired result. Here, in the last line, we used the identity
\[
  \int \bigl(\varphi-\cT_\eps \varphi\bigr) \bigl\{\xi_1\,,\Omega_\app^E\bigr\}\Bigr)\,\dd\xi
  \,=\, \int w \bigl\{\xi_1\,,\Psi_\app^E - \cT_\eps\Psi_\app^E\bigr\}\,\dd\xi
  \,=\, \int w \bigl\{\xi_1\,,\Phi_\app^E\bigr\}\,\dd\xi\,,
\]
which is established in the same way as \eqref{def:Arn2}. 

The main purpose of the heuristic arguments above is to explain how to construct
an energy functional that is naturally adapted to the leading order terms in the
evolution equation \eqref{linear}--\eqref{force}. Although the general strategy
is clear, many technical difficulties arise when turning these ideas into a
rigorous proof. For instance, we can only exploit the functional relationship
between $\Phi_\app^E$ and $\Omega_\app^E$ in the region where we are able to
justify it, namely for $|\xi| \le 2\eps^{-\sigma_1}$ with $\sigma_1 > 0$
sufficiently small. Outside the vortex core, the energy functional has to be
substantially modified, but we can exploit the fact that our approximate
solution $\Omega_\app$ is extremely small in that region. In addition,
understanding the coercivity properties of the energy functional and its
interplay with the dissipation operator $\cL$ is highly non trivial. Similar
problems were addressed in the previous works \cite{GaSarnold,GaSring}, but here
we have to face the additional difficulty of handling a perturbative expansion
to an arbitrary order in the parameter $\epsilon$.

\subsection{The weighted space and its properties}\label{sec:weight}

After these preliminaries, we construct the weight function that will
enter our energy functional. We give ourselves three real numbers $\sigma_1$, $\sigma_2$,
and $\gamma$ satisfying
\begin{equation}\label{def:parameters}
  0 \,<\, \sigma_1 \,<\, \frac12\,, \qquad \sigma_2 \,>\, 1\,, \qquad
  \gamma \,=\, \frac{\sigma_1}{\sigma_2} \,<\, \frac12\,.
\end{equation}
In the course of the proof the parameter $\sigma_1$ will be chosen small, depending on 
the order $M$ of the approximate solution \eqref{appOP}, whereas $\sigma_2$ will be large. 
In particular we assume that $\sigma_1$ is small enough to ensure the validity of
Corollary~\ref{cor:frel}, which asserts the existence of an approximate
functional relationship between the vorticity $\Omega_\app^E$ and the stream
function $\Phi_\app^E$ defined by \eqref{def:PhiappE}. More precisely, we define
\begin{equation}\label{def:Theta}
  \Theta(\eps,\xi) \,:=\, \Phi_\app^E(\xi) + F\bigl(\Omega_\app^E(\xi)\bigr)\,,
\end{equation}
where $F \in \cK$ is the function introduced in Proposition~\ref{prop:germ}.
Then, according to \eqref{eq:frelbd}, there exists an integer $N \in \NN$ such that 
\begin{equation}\label{bd:Theta}
  |\nabla_\xi \Theta(\eps,\xi)| \,\lesssim\, \eps^{M+1}(1+|\xi|)^N, \qquad \text{ for }\,
  |\xi| \,\le\, 2\eps^{-\sigma_1}\,.
\end{equation}
For later use, we also recall that $\Omega_0(\xi)/2 \le \Omega_\app^E(\xi) \le 2\Omega_0(\xi)$
when $|\xi| \le 2\eps^{-\sigma_1}$, see \eqref{eq:Ompos}. 

\smallskip
We now decompose the space domain $\RR^2$ into three disjoint regions: 
\begin{align}
  \label{inner}\tag{Inner}\mathrm{I}_{\eps} \,&=\, \bigl\{\xi\in \RR^2 \,:\, |\xi| < 2\eps^{-\sigma_1}\,,
  ~F'\bigl(\Omega_\app^E(\xi)\bigr) < \exp(\eps^{-2\sigma_1}/4)\bigr\}\,,\hspace{0.8cm}\\
  \label{intermediate}\tag{Intermediate}\mathrm{II}_{\eps} \,&=\, \bigl \{\xi\in \RR^2\,:\, 
  \xi \notin \mathrm{I}_\eps\,,~|\xi| \le \eps^{-\sigma_2}\bigr\}\,,\\
  \label{outer}\tag{Outer}\mathrm{III}_{\eps} \,&=\, \bigl\{\xi\in \RR^2\,:\, |\xi| > \eps^{-\sigma_2}\bigr\}\,,
\end{align}
which depend on time through the parameter $\eps = \sqrt{\nu t}/d$. Our weight function is defined as
\begin{equation}\label{def:Weps}
  W_{\eps}(\xi) \,=\, \begin{cases}
  F'\bigl(\Omega_{\app}^E(\xi)\bigr) \quad &\text{in } \,\mathrm{I}_{\eps}\,,\\[1mm]
  \exp(\eps^{-2\sigma_1}/4) \quad &\text{in } \,\mathrm{II}_{\eps}\,,\\[1mm]
  \exp(|\xi|^{2\gamma}/4) \quad &\text{in } \,\mathrm{III}_{\eps}\,.
\end{cases}
\end{equation}
For any fixed $\epsilon > 0$, the weight $W_\epsilon$ is a positive, locally Lipschitz and piecewise
smooth function, but the derivative $\nabla W_\epsilon$ has a discontinuity at the boundaries
of the regions $\mathrm{I}_\eps, \mathrm{II}_\eps, \mathrm{III}_\eps$, see Fig.~\ref{fig2}. 
Finally we introduce the weighted $L^2$ space 
\begin{equation}\label{def:Xeps}
  \cX_{\eps} \,=\, \Bigl\{f\in L^2(\RR^2)\,:\, \norm{f}_{\cX_\eps}^2 := \int_{\RR^2}W_\eps(\xi)|f(\xi)|^2
  \,\dd\xi < \infty\Bigr\}\,.
\end{equation}
The estimates established in Proposition~\ref{propIW} below readily imply that
\[
  \cY \,\hookrightarrow \cX_{\eps} \,\hookrightarrow\, L^p(\RR^2), \qquad \text{for all } p \in [1,2]\,,
\]
where $\cY$ is the Gaussian space defined in \eqref{def:cY}. Moreover, in the
limit where $\eps \to 0$, it is easy to verify that $W_\eps(\xi) \to W_0(\xi) :=
F_0'(\Omega_0(\xi)) \equiv A(\xi)$, where $A$ is defined in \eqref{def:A}. We denote
by $\cX_0$ the space {\eqref{def:Xeps}} with $\eps = 0$.

\begin{figure}[ht]
  \begin{center}
  \begin{picture}(200,160)
  \put(0,5){\includegraphics[width=1.2\textwidth]{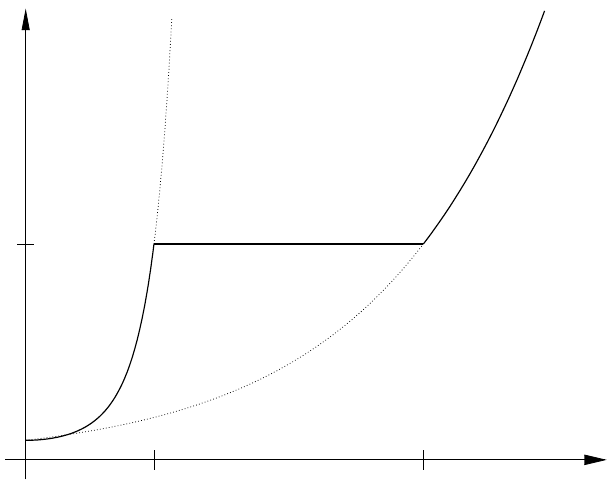}}
  \put(46,17){$\epsilon^{-\sigma_1}$}
  \put(129,17){$\epsilon^{-\sigma_2}$}
  \put(170,18){$|\xi|$}
  \put(1,16){$1$}
  \put(-60,77){$\exp(\epsilon^{-2\sigma_1}/4)$}  
  \put(80,85){$W_\epsilon(\xi)$}  
  \put(56,130){$W_0(\xi)$}
  \put(168,130){$\exp(|\xi|^{2\gamma}/4)$}
  \put(23,0){$\mathrm{I}_\eps$}
  \put(85,0){$\mathrm{II}_\eps$}
  \put(150,0){$\mathrm{III}_\eps$}
  \end{picture}
  \captionsetup{width=.9\linewidth}
  \caption{{\small 
  A schematic representation of the graph of the weight function $W_\epsilon$ defined
  in \eqref{def:Weps}. In the inner region $\mathrm{I}_\eps$, the weight is close for 
  $\eps > 0$ small to the radially symmetric function $W_0(\xi) = 4|\xi|^{-2} 
  \bigl(\e^{|\xi|^2/4}-1\bigr)$. It then takes constant values in the intermediate
  region $\mathrm{II}_\eps$, and grows like $\exp(|\xi|^{2\gamma}/4)$ in
  the outer region $\mathrm{III}_\eps$. The dashed lines illustrate the bounds
  \eqref{bd:Weps}, where the constants $C_1, C_2$ are independent of $\eps$.}}\label{fig2}
  \end{center}
\end{figure}

\begin{remark}\label{rem:weight}
In the inner region $\mathrm{I}_\eps$, the weight \eqref{def:Weps} is constructed
following exactly Arnold's approach as discussed in Section~\ref{sec:Arnold}.
This is possible because the quantity $\Theta = \Phi_\app^E + F(\Omega_\app^E)$ is
small in that region, which means that we almost have a functional relationship between
the vorticity and the stream function. In the intermediate region $\mathrm{II}_\eps$ the
weight is just a constant, so that the dangerous advection terms multiplied by
$\delta^{-1}$ in \eqref{linear} will disappear after an integration
by parts. Finally, in the outer region $\mathrm{III}_\eps$, the evolution
defined by \eqref{linear}--\eqref{force} is essentially driven by the diffusion operator
$\cL$, and it turns out that a radially symmetric weight with moderate growth at infinity is
appropriate in that case. Note that our definition of the weight in the intermediate and outer
regions is the same as in \cite{Gallay2011}, whereas the Arnold strategy for the inner region
was put forward in \cite{GaSring}. There is a minor simplification here with respect to \cite{GaSring}:
the weight \eqref{def:Weps} is independent of the inverse Reynolds number $\delta$, 
because it only involves the Eulerian approximation $\Omega_\app^E$. 
\end{remark}

We collect elementary properties of the inner region $\mathrm{I}_\eps$ and of the weight function
$W_\eps$ in the following lemma, which is the analogue of \cite[Lemmas 4.2 \& 4.3]{GaSring}.

\begin{proposition}\label{propIW}
If $\eps,\sigma_1 > 0$ are small enough, the following holds true:
\begin{itemize}[leftmargin=30pt,itemsep=2pt]
\item[i)] The inner region $\mathrm{I}_\eps$ is diffeomorphic to an open disk, and there exists a constant
$\kappa>0$ such that 
\begin{equation}\label{bd:Ieps}
  \left\{|\xi|\le \eps^{-\sigma_1}\right\} \,\subset\, \mathrm{I}_\eps\subset \left\{|\xi|^2 \le 
  \eps^{-2\sigma_1}+\kappa|\log (\eps)|\right\}\,.
\end{equation}
\item[ii)] There exist $C_1,C_2 > 0$ such that the weight $W_\eps$ satisfies the uniform bounds 
\begin{equation}\label{bd:Weps}
  C_1\exp(|\xi|^{2\gamma}/4)\,\le\, W_\eps(\xi) \,\le\, C_2 W_0(\xi)\,, \qquad \text{for }\,\xi \in \RR^2\,,
\end{equation}
where $W_0 = A$ is defined in \eqref{def:A} and $\gamma$ in \eqref{def:parameters}. 
\item[iii)] For any $\gamma_2 < 2$, there exists $C_3 > 0$ such that 
\begin{equation}\label{bd:WAI}
  |W_\eps(\xi)-W_0(\xi)| + |\nabla W_\eps(\xi)-\nabla W_0(\xi)|\,\le\, C_3\eps^{\gamma_2}W_0(\xi)\,,
  \qquad \text{for }\,\xi \in \mathrm{I}_\eps\,.
\end{equation}
\end{itemize} 
\end{proposition}

\begin{proof}
The main step is to show that, if $\sigma_1 > 0$ is small enough, 
\begin{equation}\label{bd:claim}
  \bigl|F'(\Omega_{\app}^E(\xi))-W_0(\xi)\bigr| + \bigl|\nabla \bigl(F'(\Omega_{\app}^E(\xi)) - W_0(\xi)\bigr)
  \bigr| \,\lesssim\, \eps^{\gamma_2}W_0(\xi)\,, \quad \text{for }\, |\xi| \le 2\eps^{-\sigma_1}\,,
\end{equation}
where $F = F_0 + \eps^2 F_2 + \dots + \eps^M F_M$ is the function defined in
Proposition~\ref{prop:germ}. Since $F \in \cK$ we know in particular that $F''(\Omega_0)\Omega_0^2
\in \cS_*(\RR^2)$, see Lemma~\ref{lem:Fk}. As is easily verified, this implies that there exists 
$C > 0$ and $N \in \NN$ such that
\begin{equation}\label{bd:lambound}
  \sup_{1/2 \le \lambda \le 2} \,\bigl|F''\bigl(\lambda\Omega_0(\xi)\bigr)\bigr| \,\le\, C\,
  (1+|\xi|)^N\,\Omega_0(\xi)^{-2}\,, \qquad \xi \in \RR^2\,.
\end{equation}
In the region where $|\xi| \le 2\eps^{-\sigma_1}$, we observe that
\begin{equation}\label{bd:Fbd1}
  F'(\Omega_{\app}^E(\xi)) - F'(\Omega_0(\xi)) \,=\, \bigl(\Omega_{\app}^E(\xi) - \Omega_0(\xi)\bigr)
  \int_0^1 F''\bigl((1{-}s)\Omega_0(\xi) + s\Omega_\app^E(\xi)\bigr)\,\dd s\,,
\end{equation}
where the integrand can be estimated using \eqref{bd:lambound} since $\Omega_0(\xi)/2 \le
\Omega_\app^E(\xi) \le2\Omega_0(\xi)$. By definition of $F$ we also have
\begin{equation}\label{bd:Fbd2}
  F'(\Omega_0(\xi)) - W_0(\xi) \,=\, F'(\Omega_0(\xi)) - F_0'(\Omega_0(\xi)) \,=\, \sum_{k=2}^M
  \epsilon^k F_k'(\Omega_0(\xi))\,, \qquad \xi \in \RR^2\,.
\end{equation}
Combining \eqref{bd:Fbd1}, \eqref{bd:Fbd2} and using the fact that $\Omega_\app^E - \Omega_0 =
\cO_\cZ(\epsilon^2)$, we obtain an estimate of the form
\[
  \bigl|F'(\Omega_\app^E(\xi)) - W_0(\xi)\bigr| \,\le\, C \eps^2(1+|\xi|)^N\,\Omega_0(\xi)^{-1}\,,
  \qquad |\xi| \le 2\eps^{-\sigma_1}\,,
\]
with a possibly larger exponent $N$. Since $\Omega_0(\xi)^{-1} \lesssim \e^{|\xi|^2/4} \lesssim
(1+|\xi|)^2 W_0(\xi)$, we obtain the desired estimate for the first term in the left-hand side
of \eqref{bd:claim} by taking $\sigma_1$ sufficiently small so that $\eps^{2-\sigma_1(N+2)}
\le \eps^{\gamma_2}$ when $\epsilon \ll 1$. The corresponding bound on $\nabla (F'(\Omega_\app^E)-W_0)$
is obtained by differentiating \eqref{bd:Fbd1}, \eqref{bd:Fbd2} and proceeding similarly. 

Estimate \eqref{bd:claim} shows in particular that $F'(\Omega_\app^E(\xi))$ is very close to
$W_0(\xi)$ when $|\xi| \le 2\eps^{-\sigma_1}$, and this implies that the region $\mathrm{I}_\eps$
satisfies the inclusions \eqref{bd:Ieps}. Also, using the definition \eqref{def:Weps}, 
it is straightforward to verify that the weight $W_\epsilon$ satisfies the bounds \eqref{bd:Weps}
in all three regions $\mathrm{I}_\eps$, $\mathrm{II}_\eps$, and $\mathrm{III}_\eps$. Finally
\eqref{bd:WAI} immediately follows from \eqref{bd:claim} since $W_\epsilon = F'(\Omega_\app^E)$
in region $\mathrm{I}_\eps$. 
\end{proof}

Assuming that the vorticity lies in the function space $\cX_\eps$ for some $\epsilon > 0$,
we next recall a few classical estimates for the stream function and the velocity field.
For the reader's convenience, the proof of the following statements is reproduced
in Section~\ref{ssecA4}.

\begin{lemma}\label{lem:velocityLq}
If $\ww\in \cX_\eps$ and $ \varphi=\Delta^{-1}\ww$ in the sense of \eqref{UPsirel}, there
exists a constant $C > 0$ such that $|\varphi(\xi)| \le C\|\ww\|_{\cX_\eps}\log(2+|\xi|)$
for all $\xi \in \RR^2$. Moreover, for all $q \in (2,\infty)$, there holds
\begin{equation}\label{bd:velocityLq}
  \|(1+|\cdot|)^{-1}\varphi\|_{L^q} + \|\nabla \varphi\|_{L^q} \,\le\, C\norm{\ww}_{\cX_\eps}.
\end{equation}
If in addition $\nabla \ww_\eps \in \cX_\eps$, then
\begin{equation}\label{bd:velocityLinfty}
  \norm{(1+|\cdot|)\nabla \varphi}_{L^\infty}\,\le\, C\bigl(\norm{\nabla\ww}_{\cX_\eps}^{1/2} +
  \norm{\ww}_{\cX_\eps}^{1/2}\bigr)\norm{\ww}_{\cX_\eps}^{1/2}\,.
\end{equation}
\end{lemma}

The conclusions of Lemma~\ref{lem:velocityLq} can be considerably strengthened if
the mean and the first order moments of $\ww$ are assumed to vanish.

\begin{lemma}\label{lem:withmom}
Assume that $\ww \in \cX_\eps$ satisfies $\mrM[\ww]=\mrm_1[\ww]=\mrm_2[\ww]=0$, and let
$\varphi=\Delta^{-1}\ww$ as in \eqref{UPsirel}. Then for all $q > 2$ there exists a constant
$C > 0$ such that
\begin{equation}\label{bd:streammom}
  \|(1+|\cdot|)^2\varphi\|_{L^\infty} + \|(1+|\cdot|)^2\nabla\varphi\|_{L^q} 
   \,\le\, C\norm{\ww}_{\cX_\eps}.
\end{equation}
If in addition $\nabla \ww_\eps \in \cX_\eps$, then 
\begin{equation}\label{bd:velocitymom}
  \|(1+|\cdot|)^3\nabla\varphi\|_{L^\infty} \,\le\, C\bigl(\norm{\nabla\ww}_{\cX_\eps}^{1/2} +
  \norm{\ww}_{\cX_\eps}^{1/2}\bigr)\norm{\ww}_{\cX_\eps}^{1/2}\,.
\end{equation}
\end{lemma}

The bounds \eqref{bd:streammom}, \eqref{bd:velocitymom} will be used in particular
to estimate the translated stream function $\cT_\eps\varphi$ in the inner region near the origin. 

\begin{corollary}\label{cor:TepsXeps}
Under the assumptions of Lemma~\ref{lem:withmom} we have, for any $q \in (2,\infty)$, 
\begin{equation}\label{bd:UTesp}
  \bigl\|\mathbbm{1}_{\mrI_\eps}\cT_\eps \varphi\bigr\|_{L^\infty} \,\le\, C\eps^2
  \|\ww \|_{\cX_\eps}\,, \qquad
  \bigl\|\mathbbm{1}_{\mrI_\eps}(1+|\cdot|)^{-1}\nabla \cT_\eps\varphi\bigl\|_{L^q}
  \,\le\, C\eps^3\|\ww \|_{\cX_\eps}\,.
\end{equation}
\end{corollary}

\subsection{Bound on the vertical speed}\label{sec:velocity}

Before setting up the nonlinear energy estimates, we apply Lemma~\ref{lem:velocityLq} and
Corollary~\ref{cor:TepsXeps} to control the size of the correction $\zeta$ to the vertical speed. 

\begin{lemma}\label{lem:zeta2}
Let $\zeta$ be defined as in \eqref{def:zeta}. Then there exists a constant $C>0$ such that
\begin{equation}\label{bd:zeta2}
  \frac{\eps}{2\pi}|\zeta(t)| \,\le\, C\bigl(\eps^{M+1}+\delta^2\eps^2+\delta\eps
  \norm{\ww}_{\cX_\eps}+\delta^2\eps^3\norm{\ww}_{\cX_\eps}^2\bigr)\,.
\end{equation}
\end{lemma}

\begin{proof}
In the definition \eqref{def:zeta} of $\zeta$, we expand the vorticity $\Omega$ and the
stream function $\Psi$ as in \eqref{def:per}, and we subtract the approximation $\zeta_\app$
given by \eqref{zetaapp}. This gives the expression
\begin{equation}\label{eq:expansion_zeta}
  \frac{\eps}{2\pi}\zeta \,=\, \delta \langle \partial_1 \cT_\eps \Psi_\app,\ww\rangle_{L^2}
  + \delta\langle \partial_1 \cT_\eps \varphi,\Omega_\app\rangle_{L^2}+\delta^2\langle \partial_1
  \cT_\eps \varphi,\ww\rangle_{L^2} + \cO\bigl(\epsilon^{M+1} + \delta^2\epsilon^2\bigr)\,,
\end{equation}
which is the starting point of our analysis. Since $\Omega_\app\in \cZ$, we already know
from Lemma~\ref{lem:Teps} that $\partial_1 \cT_\eps\Psi_\app = \cO_{\cS_*}\bigl(\eps\bigr)$.
Using the fact that $W_\eps^{-1}$ vanishes rapidly at infinity (see Proposition~\ref{propIW}),
we deduce that
\begin{equation}\label{bd:zeta_app_w}
  |\langle \partial_1 \cT_\eps \Psi_\app,\ww\rangle_{L^2}| \,\le\, \bigl\| W_{\eps}^{-1/2}
  \partial_1 \cT_\eps \Psi_\app\bigr\|_{L^2}\|\ww\|_{\cX_\eps} \,\lesssim\, \eps \|\ww\|_{\cX_\eps}\,.
\end{equation}
To bound the term involving $\Omega_\app$, we split the integration domain into the region
$\mrI_{\eps}$ and its complement $\mrI_{\eps}^c$, and we apply H\"older's inequality with
$1/p+1/q=1$ and $q>2$. We thus find
\[
  |\langle \partial_1 \cT_\eps \varphi,\Omega_\app\rangle_{L^2}| \,\lesssim \,
  \|\mathbbm{1}_{\mrI_{\eps}}(1+|\cdot|)^{-1}\partial_1 \cT_\eps \varphi\|_{L^q}
  \|(1+|\cdot|)\Omega_\app\|_{L^p} +\|\partial_1 \varphi\|_{L^q}\|\mathbbm{1}_{\mrI_{\eps}^c}
  \Omega_\app\|_{L^p}\,.
\]
As $\Omega_\app\in \cZ$ and $\mrI_{\eps}^c \subset \{\xi \in \RR^2\,;\, |\xi| > \eps^{-\sigma_1}\}$
by Proposition~\ref{propIW}, it is straightforward to verify that $\|\mathbbm{1}_{\mrI_{\eps}^c}
\Omega_\app\|_{L^p} \le \exp(-c_*\eps^{-2\sigma_1})$ for some $c_*>0$. Therefore
using Corollary~\ref{cor:TepsXeps} to bound the first term in the right-hand side and
Lemma~\ref{lem:velocityLq} for the second one, we arrive that
\begin{equation}\label{bd:zeta_phi_Omapp}
    |\langle \partial_1 \cT_\eps \varphi,\Omega_\app\rangle_{L^2}| \,\lesssim\, \eps^3\|\ww\|_{\cX_\eps}\,.
\end{equation}
Finally, the nonlinear term can be estimated in a similar way: 
\[
  |\langle \partial_1 \cT_\eps \varphi,\ww\rangle_{L^2}| \,\lesssim\, \|\mathbbm{1}_{\mrI_{\eps}}
  (1+|\cdot|)^{-1}\partial_1 \cT_\eps \varphi\|_{L^q}\|(1+|\cdot|)\ww\|_{L^p}
  +\|\mathbbm{1}_{\mrI_{\eps}^c} W_{\eps}^{-1/2}\partial_1 \cT_\eps \varphi\|_{L^2}
  \|\ww\|_{\cX_\eps}\,.
\]
Since $1 < p < 2$, H\"older's inequality readily implies that $\|(1+|\cdot|)\ww\|_{L^p}\lesssim
\|\ww\|_{\cX_\eps}$. Thus, using again \eqref{bd:velocityLq}, \eqref{bd:UTesp} and applying
H\"older's inequality with $1/q+1/r=1/2$, we obtain 
\begin{equation}\label{bd:zeta_phi_w}
  |\langle \partial_1 \cT_\eps \varphi,\ww\rangle_{L^2}| \,\lesssim\, \eps^3 \|\ww\|_{\cX_\eps}^2
  + \|\mathbbm{1}_{\mrI_{\eps}^c} W_{\eps}^{-1/2}\|_{L^r}\|\partial_1 \varphi\|_{L^q}
  \|\ww\|_{\cX_\eps} \,\lesssim\, \eps^3 \|\ww\|^2_{\cX_\eps}\,.
\end{equation}
Combining \eqref{eq:expansion_zeta} with \eqref{bd:zeta_app_w}, \eqref{bd:zeta_phi_Omapp} and
\eqref{bd:zeta_phi_w}, we arrive at \eqref{bd:zeta2}.
\end{proof}

\subsection{The energy functional}\label{sec:energy}

We now have all the necessary ingredients to define the energy functional that will allow us
to control the vorticity perturbation $w$ in \eqref{def:per}. The idea is to mimic the Arnold
quadratic form that was heuristically derived in Section~\ref{sec:Arnold}. 

Using the function space $\cX_\eps$ introduced in \eqref{def:Xeps}, we define
\begin{equation}\label{def:Eeps}
  E_{\eps}[\ww] \,=\, \frac12\Bigl(\|\ww\|_{\cX_\eps}^2 + \langle \varphi-\cT_\eps\varphi\,,
  \ww\rangle_{L^2}\Bigr)\,, \qquad w \in \cX_\eps\,.
\end{equation}
According to \eqref{def:Weps}, the functional $E_\eps[\ww]$ coincides with
\eqref{def:Arnold_energy} when $\ww$ is supported in the inner region $\mrI_\eps$.
That region expands to the whole plane $\RR^2$ as $\eps \to 0$ and, in view of
Proposition~\ref{propIW}, the weight $W_\epsilon$ in \eqref{def:Xeps} converges to $W_0 = A$,
where $A(\xi)$ is defined in \eqref{def:A}. Taking formally the limit $\eps \to 0$ in
\eqref{def:Eeps}, we thus obtain
\begin{equation}\label{def:E0}
  E_0[\ww] \,=\, \frac12\Bigl(\|\ww\|_{\cX_0}^2 + \langle \varphi,\ww\rangle_{L^2}\Bigr)\,,
  \qquad w \in \cX_0\,,
\end{equation}
where $\cX_0 = L^2(\RR^2,W_0\,\dd\xi)$. This limiting functional was studied in detail
in \cite{GaSarnold}, in connection with the stability of the Gaussian vortex
$\Omega_0$ for the Euler and the Navier-Stokes equations. A key property is that $E_0[\ww]$
is coercive on the subspace of functions $w \in \cX_0$ satisfying the moments conditions
\eqref{eq:zeromoments}. The main goal of this section is to establish a similar result for
the functional $E_\eps$ when  $\eps > 0$ is sufficiently small.

\begin{proposition}\label{prop:coercivity}
Assume that $\ww\in \cX_\eps$ satisfies $\mathrm{M}[\ww]=\mathrm{m}_1[\ww]=\mathrm{m}_2[\ww]=0$,
where $\mathrm{M},\mathrm{m}_1,\mathrm{m}_2$ are as in \eqref{def:mean-mom}. Then, if
$\eps,\sigma_1>0$ are sufficiently small, there exists a constant $\kappa_1\in(0,1)$
independent of $\eps$ such that
\begin{equation}\label{bd:coercivity}
  E_{\eps}[\ww] \,\ge\, \kappa_1 \norm{\ww}_{\cX_\eps}^2\,.
\end{equation}
\end{proposition}

\begin{proof}
Following \cite[Section~4.3]{GaSarnold}, the idea is to compare $E_\eps$ with $E_0$. First of all,
using Corollary~\ref{cor:TepsXeps} and arguing as in the proof of \eqref{bd:zeta_phi_w}, it is not
difficult to verify that 
\begin{equation}\label{bd:Eeps1}
  |\langle \cT_\eps\varphi,\ww\rangle_{L^2}| \,\lesssim\, \eps^2\|\ww\|_{\cX_\eps}^2\,.
\end{equation}
As a consequence, there exists a constant $C_0 > 0$ such that
\begin{equation}\label{bd:lowEeps1}
  E_\eps[\ww] \,\ge\, E_\eps^1[\ww] - C_0\,\eps^2\|\ww\|_{\cX_\eps}^2, \qquad\text{where}\quad
  E_\eps^1[\ww] \,=\, \frac12\Bigl(\norm{\ww}^2_{\cX_\eps}+\langle \varphi,\ww\rangle_{L^2}\Bigr)\,.
\end{equation}
To compare $E_\eps^1[\ww]$ and $E_0[\ww]$, we decompose
\[
  \ww \,=\, \mathbbm{1}_{\mrI_\eps}\ww + \bigl(1 - \mathbbm{1}_{\mrI_\eps}\bigr)\ww \,=:\,
  \ww_{\mathrm{in}}+\ww_{\mathrm{out}}\,.
\]
Denoting $\varphi_{\mathrm{in}} = \Delta^{-1}\ww_{\mathrm{in}}$ and $\varphi_{\mathrm{out}}=\Delta^{-1}
\ww_{\mathrm{out}}$, where $\Delta^{-1}$ is defined according to \eqref{UPsirel}, we find 
\begin{equation}\label{Eepsdecomp}
  \begin{split}
  E_{\eps}^1[\ww] \,&=\, E_{\eps}^1[\ww_{\mathrm{in}}] + E_{\eps}^1[\ww_{\mathrm{out}}] +
  \frac12\langle \varphi_{\mathrm{in}},\ww_{\mathrm{out}}\rangle_{L^2}
  +\frac12\langle \varphi_{\mathrm{out}},\ww_{\mathrm{in}}\rangle_{L^2} \\
  \,&=\, E_{\eps}^1[\ww_{\mathrm{in}}] + E_{\eps}^1[\ww_{\mathrm{out}}] + \langle
  \varphi_{\mathrm{in}},\ww_{\mathrm{out}}\rangle_{L^2}\,,  
\end{split}
\end{equation}
where in the second line we integrate by parts using $w_{\mathrm{in}}=\Delta \varphi_{\mathrm{in}}$
and $w_{\mathrm{out}}=\Delta \varphi_{\mathrm{out}}$. To estimate the last term in
\eqref{Eepsdecomp}, we apply H\"older's inequality with $1/q+1/p=1$ and $q>2$, $1<p<2$.
Using \eqref{bd:velocityLq} to bound the stream function $\varphi_{\mathrm{in}}$ and
recalling that $\ww_{\mathrm{out}}$ is supported outside the region $\mrI_\eps$, we obtain
\begin{equation}\label{bd:phiinomout}
  |\langle \varphi_{\mathrm{in}},\ww_{\mathrm{out}}\rangle_{L^2}| \,\lesssim\,
  \|(1+|\cdot|)^{-1}\varphi_{\mathrm{in}}\|_{L^q} \|(1+|\cdot|) \ww_{\mathrm{out}}\|_{L^p}
  \,\lesssim\, \exp(-c_*\eps^{-2\sigma_1})\|\ww\|_{\cX_\eps}^2\,,
\end{equation}
for some $c_*>0$ sufficiently small. The same estimate holds for $\langle \varphi_{\mathrm{out}},
\ww_{\mathrm{out}}\rangle_{L^2}$ too, by the same argument. It follows in particular from 
\eqref{bd:lowEeps1}, \eqref{Eepsdecomp}, \eqref{bd:phiinomout} that 
\begin{equation}\label{Eepsinter}
  E_\eps[\ww] \,\ge\, E_{\eps}^1[\ww_{\mathrm{in}}] + E_{\eps}^1[\ww_{\mathrm{out}}]
  -C_1\,\eps^2\|\ww\|_{\cX_\eps}^2\,,
\end{equation}
for some constant $C_1 > 0$.

\smallskip
It remains to estimate from below the quantities $E_{\eps}^1[\ww_{\mathrm{in}}]$ and
$E_{\eps}^1[\ww_{\mathrm{out}}]$ in \eqref{Eepsinter}. We have just observed that
\begin{equation}\label{Eeps1bd}
  E_\eps^1[\ww_{\mathrm{out}}] \,\ge\, \frac12 \|\ww_{\mathrm{out}}\|_{\cX_\eps}^2 -
  |\langle \varphi_{\mathrm{out}},\ww_{\mathrm{out}}\rangle_{L^2}| \,\ge\, \frac12
  \|\ww_{\mathrm{out}}\|_{\cX_\eps}^2-\exp(-c_*\eps^{-2\sigma_1})\|\ww\|_{\cX_\eps}^2\,.
\end{equation}
To bound the other term, the strategy is to compare $E_{\eps}^1[\ww_{\mathrm{in}}]$
with $E_0[\ww_{\mathrm{in}}]$. Using estimate \eqref{bd:WAI} in Proposition \ref{propIW}
and the fact that $\ww_{\mathrm{in}}$ is supported in region $\mrI_\eps$, we find that,
for any $\gamma_2 < 2$, there exists $C_2 > 0$ such that 
\begin{equation}\label{Eeps2bd}
  \bigl|\|\ww_{\mathrm{in}}\|_{\cX_\eps}^2 - \|\ww_{\mathrm{in}}\|_{\cX_0}^2\bigr|
  \,\le\, 2C_2\,\eps^{\gamma_2}\|\ww\|_{\cX_\eps}^2\,, \quad \text{hence}\quad
  E_{\eps}^1[\ww_{\mathrm{in}}] \,\ge\, E_0[\ww_{\mathrm{in}}] - C_2\,\eps^{\gamma_2}
  \|\ww\|_{\cX_\eps}^2.
\end{equation}
We next invoke Proposition~4.5 in \cite{GaSring}, which provides the lower bound
\begin{equation}\label{Eeps3bd}
  E_{0}[\ww_{\mathrm{in}}] \,\ge\, \kappa_0 \|\ww_{\mathrm{in}}\|_{\cX_0}^2 - C_3
  \bigl(\mathrm{M}[\ww_{\mathrm{in}}]^2 + \mathrm{m}_1[\ww_{\mathrm{in}}]^2 +
  \mathrm{m}_2[\ww_{\mathrm{in}}]^2\bigr)\,,
\end{equation}
for some constants $\kappa_0 \in (0,1/2)$ and $C_3 > 0$. At this point, it is
important to observe that, although the first moments of $\ww_{\mathrm{in}}$ do not vanish
exactly, they are extremely small. Indeed, since $\mathrm{M}[\ww_{\mathrm{in}}] =
-\mathrm{M}[\ww_{\mathrm{out}}]$ by assumption, we have $|\mathrm{M}[\ww_{\mathrm{in}}]|
\le \exp(-c_*\eps^{-2\sigma_1})\|\ww\|_{\cX_\eps}$, and a similar estimate holds for
$\mathrm{m}_1[\ww_{\mathrm{in}}]$ and $\mathrm{m}_2[\ww_{\mathrm{in}}]$.
It thus follows from \eqref{Eeps2bd}, \eqref{Eeps3bd} that
\begin{equation}\label{Eeps4bd}
  E_{0}[\ww_{\mathrm{in}}] \,\ge\, \kappa_0 \|\ww_{\mathrm{in}}\|_{\cX_\eps}^2 - C_4
  \,\eps^{\gamma_2}\|\ww\|_{\cX_\eps}^2\,.
\end{equation}

Finally, combining \eqref{Eepsinter}, {\eqref{Eeps1bd}}, \eqref{Eeps2bd} and \eqref{Eeps4bd}, we obtain
\[
  E_\eps[\ww] \,\ge\, \kappa_0\|\ww_{\mathrm{in}}\|_{\cX_\eps}^2  + \frac12 \|\ww_{\mathrm{out}}
  \|_{\cX_\eps}^2 - C_5\eps^{\gamma_2}\|\ww\|_{\cX_\eps}^2 \,\ge\, \bigl(\kappa_0 - C_5
  \,\eps^{\gamma_2}\bigr)\|\ww\|_{\cX_\eps}^2\,,
\]
which gives the desired estimate \eqref{bd:coercivity} if $\eps > 0$ is small enough. 
\end{proof}

\subsection{The energy identity}\label{sec:energyId}

We next compute the time evolution of the energy $E_{\eps}[\ww]$ defined in \eqref{def:Eeps},
assuming that $\ww$ is the solution of \eqref{linear}-\eqref{force} with zero initial data.

\begin{lemma}\label{lem:EnId}
Let $\ww$ be the solution of \eqref{linear}-\eqref{force} and $E_\eps[\ww]$ be defined as
in \eqref{def:Eeps}. Then
\begin{equation}\label{EnergyIdentity}
  t\partial_t E_{\eps}[\ww] + D_{\eps}[\ww] \,=\, \mathrm{A} + \mathrm{F} + \mathrm{NL}\,,
\end{equation}
where we define:
\begin{itemize}[leftmargin=15pt,itemsep=2pt]

\item The diffusion functional 
\begin{equation}\label{def:Deps}
  D_\eps[\ww] \,=\, -\frac12\jap{(t\partial_t W_\eps)\ww\,,\ww}_{L^2} -\jap{\cL\ww\,,W_\eps
  \ww+\varphi-\cT_\eps \varphi}_{L^2}\,;
\end{equation}

\item The advection terms
\begin{equation}\label{def:Adv}
\begin{split}
  \mathrm{A} \,=\, &\,\frac{1}{\delta} \jap{\left\{\Phi_\app^E\,,\ww\right\} + \left\{\varphi-\cT_\eps\varphi\,,
  \Omega_\app^E\right\},W_\eps \ww + \varphi-\cT_\eps \varphi}_{L^2}\\
  &+\jap{\left\{\Phi_\app^{NS}\,,\ww\right\} + \left\{\varphi-\cT_\eps\varphi\,,\Omega_\app^{NS}\right\},
  W_\eps \ww+\varphi-\cT_\eps \varphi}_{L^2}\,;
\end{split}  
\end{equation}

\item The forcing terms generated by the remainder and the vertical speed 
\begin{equation}\label{errF}
  \mathrm{F} \,=\, \frac{1}{\delta^2}\Big\langle \cR_M+\frac{\eps\zeta}{2\pi}\,\{\xi_1,\Omega_{\app}\}\,,
  W_\eps \ww +\varphi-\cT_\eps \varphi\Big\rangle_{L^2}\,;
\end{equation}

\item The nonlinear terms
\begin{equation} \label{errNL}
  \mathrm{NL} \,=\, \frac{1}{4\eps} \langle w, \partial_1(\cT_\eps\varphi)\rangle_{L^2} +
  \Big\langle\left\{\varphi
  -\cT_\eps \varphi,\ww\right\}+\frac{\eps\zeta}{2\pi \delta}\,\{\xi_1,\ww\}\,,W_\eps \ww +\varphi
  -\cT_\eps\varphi\Big\rangle_{L^2}\,.
\end{equation}
\end{itemize}
\end{lemma}

\begin{proof}
A direct computation gives 
\begin{align*}
  t\partial_t E_{\eps}[\ww] \,=\, &\jap{t\partial_t \ww, W_\eps \ww}_{L^2}+ \frac12\jap{\ww,
  (t\partial_t W_\eps)\ww}_{L^2}\\
  &+\frac12\jap{t\partial_t \ww, \varphi-\cT_\eps \varphi}_{L^2}+\frac12\jap{\ww, t\partial_t
  (\varphi-\cT_\eps \varphi)}_{L^2}\,.
\end{align*}
The last term in the right-hand side can be handled as in Section~\ref{sec:Arnold}, but we have
to be more careful here because $\cT_\eps$ is now a time-dependent operator. From the definition
of $\cT_\eps$ in \eqref{def:Teps} and the fact that $t\partial_t \eps = \eps/2$, we obtain
\[
  t\partial_t \bigl(\cT_\eps \varphi\bigr)(\xi,t) \,=\, t\partial_t\bigl(\varphi(-\xi_1-\eps^{-1}(t),\xi_2,t)
  \bigr) \,=\, \cT_\eps \Bigl(t\partial_t \varphi+\frac{1}{2\eps}\,\partial_1  \varphi\Bigr)(\xi,t)\,.
\]
Observing that $\cT_\eps \partial_1\varphi = -\partial_1\cT_\eps\varphi$ and using the identities
\eqref{def:Arn2}, we thus find
\[
  \frac12\jap{\ww, t\partial_t (\varphi-\cT_\eps \varphi)}_{L^2} \,=\, \frac12\jap{t\partial_t\ww\,,
  \varphi-\cT_\eps \varphi}_{L^2}+\frac{1}{4\eps} \langle w, \partial_1(\cT_\eps\varphi)\rangle_{L^2},
\]
and it follows that
\begin{equation}\label{eq:enId}
  t\partial_t E_{\eps}[\ww] \,=\,\jap{t\partial_t \ww, (W_\eps \ww+\varphi-\cT_\eps \varphi)}_{L^2}
  +\frac12\jap{\ww,(t\partial_t W_\eps)\ww}_{L^2}+\frac{1}{4\eps} \langle w, \partial_1
  (\cT_\eps\varphi)\rangle_{L^2}\,.
\end{equation}
We now replace the time derivative $t\partial_t\ww$ in \eqref{eq:enId} by its expression
\eqref{linear}--\eqref{force}, and this generates exactly the quantities $D_\eps[\ww]$, $\mathrm{A}$,
$\mathrm{F}$ and $\mathrm{NL}$ defined in \eqref{def:Deps}, \eqref{def:Adv}, \eqref{errF} and
\eqref{errNL}. Note that we chose to include the last term in  \eqref{eq:enId} in the 
nonlinearity $\mathrm{NL}$, and the previous one in the diffusion functional $D_\eps[\ww]$.
\end{proof}

Our goal in what follows is to estimate each term in \eqref{def:Deps}-\eqref{errNL}. In
Section~\ref{sec:diffusion} we show that the diffusion functional $D_\eps[\ww]$ controls,
roughly speaking, the $H^1$ analogue of the weighted $L^2$ norm $\|\ww\|_{\cX_\eps}$.
The remaining terms in \eqref{EnergyIdentity} can be estimated using the properties of
the approximate solution $\Omega_\app$ and of the weight function $W_\eps$. The calculations
are performed in Sections~\ref{sec:A}--\ref{sec:NL}, and result in the following crucial
energy estimate, which is the core of the proof of Theorem~\ref{th:mainNL} and will
be established in Section~\ref{sec:continuity}. 

\begin{proposition}\label{prop:key}
Fix $\sigma \in [0,1)$, take $M \in \NN$ such that $M > (3+\sigma)/(1-\sigma)$, and let $\ww$ be
the solution of \eqref{linear0}--\eqref{eq:zeromoments} with initial data $\ww|_{t=0}=0$. Let $\sigma_1,
\sigma_2, \gamma$ be as in \eqref{def:parameters}, with $\sigma_1$ sufficiently small and
$\sigma_2$ sufficiently large. Then for any $\kappa_* > 0$ there exist positive constants
$s_* \in (0,1)$ and $C_*\ge 1$ such that, if $\delta > 0$ is sufficiently small, the quantities
\eqref{def:Adv}--\eqref{errNL} satisfy
\begin{align}
   \label{bd:Aprop} |\mathrm{A}| \,&\le\, \delta^{s_*} D_{\eps}[\ww]\,,\\
  \label{bd:Fprop} |\mathrm{F}| \,&\le\, C_*\bigl(\delta^{-4}\eps^{2(M+1)}+\eps^4\bigr) +
  \kappa_* D_\eps[\ww]\bigl(1+\sqrt{E_{\eps}[\ww]}\bigr)\,,\\
  \label{bd:NLprop} |\mathrm{NL}| \,&\le\, \delta^{s_*}D_\eps[\ww]+C_*(\sqrt{E_{\eps}[\ww]}
  +E_{\eps}[\ww])D_{\eps}[\ww]\,,
\end{align}
as long as $t \le T_\adv\delta^{-\sigma}$. As a consequence, the energy \eqref{def:Eeps}
satisfies the estimate
\begin{equation}\label{bd:key}
  E_{\eps}[\ww](t) + \frac{1}{2}\int_0^t \frac{D_{\eps}[\ww](\tau)}{\tau}\dd \tau \,\le\,
  C_*\bigl(\delta^{-4}\eps^{2(M+1)}+\eps^4\bigr)\,, \qquad t \in \bigl(0,T_\adv\delta^{-\sigma}\bigr)\,.
\end{equation}
\end{proposition}

In view of Propositions~\ref{prop:appsol} and \ref{prop:coercivity}, estimate \eqref{bd:key}
implies that the solution $w$ of \eqref{linear0}--\eqref{eq:zeromoments} does not become much larger
than the source term $\delta^{-2}\cR_M$, which is of size $\delta^{-2}\eps^{M+1} + \eps^2$. 

\begin{remark}\label{rem:scales}
The assumption that $t \le T_\adv\delta^{-\sigma}$ translates into an upper bound on the aspect ratio
$\eps = \sqrt{\nu t}/d$ in terms of the inverse Reynolds number $\delta = \nu/\Gamma$. Indeed,
if we define $s_0 > 0$ such that
\begin{equation}\label{def:sigma}
  s_0 \,=\, \frac{M-3}{M+1}-\sigma, \qquad \text{namely}\qquad 1 - \sigma \,=\, \frac{4}{M+1} + s_0\,,
\end{equation}
we then have
\begin{equation}\label{eq:tTadv}
  \eps \,\le\, \delta^{\frac{1-\sigma}{2}} \,=\, \delta^{\frac{2}{M+1}+\frac{s_0}{2}}\,, \qquad \text{or}
  \qquad \delta^{-2}\eps^{M+1} \,\le\, \delta^{\frac{M+1}{2}s_0}\,.
\end{equation}
The constant $s_*$ in \eqref{bd:Aprop}, \eqref{bd:NLprop} can be taken as a small multiple of $s_0$. 

\end{remark}

\subsection{The diffusion functional}\label{sec:diffusion}

As is clear from Proposition~\ref{prop:key}, our strategy is to control the various terms
in the right-hand side of the energy identity \eqref{EnergyIdentity} using the diffusion
functional $D_{\eps}[\ww]$ defined in \eqref{def:Deps}. We thus need an accurate lower 
bound on $D_{\eps}[\ww]$, which can be obtained for $\eps > 0$ small enough by exploiting
the coercivity properties of a limiting quadratic form that was already studied  in
\cite{GaSarnold}, see also \cite{GaSring} for a similar approach. To state our result, it is
convenient to introduce the continuous function $\rho_\eps : \RR^2 \to \RR_+$ defined by
\begin{equation}\label{def:rhoeps}
  \rho_\eps(\xi) \,=\, \begin{cases}
  |\xi|\,, \quad & \text{if}\quad |\xi| < \eps^{-\sigma_1}, \\
  \eps^{-\sigma_1}\,, \quad & \text{if}\quad \eps^{-\sigma_1}\le |\xi| \le \eps^{-\sigma_2}, \\
  |\xi|^\gamma\,, \quad & \text{if}\quad |\xi| > \eps^{-\sigma_2}.
  \end{cases}
\end{equation}
In agreement with \eqref{def:rhoeps}, we denote $\rho_0(\xi) = |\xi|$ in the limiting case $\eps = 0$. 

\begin{proposition}\label{prop:Diffusion}
If $\sigma_1$ is sufficiently small and $\sigma_2$ sufficiently large, there exists $\kappa_D>0$
such that the following holds for $\eps > 0$ small enough. If $\ww \in \cX_\eps$ satisfies
$\rho_\eps\ww \in \cX_\eps$, $\nabla\ww \in \cX_\eps^2$, and $\mathrm{M}[\ww] = \mathrm{m}_1[\ww] =
\mathrm{m}_2[\ww]=0$, then
\begin{equation}\label{bd:Deps}
  D_{\eps}[\ww] \,\ge\, \kappa_D \bigl(\norm{\nabla \ww}^2_{\cX_{\eps}} + \norm{\rho_\eps
    \ww}^2_{\cX_\eps} + \norm{ \ww}^2_{\cX_\eps}\bigr)\,.
\end{equation}
\end{proposition}

The proof of Proposition~\ref{prop:Diffusion} is rather lengthy and can be divided into
four main steps.

\medskip\noindent
{\it Step 1: Preliminaries}. We recall the definitions of $\cL$ in \eqref{def:cLdom},
$W_\eps$ in \eqref{def:Weps} and $D_{\eps}[\ww]$ in \eqref{def:Deps}. We first handle the term
involving $\cT_\eps$ in \eqref{def:Deps}. Since $\cL\ww = \Delta\ww+\mathrm{div}(\xi \ww)/2$,
a simple integration by parts yields
\[
  \jap{\cL\ww,\cT_\eps \varphi}_{L^2} \,=\, -\jap{\nabla \ww,\nabla (\cT_\eps\varphi)}_{L^2}
  -\frac12\jap{ \ww,\xi\cdot \nabla (\cT_\eps\varphi)}_{L^2}\,.
\]
To estimate the right-hand side, we can argue as in the proof of Lemma~\ref{lem:zeta2},
by splitting the integration domain into the region $\mrI_{\eps}$ and its complement
$\mrI_{\eps}^c$. Using the bound \eqref{bd:UTesp} in the inner region and the rapid
decay of the weight $W_\eps^{-1/2}$ in the complement, we easily obtain
\[
  |\jap{\cL\ww,\cT_\eps \varphi}_{L^2}| \,\lesssim\, \eps^3\|\ww\|_{\cX_\eps}\bigl(\|\ww\|_{\cX_\eps}
  + \|\nabla \ww\|_{\cX_\eps}\bigr) \,\lesssim\, \eps^3\bigl(\|\ww\|_{\cX_\eps}^2
  +\|\nabla \ww\|_{\cX_\eps}^2\bigr)\,.
\]
We deduce that there exists a constant $C_1 > 0$ such that
\begin{equation}\label{bd:Deps1}
  D_\eps[\ww] \,\ge\, \cD_{W_\eps}[\ww] + \cD_{\cL,\eps}[\ww] - C_1\eps^3
  \bigl(\|\ww\|_{\cX_\eps}^2+\|\nabla \ww\|_{\cX_\eps}^2\bigr)\,,
\end{equation}
where we denote
\begin{equation}\label{def:cDWL}
  \cD_{W_\eps}[\ww] \,:=\, -\frac12\jap{(t\partial_t W_\eps)\ww,\ww}_{L^2}\,, \qquad
  \cD_{\cL,\eps}[\ww] \,:=\, -\jap{\cL \ww, W_\eps \ww+\varphi}_{L^2}\,.
\end{equation}

Our next task is to estimate the quadratic form $\cD_{W_\eps}$ which involves the time
derivative of the weight function. From the definition of $W_\eps$ in \eqref{def:Weps} we get
\[
  \cD_{W_\eps}[\ww] \,=\, -\frac12 \int_{\mathrm{I}_\eps}t\partial_t (F'(\Omega_{\app}^E)) |\ww|^2\,\dd \xi
  -\frac12\int_{\mathrm{II}_\eps}\bigl(t\partial_t(\exp(\eps(t)^{-2\sigma_1}/4))\bigr) |\ww|^2\,\dd \xi\,.
\]
Since $t\partial_t \eps = \eps/2$ and $\rho_\eps=\eps^{-\sigma_1}$ in $\mathrm{II}_\eps$ by \eqref{bd:Ieps} 
and \eqref{def:rhoeps}, a direct calculation shows that
\begin{equation}
\label{bd:goodDWeps}
  -\frac12\int_{\mathrm{II}_\eps}\bigl(t\partial_t(\exp(\eps(t)^{-2\sigma_1}/4))\bigr)|\ww|^2\,\dd \xi
  \,=\, \frac{\sigma_1}{8}\int_{\mathrm{II}_\eps}\rho_\eps^2|\ww|^2W_\eps\,\dd \xi \,=\,
  \frac{\sigma_1}{8}\norm{\1_{\mathrm{II}_\eps}\rho_\eps\ww}^2_{\cX_\eps}\,.
\end{equation}
To bound the other term, we recall that $F = F_0 + \eps^2 F_2 + \dots + \eps^M F_M$, so that
\[
  t\partial_t\bigl(F'(\Omega_{\app}^E)\bigr) \,=\, \sum_{k=2}^M\frac{k}{2}\,\eps^{k} F_k'
  \bigl(\Omega_\app^E\bigr) + F''(\Omega_\app^E)\bigl(t\partial_t\Omega_\app^E\bigr)\,.
\]
According to \eqref{appOP} we have $\Omega_\app^E = \Omega_0 + \eps^2 \Omega_2^E + \dots +
\eps^M \Omega_M^E$, so that $t\partial_t\Omega_\app^E = \cO_\cZ(\eps^2)$ in the sense 
of Definition~\ref{def:topo}. Moreover, since $\Omega_0/2 \le \Omega_\app^E\le 2\Omega_0$
in $\mathrm{I}_\eps$ by \eqref{eq:Ompos}, we can proceed as in \eqref{bd:lambound} to
prove that, for some $N \in \NN$, 
\[
  \sum_{k=2}^M \bigl|F_k'\bigl(\Omega_\app^E(\xi)\bigr)\bigr|+
  \bigl|F''(\Omega_\app^E(\xi))\bigr|\,\Omega_0(\xi) \,\lesssim\, (1+|\xi|)^N
  \,\Omega_0(\xi)^{-1} \,, \qquad \xi \in \mathrm{I}_\eps\,.
\]
It follows that
\begin{equation}\label{bd:badDWeps}
  \int_{\mathrm{I}_\eps}\bigl|t\partial_t (F'(\Omega_{\app}^E))\bigr| |\ww|^2\,\dd \xi \,\lesssim\,
  \eps^2\int_{\mathrm{I}_\eps}(1+|\xi|)^N\e^{|\xi|^2/4}|w|^2\,\dd \xi \,\lesssim\,
  \eps^{2-\sigma_1(N+2)}\|\ww\|_{\cX_\eps}^2\,,
\end{equation}
for a possibly larger integer $N$. In the last inequality we used Proposition~\ref{propIW}
which implies that $W_\eps(\xi) \approx W_0(\xi) \approx \e^{|\xi|^2/4}(1+|\xi|)^{-2}$ in region
$\mathrm{I}_\eps$. Combining \eqref{bd:Deps1}, \eqref{bd:goodDWeps}, \eqref{bd:badDWeps}
and taking $\sigma_1$ sufficiently small, we arrive at
\begin{equation}\label{bd:Deps2}
  D_\eps[\ww] \,\ge\, \frac{\sigma_1}{8}\norm{\1_{\mathrm{II}_\eps}\rho_\eps\ww}^2_{\cX_\eps}
  +\cD_{\cL,\eps}[\ww]-C_2\eps\bigl(\|\ww\|_{\cX_\eps}^2+\|\nabla \ww\|_{\cX_\eps}^2\bigr)\,,
\end{equation}
for some constant $C_2 > 0$.

Finally we derive a more explicit expression of the diffusive quadratic form $\cD_{\cL,\eps}[\ww]$.
Using the definition \eqref{def:cDWL} and integrating by parts, we easily find
\begin{equation}\label{eq:Dclfirst}
  \cD_{\cL,\eps}[\ww] \,=\, \jap{\nabla\ww, \nabla(W_\eps \ww)+\nabla\varphi}_{L^2}+\frac12
  \jap{\ww, \xi \cdot \nabla(W_\eps \ww+\varphi)}_{L^2}\,.
\end{equation}
Similarly, since $\ww = \Delta \varphi$, we observe that $\langle\nabla \ww,\nabla
\varphi\rangle_{L^2}=-\|\ww\|_{L^2}^2$ and
\[
  \langle\ww, \xi\cdot \nabla \varphi\rangle_{L^2} \,=\, \langle\Delta\varphi, \xi\cdot
  \nabla \varphi\rangle_{L^2} \,=\, -\int_{\RR^2}\Bigl(|\nabla\varphi|^2 + \frac12 \,\xi\cdot\nabla
  \,|\nabla\varphi|^2\Bigr)\dd\xi \,=\, 0\,.
\]
On the other hand, since $\ww(\xi\cdot\nabla \ww)=\mathrm{div}(\xi \ww^2/2)-\ww^2$, it is
not difficult to verify that
\begin{equation}\label{eq:idD12}
\begin{split}
  \jap{\ww, \xi \cdot \nabla(W_\eps \ww)}_{L^2} \,&=\, \jap{\ww, (\xi\cdot \nabla W_\eps)\ww}_{L^2}
  +\jap{\ww, (\xi \cdot \nabla \ww) W_\eps}_{L^2}\\
  \,&=\, \frac12\jap{\ww, (\xi\cdot \nabla W_\eps)\ww}_{L^2}-\norm{\ww}_{\cX_\eps}^2\,.
\end{split}
\end{equation}
Thus we can write the quantity $\cD_{\cL,\eps}[\ww]$ in the equivalent form 
\begin{equation}\label{eq:Dclfinal}
  \cD_{\cL,\eps}[\ww] \,=\, \norm{\nabla \ww}_{\cX_\eps}^2 + \jap{\ww, \nabla \ww\cdot \nabla
  W_\eps}_{L^2}+\frac14 \jap{\ww,\ww(\xi\cdot \nabla W_\eps)}_{L^2} -\norm{\ww}^2_{L^2}
  -\frac12 \norm{\ww}_{\cX_\eps}^2\,.
\end{equation}

\medskip\noindent
{\it Step 2: Decomposition of the diffusive quadratic form}. The expression \eqref{eq:Dclfinal}
is the analogue of the quadratic form $Q_\eps[\ww]$ appearing in \cite[Section~4.7]{GaSring}.
Taking formally the limit $\eps \to 0$, we obtain 
\begin{equation}\label{eq:Dcl0final}
  \cD_{\cL,0}[\ww] \,=\, \norm{\nabla \ww}_{\cX_0}^2 + \jap{\ww,\nabla \ww\cdot
    \nabla W_0}_{L^2}+\frac14 \jap{\ww,\ww(\xi\cdot \nabla W_0)}_{L^2} -\norm{\ww}^2_{L^2}
  -\frac12 \norm{\ww}_{\cX_0}^2\,,
\end{equation}
which is exactly the quadratic form $Q_0[\ww]$ studied \cite{GaSarnold,GaSring}.
In particular, it is established in \cite[Proposition~4.14]{GaSring} that, for any
$\ww \in \cX_0$ with $\rho_0\ww \in \cX_0$ and $\nabla \ww \in \cX_0^2$, the following
holds
\begin{equation}\label{bd:lwD0}
  \cD_{\cL,0}[\ww] \,\ge\, \kappa_2 \bigl(\norm{\nabla \ww}^2_{\cX_0} + \norm{ \rho_0\ww}^2_{\cX_0}
  +\norm{\ww}^2_{\cX_0}\bigr) - C_3\bigl(\mathrm{M}[\ww]^2 + \mathrm{m}_1[\ww]^2 +
  \mathrm{m}_2[\ww]^2\bigr)\,,
\end{equation}
for some constants $\kappa_2\in (0,1/2)$ and $C_3>0$. 

To obtain a similar lower bound on $\cD_{\cL,\eps}[\ww]$ for $\eps>0$, the idea is
to decompose the vorticity $\ww$ so as to single out the contribution of the inner
region. As opposed to what we have done in the proof of Proposition~\ref{prop:coercivity},
we use here a smooth cut-off since our quadratic form involves first-order derivatives.
Let $\chi_1,\chi_2 : \RR^2 \to [0,1]$ be smooth functions satisfying $\chi_1^2+\chi_2^2=1$,
and such that $\chi_1(\xi) = 1$ when $|\xi|\le \eps^{-\sigma_1}/2$ and $\chi_1(\xi) = 0$
when $|\xi|\ge \eps^{-\sigma_1}$. In addition, we assume that $|\nabla \chi_i| \lesssim
\eps^{\sigma_1}$. Then, with a slight abuse in notation, we define
\begin{equation}\label{def:partition}
  \ww_{\mathrm{in}} \,:=\, \chi_1\ww\,, \qquad \ww_{\mathrm{out}} \,:=\, \chi_2\ww\,, \qquad
  \text{so that}\quad \ww^2 \,=\, \ww_{\mathrm{in}}^2+\ww_{\mathrm{out}}^2\,.
\end{equation}
We obviously have $\nabla (\chi_i \ww)=(\nabla\chi_i)\ww+\chi_i\nabla\ww$ and $\chi_1\nabla \chi_1
+\chi_2\nabla\chi_2=0$, hence
\begin{equation}\label{def:partition0}
  |\nabla \ww|^2 \,=\, |\nabla \ww_{\mathrm{in}}|^2 + |\nabla \ww_{\mathrm{out}}|^2
  -(|\nabla \chi_1|^2+|\nabla \chi_2|^2)\ww^2\,.
\end{equation}
Since
\[
  \jap{\nabla \ww\cdot \nabla W_\eps, \ww }_{L^2} \,=\, \frac12\int_{\RR^2}\nabla(\ww^2)\cdot
  \nabla W_\eps\,\dd \xi\,,
\]
and since the remaining terms in $\cD_{\cL,\eps}$ do not involve derivatives of $\ww$,
it follows from \eqref{def:partition}, \eqref{def:partition0} that
\begin{equation}\label{Dclepsdecomp}
\begin{split}
  \cD_{\cL,\eps}[\ww] \,&=\, \cD_{\cL,\eps}[\ww_{\mathrm{in}}] + \cD_{\cL,\eps}[\ww_{\mathrm{out}}]
  -\norm{\sqrt{|\nabla \chi_1|^2+|\nabla \chi_2|^2}\,\ww}_{\cX_\eps}^2\,, \\
  \,&\ge\, \cD_{\cL,\eps}[\ww_{\mathrm{in}}] + \cD_{\cL,\eps}[\ww_{\mathrm{out}}]
  - C_4 \eps^{2\sigma_1}\|\ww\|_{\cX_\eps}^2\,,
\end{split}
\end{equation}
where in the last inequality we use the assumption that $|\nabla \chi_i| \lesssim\eps^{\sigma_1}$.

\medskip\noindent
{\it Step 3: Contribution of the inner region}. We now decompose
\[
  \cD_{\cL,\eps}[\ww_{\mathrm{in}}] \,=\, \cD_{\cL,0}[\ww_{\mathrm{in}}]
  + \cD_{\mathrm{err}}[\ww_{\mathrm{in}}]\,,
\]
where
\begin{align*}
  \cD_{\mathrm{err}}[\ww_{\mathrm{in}}] \,&:=\, \bigl(\norm{\nabla \ww_{\mathrm{in}}}_{\cX_\eps}^2
  -\norm{\nabla \ww_{\mathrm{in}}}_{\cX_0}^2\bigr) + \jap{\ww_{\mathrm{in}},\nabla \ww_{\mathrm{in}}
  \cdot \nabla(W_\eps-W_0)}_{L^2}\\
  &\quad\, +\frac14 \jap{\ww_{\mathrm{in}},\ww_{\mathrm{in}}(\xi\cdot \nabla (W_\eps-W_0)}_{L^2}
  -\frac12 \bigl(\norm{\ww_{\mathrm{in}}}_{\cX_\eps}^2-\norm{\ww_{\mathrm{in}}}_{\cX_0}^2\bigr)\,.
\end{align*}
Since $\ww_{\mathrm{in}}$ is supported in the region $\mathrm{I}_{\eps}$, the bounds \eqref{bd:WAI}
in Proposition~\ref{propIW} readily imply that
\[
  \bigl|\cD_{\mathrm{err}}[\ww_{\mathrm{in}}]\bigr| \,\lesssim\, \eps^{\gamma_2}\bigl
  (\norm{\nabla \ww_{\mathrm{in}}}_{\cX_0}^2+\norm{\rho_0\ww_{\mathrm{in}}}_{\cX_0}^2
  +\norm{ \ww_{\mathrm{in}}}_{\cX_0}^2\bigr)\,,
\]
for some $\gamma_2 \in (1,2)$. To estimate $\cD_{\cL,0}[\ww_{\mathrm{in}}]$, we use the lower
bound \eqref{bd:lwD0}. As was observed in the proof of Proposition~\ref{prop:coercivity},
our assumptions on $\ww$ imply that the moments of $\ww_{\mathrm{in}}$ are extremely small,
namely
\[
  \mathrm{M}[\ww_{\mathrm{in}}]^2 + \mathrm{m}_1[\ww_{\mathrm{in}}]^2 +
  \mathrm{m}_2[\ww_{\mathrm{in}}]^2 \,\le\, \exp(-c_*\eps^{-2\sigma_1})\|\ww\|_{\cX_\eps}^2
  \,\le\, \epsilon \|\ww\|_{\cX_\eps}^2\,.
\]
Altogether, if $\eps > 0$ small enough, we obtain
\begin{align}\label{bd:lowomin}
  \cD_{\cL,\eps}[\ww_{\mathrm{in}}] \,&\ge\, \kappa_2\bigl(\norm{\nabla \ww_{\mathrm{in}}}_{\cX_0}^2
  + \norm{\rho_0\ww_{\mathrm{in}}}_{\cX_0}^2 +\norm{ \ww_{\mathrm{in}}}_{\cX_0}^2\bigr)
  - C_3\epsilon \|\ww\|_{\cX_\eps}^2 - \bigl|\cD_{\mathrm{err}}[\ww_{\mathrm{in}}]\bigr| \\ \nonumber
  \,&\ge\, \frac{\kappa_2}{2}\bigl(\norm{\nabla\ww_{\mathrm{in}}}_{\cX_\eps}^2 +
  \norm{\rho_\eps\ww_{\mathrm{in}}}_{\cX_\eps}^2 +\norm{ \ww_{\mathrm{in}}}_{\cX_\eps}^2\bigr)
  - C_3\epsilon \|\ww\|_{\cX_\eps}^2\,.
\end{align}
In the second line, we used Proposition~\ref{propIW} again to compare the norms of $\cX_0$ and $\cX_\eps$. 

\medskip\noindent
{\it Step 4: Contribution of the outer region}. It remains to estimate the term
$\cD_{\cL,\eps}[\ww_{\mathrm{out}}]$ in \eqref{Dclepsdecomp}, which is more complicated
because $\ww_{\mathrm{out}}$ is nonzero in all three regions $\mathrm{I}_{\eps}$,
$\mathrm{II}_{\eps}$, and $\mathrm{III}_{\eps}$. We recall, however, that $\ww_{\mathrm{out}}$
is supported in the domain where $|\xi|\ge\eps^{-\sigma_1}/2$, which implies 
\begin{equation}\label{eq:wweps}
  \norm{\ww_{\mathrm{out}}}_{L^2} \,\le\, \norm{\ww_{\mathrm{out}}}_{\cX_\eps}
  \,\lesssim\, \eps^{\sigma_1}\norm{\rho_{\eps}\ww_{\mathrm{out}}}_{\cX_\eps}
  \,\le\, \eps^{\sigma_1}\norm{\rho_{\eps}\ww}_{\cX_\eps}\,.
\end{equation}
On the other hand, using Young's inequality, we find 
\[
  \bigl|\jap{\ww_{\mathrm{out}},\nabla \ww_{\mathrm{out}}\cdot \nabla W_{\eps}}_{L^2}\bigr|
  \,\leq\, \frac38\,\bigl\|\bigl(W_\eps^{-1} \nabla W_\eps\bigr)\ww_{\mathrm{out}}
  \bigr\|_{\cX_\eps}^2 + \frac23\,\|\nabla \ww_{\mathrm{out}}\|_{\cX_\eps}^2\,.
\]
In view of \eqref{eq:Dclfinal} we thus have, for some constant $C_5>0$,
\begin{equation}\label{bd:wout1}
  \cD_{\cL,\eps}[\ww_{\mathrm{out}}] \,\ge\, \frac13 \norm{\nabla \ww_{\mathrm{out}}}^2_{\cX_\eps}
  + \cI_\eps[\ww_{\mathrm{out}}] -C_5\eps^{2\sigma_1}\norm{\rho_{\eps}\ww}^2_{\cX_\eps}\,,
\end{equation}
where
\[
  \cI_\eps[\ww_{\mathrm{out}}] \,:=\, \frac14 \jap{\ww_\mathrm{out}^2, \xi\cdot \nabla W_\eps}_{L^2}
  - \frac38\,\bigl\|\bigl(W_\eps^{-1} \nabla W_\eps\bigr)\ww_{\mathrm{out}}
  \bigr\|_{\cX_\eps}^2\,.
\]

Our goal is to find a lower bound on $\cI_\eps[\ww_{\mathrm{out}}]$ in regions $\mathrm{I}_{\eps}$
and $\mathrm{III}_{\eps}$, keeping in mind that $\nabla W_\eps = 0$ in $\mathrm{II}_{\eps}$.
In the outer region $\mathrm{III}_{\eps}$ we have $W_\eps = \e^{|\xi|^{2\gamma}/4}$ with
$\gamma = \sigma_1/\sigma_2 < 1/2$, so that
\begin{equation}\label{eq:gradWeps}
  \xi\cdot\nabla W_\eps \,=\, \frac{\gamma}{2}\,|\xi|^{2\gamma}\,W_\eps \,=\,
  \frac{\gamma}{2}\,\rho_\eps^2\,W_\eps\,, \qquad |\nabla W_\eps| \,\le\, \frac{\gamma}{2}\,|\xi|^{\gamma-1} 
  \rho_\eps W_\eps \,\le\ \frac{\gamma}{2}\,\eps^{\sigma_2/2}\rho_\eps W_\eps\,.
\end{equation}
It follows that
\begin{equation}\label{bd:IepsIII}
  \cI_\eps\bigl[\1_{\mathrm{III}_\eps}\ww_{\mathrm{out}}\bigr]  \,=\, 
  \frac{\gamma}{8}\norm{\1_{\mathrm{III}_\eps}\rho_\eps \ww}^2_{\cX_\eps} + 
  \cO\bigl(\eps^{\sigma_2}\norm{\rho_\eps \ww}_{\cX_\eps}^2\bigr)\,.
\end{equation}
In the inner region $\mathrm{I}_\eps$, we can use Proposition~\ref{propIW}
to compare the weight $W_\eps$ with the function $W_0 = A$ defined in \eqref{def:A}. 
This gives the estimates
\begin{equation}\label{eq:comp1}
  \bigl| \langle \1_{\mathrm{I}_\eps}\ww_\mathrm{out}^2, \xi\cdot \nabla W_\eps\rangle_{L^2}
  - \langle\1_{\mathrm{I}_\eps}\ww_\mathrm{out}^2, \xi\cdot \nabla W_0\rangle_{L^2}\bigr|
  \,\lesssim\, \eps^{\gamma_2+\sigma_1}\norm{\rho_\eps \ww}_{\cX_\eps}^2\,,
\end{equation}
and
\begin{equation}\label{eq:comp2}
  \Bigl|\bigl\|\1_{\mathrm{I}_\eps}\bigl(W_\eps^{-1}\nabla W_\eps\bigr)\ww_{\mathrm{out}}\bigr\|_{\cX_\eps}^2
  - \bigl\|\1_{\mathrm{I}_\eps}\bigl(W_0^{-1}\nabla W_0\bigr)\ww_{\mathrm{out}}
  \bigr\|_{\cX_0}^2\Bigr| \,\lesssim\, \eps^{\gamma_2+2\sigma_1}\norm{\rho_\eps\ww}_{\cX_\eps}^2\,.
\end{equation}
Moreover, using the explicit expression \eqref{def:A}, we easily find
\[
  \xi \cdot\nabla W_0(\xi) \,=\, \frac{|\xi|^2}{2}\,W_0(\xi) - 2W_0(\xi) + 2\,, \qquad
  \xi \in \RR^2\,.
\]
In view of \eqref{eq:wweps}, we deduce that
\begin{equation}\label{eq:comp3}
\begin{split}
  \frac14 \langle \1_{\mathrm{I}_\eps}\ww_\mathrm{out}^2, \xi\cdot \nabla W_0\rangle_{L^2}
  \,&=\, \frac18\,\bigl\|\1_{\mrI_\eps}\rho_0\ww_{\mathrm{out}}\bigr\|_{\cX_0}^2 +
  \cO\bigl(\eps^{2\sigma_1}\norm{\rho_\eps\ww}_{\cX_\eps}^2\bigr)\,, \\
  \frac38\,\bigl\|\1_{\mathrm{I}_\eps}\ww_{\mathrm{out}}\bigl(W_0^{-1} \nabla W_0\bigr)
  \bigr\|_{\cX_0}^2 \,&=\, \frac{3}{32}\,\bigl\|\1_{\mrI_\eps}\rho_0\ww_{\mathrm{out}}
  \bigr\|_{\cX_0}^2 + \cO\bigl(\eps^{2\sigma_1}\norm{\rho_\eps\ww}_{\cX_\eps}^2\bigr)\,.
\end{split}
\end{equation}
Combining \eqref{eq:comp1}, \eqref{eq:comp2}, \eqref{eq:comp3}, and using \eqref{bd:WAI} again,
we obtain
\begin{equation}\label{bd:IepsI}
\begin{split}
  \cI_\eps\bigl[\1_{\mathrm{I}_\eps}\ww_{\mathrm{out}}\bigr]  \,&\ge\,
  \frac{1}{32}\,\bigl\|\1_{\mrI_\eps}\rho_0\ww_{\mathrm{out}}\bigr\|_{\cX_0}^2
  - C\bigl(\eps^{2\sigma_1} + \eps^{\gamma_2+\sigma_1}\bigr)\norm{\rho_\eps \ww}_{\cX_\eps}^2 \\
  \,&\ge\, \frac{1}{32}\,\bigl\|\1_{\mrI_\eps}\rho_\eps\ww_{\mathrm{out}}\bigr\|_{\cX_\eps}^2
  - C_6\bigl(\eps^{2\sigma_1} + \eps^{\gamma_2}\bigr)\norm{\rho_\eps \ww}_{\cX_\eps}^2\,,
\end{split}
\end{equation}
for some $C_6 > 0$. Finally, in view of \eqref{bd:wout1}, \eqref{bd:IepsIII}, \eqref{bd:IepsI},
we arrive at
\begin{equation}\label{bd:lowomout}
  \cD_{\cL,\eps}[\ww_{\mathrm{out}}]\geq \frac13\norm{\nabla \ww_{\mathrm{out}}}^2_{\cX_\eps}
  + \frac{1}{32}\,\bigl\|\1_{\mrI_\eps}\rho_\eps\ww_{\mathrm{out}}\bigr\|_{\cX_\eps}^2
  +\frac{\gamma}{8}\norm{\1_{\mathrm{III}_\eps}\rho_\eps \ww}^2_{\cX_\eps}
  -C_7\eps^{2\sigma_1} \norm{\rho_\eps\ww}_{\cX_\eps}^2\,,
\end{equation}
for some $C_7 > 0$. Here we used the fact that $2\sigma_1 < 1 < \min(\gamma_2,\sigma_2)$. 

\smallskip
It is now a simple task to conclude the proof of Proposition~\ref{prop:Diffusion}.
If we combine \eqref{bd:Deps2}, \eqref{Dclepsdecomp}, \eqref{bd:lowomin},
\eqref{bd:lowomout}, we see that there exist $\kappa_3 > 0$ and $C_8>0$ such that
\begin{align*}  
  \cD_{\cL,\eps}[\ww] \,&\ge\,
  \kappa_3 \bigl(\norm{\nabla\ww_{\mathrm{in}}}_{\cX_\eps}^2 + \norm{\nabla\ww_{\mathrm{out}}}_{\cX_\eps}^2 +
  \norm{\rho_\eps\ww}_{\cX_\eps}^2 +\norm{\ww}_{\cX_\eps}^2\bigr) \\
  &\quad -C_8\,\eps^{2\sigma_1} \bigl(\norm{\nabla\ww}_{\cX_\eps}^2 + \norm{\rho_\eps\ww}_{\cX_\eps}^2
  + \norm{\ww}_{\cX_\eps}^2\bigr)\,. 
\end{align*}
Since $|\nabla \ww_{\mathrm{in}}|^2 + |\nabla \ww_{\mathrm{out}}|^2 \ge |\nabla \ww|^2$ by 
\eqref{def:partition0}, we arrive at \eqref{bd:Deps} by taking $\eps > 0$ small enough. \qed

\subsection{The advection terms}\label{sec:A}

In this section we estimate the advection terms defined in \eqref{def:Adv}, which are
potentially dangerous because of the large prefactor $\delta^{-1}$. In the inner region
$\mathrm{I}_\eps$, we exploit crucial cancellations related to the structure of the energy
functional \eqref{def:Eeps}, which were explained in an informal way in Section~\ref{sec:Arnold}.
In the other regions the dominant term in the energy is the weighted enstrophy 
$\frac12\|\ww\|_{\cX_\eps}^2$, and the influence of the large advection terms can be 
controlled by an appropriate choice of the parameters \eqref{def:parameters}.

Starting from \eqref{def:Adv} we decompose $\mathrm{A} = \delta^{-1}\bigl(\mathrm{A}_1
+ \mathrm{A}_2+  \mathrm{A}_3\bigr) +\mathrm{A}_{NS}$, where
\begin{align}
  \label{def:A1} \mathrm{A}_1 \,&=\, \jap{\left\{\Phi_\app^E,\ww\right\},W_\eps \ww}_{L^2}\,, \\
  \label{def:A2} \mathrm{A}_2 \,&=\, \jap{\left\{\varphi-\cT_\eps\varphi,\Omega_\app^E\right\},
  W_\eps \ww}_{L^2} + \jap{\1_{\mrI_\eps}\left\{\Phi_\app^E, \ww\right\},\varphi-\cT_\eps \varphi}_{L^2}\\
  \label{def:A3} \mathrm{A}_3 \,&=\, \jap{\1_{\mathrm{II}_\eps\cup\mathrm{III}_\eps}\left\{\Phi_\app^E,
  \ww\right\},\varphi-\cT_\eps \varphi}_{L^2},\\
  \label{def:ANS}\mathrm{A}_{NS} \,&=\, \jap{\left\{\Phi_\app^{NS},\ww\right\}+\left\{\varphi-\cT_\eps
  \varphi,\Omega_\app^{NS}\right\},W_\eps \ww+\varphi-\cT_\eps \varphi}_{L^2}\,.
\end{align}
Here we use the fact that $\big\langle\{\varphi-\cT_\eps\varphi,\Omega_\app^E\},\varphi-\cT_\eps
\varphi\big\rangle_{L^2} = 0$. This is a consequence of standard identities such as
\begin{equation}\label{eq:basic}
  \jap{\{f,g\},h}_{L^2} \,=\, -\jap{\{f,h\},g}_{L^2}\,, \qquad \jap{\{f,g\},hg} \,=\,-\frac12\jap{\{f,h\}g,g}\,,
\end{equation}
which are repeatedly used in the sequel. Our goal is to prove the following set of
estimates. 

\begin{lemma}\label{lem:Abdd}
There exists an integer $N > 0$ depending only on $M$ such that the quantities 
\eqref{def:A1}--\eqref{def:ANS} satisfy 
\begin{align}
  \label{bd:A1}  \delta^{-1}|\mathrm{A}_1|\,&\lesssim\,  \delta^{-1}\eps^{M+1-\sigma_1 N}
  \|\ww\|_{\cX_\eps}^2+\delta^{-1}\eps^{1+\sigma_2}\|\rho_\eps \ww\|_{\cX_\eps}^2\,, \\
  \label{bd:A2} \delta^{-1}|\mathrm{A}_2|\,&\lesssim\, \delta^{-1}\eps^{M+1}\bigl(
  \|\ww_\eps\|_{\cX_\eps}^2+\|\nabla \ww\|_{\cX_\eps}^2\bigr)\,, \\
  \label{bd:A3} \delta^{-1}|\mathrm{A}_3|\,&\lesssim\, \delta^{-1}\exp(-c_*\eps^{-2\sigma_1})
  \bigl(\|\ww\|_{\cX_\eps}^2+\|\nabla \ww\|_{\cX_\eps}^2\bigr)\,,\\
  \label{bd:ANS} |\mathrm{A}_{NS}|\,&\lesssim\, \eps^2\bigl(\|\ww\|_{\cX_\eps}^2+\|\nabla \ww\|_{\cX_\eps}^2
  \bigr)\,,
\end{align}
for some $c_*>0$ sufficiently small. Consequently, under the assumptions of Proposition~\ref{prop:key}, 
there exists $s_* \in (0,1)$ such that estimate \eqref{bd:Aprop} holds. 
\end{lemma}

\begin{proof}
We start with the bound for $\mathrm{A}_1$. Since $W_\eps = F'(\Omega_\app^E)$ in region $\mathrm{I}_\eps$
and $\nabla W_\eps = 0$ in region $\mrI\mrI_\eps$, we find using \eqref{eq:basic}\:
\begin{align*}
  \mathrm{A}_1 \,&=\, \jap{\left\{\Phi_\app^E,\ww\right\},W_\eps \ww}_{L^2} \,=\,
  -\frac12\jap{\left\{\Phi_\app^E,W_\eps\right\}\ww, \ww}_{L^2}\\ \notag
  \,&=\, -\frac12\jap{\1_{\mrI_\eps}\left\{\Phi_\app^E,F'(\Omega_\app^E)\right\}\ww,\ww}_{L^2}
  -\frac12\jap{\1_{\mrI\mrI\mrI_\eps}\left\{\Phi_\app^E,W_\eps\right\}\ww, \ww}_{L^2}\,.
\end{align*}
The first term in the right-hand side is clearly unchanged if we replace $\Phi_\app^E$
by the quantity $\Theta =\Phi_\app^E + F(\Omega_\app^E)$ which is introduced in
\eqref{def:Theta}. We thus obtain
\begin{equation}\label{A1first}
   \mathrm{A}_1 \,=\, -\frac12\jap{\1_{\mrI_\eps}\left\{\Theta,W_\eps\right\}\ww, \ww}_{L^2}
  -\frac12\jap{\1_{\mrI\mrI\mrI_\eps}\left\{\Phi_\app^E,W_\eps\right\}\ww, \ww}_{L^2}\,.
\end{equation}
To control the first term in \eqref{A1first}, we use the bound \eqref{bd:Theta} on $\nabla \Theta$
and the estimate \eqref{bd:WAI} on $\nabla W_\eps$ in region $\mrI_\eps$, which give
$|\{\Theta,W_\eps\}| \lesssim \eps^{M+1}(1+|\xi|)^NW_\eps$ for some $N \in \NN$. Since
$|\xi| \le 2\eps^{-\sigma_1}$ in that region, we infer that
\begin{equation}\label{bd:A11}
  \bigl|\jap{\1_{\mrI_\eps}\left\{\Theta,W_\eps\right\}\ww, \ww}_{L^2}\bigr|
  \,\lesssim\, \eps^{M+1-\sigma_1 N}\|\ww\|_{\cX_\eps}^2\,.
\end{equation} 
To bound the second term in \eqref{A1first}, we recall that $W_\eps = \exp(|\xi|^{2\gamma}/4)$
in region $\mathrm{III}_\eps$, so that
\[
  \bigl|\jap{\1_{\mrI\mrI\mrI_\eps}\left\{\Phi_\app^E,W_\eps\right\}\ww, \ww}_{L^2}\bigr| \,\lesssim\,
  \int_{\{|\xi|>\eps^{-\sigma_2}\}}|\nabla \Phi_{\app}^E|\,|\xi|^{2\gamma-1}\,W_\eps|\ww|^2\,\dd \xi\,.
\]
In view of \eqref{def:PhiappE} we have $|\nabla\Phi_{\app}^E| \le |\nabla \Psi_{\app}^E| +
|\nabla\cT_\eps\Psi_{\app}^E| + C\eps$, and applying estimate \eqref{bd:velocityLinfty} with
$\ww = \Omega_\app^E$ we obtain $|\nabla \Psi_{\app}^E(\xi)| \le C(1+|\xi|)^{-1}$ and
$|\nabla \cT_\eps\Psi_{\app}^E(\xi)| \le C(1+|\xi+\eps^{-1}e_1|)^{-1}$. It follows that
$\1_{\{|\xi|>\eps^{-\sigma_2}\}}|\nabla \Phi_\app^E(\xi)| \lesssim \eps+\eps^{\sigma_2}
\lesssim \eps$. Since $\rho_\eps(\xi) = |\xi|^\gamma$ in region $\mathrm{III}_\eps$, we conclude that
\begin{equation}\label{bd:A12}
  |\jap{\1_{\mrI\mrI\mrI_\eps}\left\{\Phi_\app^E,W_\eps\right\}\ww, \ww}_{L^2}|
  \,\lesssim\, \eps^{1+\sigma_2} \|\rho_\eps\ww\|_{\cX_\eps}^2\,,
\end{equation}
and estimate \eqref{bd:A1} follows directly from \eqref{bd:A11}, \eqref{bd:A12}. 

\smallskip
We next consider the term $\mathrm{A}_2$ defined by \eqref{def:A2}. Using \eqref{eq:basic}
and the Leibniz rule, we find 
\[
  \jap{\left\{\varphi-\cT_\eps\varphi,\Omega_\app^E\right\},W_\eps \ww}_{L^2} \,=\,
  \jap{\left\{\Omega_\app^E,\ww\right\}W_\eps,\varphi-\cT_\eps \varphi}_{L^2} +
  \jap{\left\{\Omega_\app^E,W_\eps\right\}\ww,\varphi-\cT_\eps \varphi}_{L^2}\,.
\]
Only the region $\mathrm{III}_\eps$ contributes to the last term, since $W_\eps = F'(\Omega_\app^E)$
in $\mathrm{I}_\eps$ and $\nabla W_\eps = 0$ in $\mrI\mrI_\eps$. As for the first term,
we observe that
\[
  \1_{\mrI_\eps}\bigl\{\Omega_\app^E,\ww\bigr\}W_\eps \,=\, \1_{\mrI_\eps}\bigl\{F(\Omega_\app^E),
  \ww\bigr\} \,=\, \1_{\mrI_\eps}\bigl\{\Theta,\ww\bigr\} - \1_{\mrI_\eps}\bigl\{\Phi_\app^E,\ww\bigr\}\,,
\]
where $\Theta =\Phi_\app^E + F(\Omega_\app^E)$. We thus obtain the following alternative expression
\begin{equation}\label{A2aux}
  \mathrm{A}_2 \,=\, \jap{\1_{\mrI_\eps}\left\{\Theta, \ww\right\} + \1_{\mrI\mrI_\eps\cup\mathrm{III}_\eps}
  \{\Omega_{\app}^E,\ww\}W_\eps + \1_{\mrI\mrI\mrI_\eps}\left\{\Omega_\app^E,W_\eps\right\}\ww,\varphi
  -\cT_\eps \varphi}_{L^2}\,.
\end{equation}
By H\"older's inequality we have 
\[
  |\langle \1_{\mrI_\eps}\{\Theta,\,\ww\},\varphi-\cT_\eps \varphi\rangle_{L^2}| 
  \,\le\, \|\1_{\mrI_\eps} W_\eps^{-1/2}\nabla \Theta\|_{L^2}\,
  \|\varphi-\cT_\eps \varphi\|_{L^\infty}\,\|\nabla \ww\|_{\cX_\eps}\,.
\]
Using the bound \eqref{bd:Theta} on $\nabla\Theta$ and the rapid decay of the function $W_\eps^{-1/2}$,
we see that the first factor in the right-hand side is of order $\eps^{M+1}$. Moreover, we
know from Lemma~\ref{lem:withmom} that $\|\varphi-\cT_\eps \varphi\|_{L^\infty} \le 2
\|\varphi\|_{L^\infty} \lesssim \|\ww\|_{\cX_\eps}$. This gives
\begin{align}\label{bd:A21}
  |\jap{\1_{\mrI_\eps}\left\{\Theta, \ww\right\},\varphi-\cT_\eps \varphi}_{L^2}| \,\lesssim\,
  \eps^{M+1} \|\ww\|_{\cX_\eps}\|\nabla \ww\|_{\cX_\eps}\,.
\end{align}
To treat the remaining terms in \eqref{A2aux}, we use the estimate $|\nabla \Omega_\app^E|
(W_\eps+|\nabla W_\eps|)\lesssim 1$, which follows from the definition \eqref{def:Weps}
of $W_\eps$ and from the properties of the approximate solution $\Omega_\app^E \in \cZ$.
Proceeding as above we find 
\begin{equation}\label{bd:A22}
\begin{split}
  &|\jap{\1_{\mrI\mrI_\eps\cup\mathrm{III}_\eps}\{\ww,\Omega_{\app}^E\}W_\eps
  +\1_{\mrI\mrI\mrI_\eps}\left\{W_\eps ,\Omega_\app^E\right\}\ww,\varphi-\cT_\eps \varphi}_{L^2}|\\
  &\qquad\qquad\qquad \,\lesssim\, \exp(-c_*\eps^{-2\sigma_1})\|\ww\|_{\cX_\eps}(\|\ww\|_{\cX_\eps}
  +\|\nabla\ww\|_{\cX_\eps})\,,
\end{split}
\end{equation}
for some $c_*>0$ sufficiently small, and \eqref{bd:A2} is a direct consequence of \eqref{bd:A21},
\eqref{bd:A22}. We estimate $\mathrm{A}_3$ in a similar way:
\[
  |\mathrm{A}_3| \,\le\, \|\1_{\mathrm{II}_\eps\cup\mathrm{III}_\eps}W_\eps^{-1/2}
  \nabla \Phi_\app^E\|_{L^2}\,\|\ww\|_{\cX_\eps}\|\nabla \ww\|_{\cX_\eps}\,\lesssim\,
  \exp(-c_*\eps^{-2\sigma_1})\|\ww\|_{\cX_\eps}\|\nabla \ww\|_{\cX_\eps}\,,
\]
where in the last inequality we used the fact that $\Phi_\app^E\in \cS_*$. This proves \eqref{bd:A3}.

\smallskip
We now consider the viscous term $\mathrm{A}_{NS}$. Since $\Omega_\app^{NS} = \mathcal{O}_{\cZ}(\eps^2)$
we find using \eqref{bd:velocityLq}
\begin{align}\label{bd:ANS1}
  |\jap{\left\{\varphi-\cT_\eps\varphi,\Omega_\app^{NS}\right\},W_\eps \ww+\varphi-\cT_\eps
  \varphi}_{L^2}| \,=\, |\jap{\left\{\varphi-\cT_\eps\varphi,\Omega_\app^{NS}\right\},W_\eps
  \ww}_{L^2}| \,\lesssim\, \eps^2\|\ww \|_{\cX_\eps}^2\,.
\end{align}
On the other hand, using the definition \eqref{def:PhiappNS}, the bound
\eqref{bd:velocityLinfty}, and the fact that $\Omega_{\app}^{NS} = \cO_{\cZ}(\eps^2)$ 
and $\zeta_\app^{NS} = \cO(\eps)$, one verifies that $\|\nabla \Phi_{\app}^{NS}\|_{L^{\infty}}
\lesssim \eps^2$. This gives
\begin{equation}\label{bd:ANS2}
  \bigl|\jap{\left\{\Phi_\app^{NS},\ww\right\},W_\eps \ww+\varphi-\cT_\eps \varphi}_{L^2}\bigr|
  \,\lesssim\, \eps^2 \,\|\nabla\ww\|_{\cX_\eps}\|\ww\|_{\cX_\eps}\,,
\end{equation}
and \eqref{bd:ANS} results from the combination of \eqref{bd:ANS1}, \eqref{bd:ANS2}.

\smallskip
Finally, if we assume that $\sigma_1 > 0$ is small enough so that $\sigma_1 N \le (M+1)/2$,
and $\sigma_2 > 1$ large enough so that $1 + \sigma_2 \ge (M+1)/2$, it follows from \eqref{eq:tTadv}
that
\[
  \delta^{-1}\eps^{M+1-\sigma_1 N} + \delta^{-1}\eps^{1+\sigma_2} \,\lesssim\,
  \delta^{-1}\eps^{\frac{M+1}{2}} \,\le\, \delta^{s_*}\,,
\]
provided $0 < s_* \le s_0(M+1)/4$. Under this assumption the bound \eqref{bd:Aprop} is a
straightforward consequence of \eqref{bd:A1}--\eqref{bd:ANS} and Proposition~\ref{prop:Diffusion}.
\end{proof}

\subsection{The forcing terms}\label{sec:F}

To control the forcing term $\mathrm{F}$ defined in \eqref{errF}, we exploit the estimate
of the remainder $\cR_M$ given in Proposition~\ref{prop:appsol}, and we use the argument
sketched at the end of Section~\ref{sec:Arnold} to handle the term involving the
correction $\zeta$ to the vertical speed. For the latter, we can rely on the bound
established in Lemma~\ref{lem:zeta2}. 

\begin{lemma}\label{prop:F}
Let $\mathrm{F}$ be defined as in \eqref{errF}. Under the assumptions of
Proposition~\ref{prop:key}, for any $\kappa_* > 0$, there exists a constant $C_* > 0$
such that 
\begin{equation}\label{bd:F}
  |\mathrm{F}| \,\le\, C_*\bigl(\delta^{-4}\eps^{2(M+1)}+\eps^4\bigr) +
  \kappa_* D_\eps[\ww]\bigl(1+\sqrt{E_{\eps}[\ww]}\bigr)\,.
\end{equation}
\end{lemma}

\begin{proof}
Using the bound \eqref{eq:remest} on $\cR_M$ and applying Lemma~\ref{lem:withmom} to estimate
the stream function $\varphi-\cT_\eps \varphi$, we easily obtain
\begin{equation}\label{bd:F0}
  |\jap{\cR_M,W_\eps \ww +\varphi-\cT_\eps \varphi}_{L^2}| \,\lesssim\, \bigl(
  \eps^{M+1}+\delta^2\eps^2\bigr)\|\ww\|_{\cX_\eps}\,.
\end{equation}
To handle the term in $\mathrm{F}$ involving $\zeta$, we proceed as in 
Section~\ref{sec:Arnold}. We first observe that
\[
  \jap{\{\xi_1,\Omega_{\app}^E\},\varphi-\cT_\eps \varphi}_{L^2} = \jap{\partial_2\Delta
  \Psi^E_{\app},\varphi-\cT_\eps \varphi}_{L^2} = \jap{\partial_2\Psi^E_{\app},
  \ww -\cT_\eps\ww}_{L^2} = \jap{\partial_2\Phi_\app^E,\ww}_{L^2},
\]
because $\partial_2\Phi_\app^E=\partial_2(\Psi^E_{\app}-\cT_\eps\Psi^E_\app)$ by \eqref{def:PhiappE}.
It follows that
\begin{equation}\label{eq:zetaterm}
\begin{split}
  \jap{\{\xi_1,\Omega_{\app}\},W_\eps \ww +\varphi-\cT_\eps \varphi}_{L^2} \,&=\,
  \jap{(W_\eps\partial_2\Omega_{\app}^E+\partial_2\Phi_\app^E),\ww}_{L^2} \\
  &\quad\, +\delta\jap{\partial_2\Omega_{\app}^{NS},W_\eps \ww +\varphi-\cT_\eps \varphi}_{L^2}\,.
\end{split}
\end{equation}
To estimate the first term in the right-hand side, we split the domain of integration.
In the inner region $\mrI_\eps$, we have $W_\eps\partial_2\Omega_{\app}^E = \partial_2
F(\Omega_{\app}^E)$, so using the definition \eqref{def:Theta} and the bound \eqref{bd:Theta}
we obtain 
\begin{equation}\label{bd:F11}
  |\jap{\1_{\mrI_\eps}\partial_2\bigl(F(\Omega_{\app}^E)+\Phi_{\app}^E\bigr), \ww}_{L^2}| \,=\,
  |\jap{\1_{\mrI_\eps}\partial_2\Theta, \ww}_{L^2}| \,\lesssim\, \eps^{M+1}\|\ww\|_{\cX_\eps}\,.
\end{equation}
Outside $\mrI_\eps$ we rely on the bound $W_\eps|\partial_2\Omega_{\app}^E| +
|\partial_2\Phi_\app^E| \lesssim 1$, which gives
\begin{equation}\label{bd:F1}
  |\jap{\1_{\{\mathrm{II_\eps}\cup\mathrm{III}_\eps\}}(W_\eps\partial_2\Omega_{\app}^E
  +\partial_2\Phi_\app^E), \ww}_{L^2}| \,\lesssim\, \exp(-c_*\eps^{-2\sigma_1})\|\ww\|_{\cX_\eps}\,,
\end{equation}
for some $c_*>0$ sufficiently small. To treat the last term in the right-hand
side of \eqref{eq:zetaterm}, we remark that $\Omega_\app^{NS}=\cO_\cZ(\eps^2)$ and we apply
Lemma~\ref{lem:withmom} again to arrive at
\begin{equation}\label{bd:FNS}
  \bigl|\jap{\partial_2\Omega_{\app}^{NS},W_\eps \ww +\varphi-\cT_\eps \varphi}_{L^2}\bigr|
  \,\lesssim\, \eps^2\|\ww\|_{\cX_\eps}\,.
\end{equation}

Summarizing, in view of \eqref{bd:F0}--\eqref{bd:FNS}, the forcing term \eqref{errF}
satisfies
\[
  |\mathrm{F}| \,\lesssim\, \Bigl(\delta^{-2}\eps^{M+1} + \eps^2 + \frac{\eps|\zeta|}{\delta^2}\,
  \bigl(\eps^{M+1} + \delta \eps^2\bigr)\Bigr)\|\ww\|_{\cX_\eps}\,.
\]
Using in addition the bound \eqref{bd:zeta2} on $\eps\zeta$, we thus find
\begin{equation}\label{bd:Final}
  |\mathrm{F}| \,\lesssim\, \bigl(\delta^{-2}\eps^{M+1} + \eps^2\bigr) \|\ww\|_{\cX_\eps} +
  \bigl(\delta^{-1}\eps^{M+2} + \eps^3\bigr)\|\ww\|_{\cX_\eps}^2 + 
  \bigl(\eps^{M+4} + \delta\eps^5\bigr) \|\ww\|_{\cX_\eps}^3\,.
\end{equation}
Under the assumptions of Proposition~\ref{prop:key}, we know that $\delta^{-2}\eps^{M+1} \le
\delta^{s_*} \ll 1$, and the other coefficients in \eqref{bd:Final} are small too. 
Therefore, applying Young's inequality to the first term in the right-hand side,
and using the lower bound \eqref{bd:Deps} on the dissipation functional $D_\eps$,
we obtain the desired inequality \eqref{bd:F}. 
\end{proof}

\subsection{The nonlinear terms}\label{sec:NL}

Finally we estimate the nonlinear term in \eqref{errNL}. 

\begin{lemma}\label{prop:NL}
Let $\mathrm{NL}$ be defined as in \eqref{errNL}. Under the assumptions of
Proposition~\ref{prop:key}, there exist constants $C_*>0$ and $0<s_*<1$ such that 
\begin{equation}\label{bd:NL}
  |\mathrm{NL}| \,\le\, \delta^{s_*}D_\eps[\ww] + C_*\bigl(\sqrt{E_{\eps}[\ww]}+E_{\eps}[\ww]\bigr)
  D_{\eps}[\ww]\,.
\end{equation}
\end{lemma}

\begin{proof}
First of all, applying Cauchy-Schwarz's inequality, we find
\[
  \bigl|\langle \ww,\partial_1(\cT_\eps \varphi)\rangle_{L^2}\bigr| \,\le\,
  \|W_\eps^{-\frac12}\nabla \cT_\eps \varphi\|_{L^2}\|\ww\|_{\cX_\eps}\,.
\]
To estimate the right-hand side, we consider separately the regions $\mrI_\eps$ and
$\mrI_\eps^c$. Applying H\"older's inequality with $1/q+1/p=1/2$ and $q>2$, and
using Lemma~\ref{lem:velocityLq} and Corollary~\ref{cor:TepsXeps} to bound the stream
function $\cT_\eps\varphi$, we obtain
\begin{align*}
  \|\1_{\rm I_\eps}W_\eps^{-\frac12}\nabla \cT_\eps \varphi\|_{L^2} \,&\lesssim\,
  \|\1_{\rm I_\eps}(1+|\cdot|) W_\eps^{-\frac12}\|_{L^p}\|\1_{\rm I_\eps}(1+|\cdot|)^{-1}
  \nabla \cT_\eps \varphi\|_{L^q} \,\lesssim\, \eps^3\|\ww\|_{\cX_\eps}\,,\\
  \|\1_{\rm I_\eps^c}W_\eps^{-\frac12}\nabla \cT_\eps \varphi\|_{L^2} \,&\lesssim\,
  \|\1_{\1_{\rm I_\eps^c}}W_\eps^{-\frac12}\|_{L^p}\| \nabla  \varphi\|_{L^q}
  \lesssim \exp(-c_*\eps^{-2\sigma_1})\|\ww\|_{\cX_\eps}\,,
\end{align*}
for some $c_*>0$ sufficiently small. Altogether this gives
\begin{align}\label{bd:NL1}
  \frac{1}{4\eps}\,\bigl|\langle \ww,\partial_1(\cT_\eps \varphi)\rangle_{L^2}\bigr|
  \,\lesssim\, \eps^2\|\ww\|_{\cX_\eps}^2 \,\lesssim\, \eps^2D_{\eps}[\ww]\,,
\end{align}
where in the last step we used the lower bound \eqref{bd:Deps} in Proposition \ref{prop:Diffusion}. 

Next, in view of the Poisson bracket identities \eqref{eq:basic}, we have
\[
  \jap{\left\{\varphi-\cT_\eps \varphi,\ww\right\},W_\eps \ww +\varphi-\cT_\eps\varphi}_{L^2}
  \,=\, -\frac12\jap{\left\{\varphi-\cT_\eps \varphi,W_\eps\right\}\ww, \ww }_{L^2}\,.
\]
Here we use the bound \eqref{bd:velocityLinfty} in Lemma~\ref{lem:velocityLq} to control
the streamfunction, and we make the following observations concerning $\nabla W_\eps$. 
In region $\mathrm{I}_\eps$, thanks to Proposition \ref{propIW}, we can approximate
$\nabla W_\eps$ by $\nabla W_0$, which gives $|\nabla W_\eps(\xi)| \lesssim 
|\nabla W_0(\xi)| + W_0(\xi) \lesssim (1+\rho_\eps(\xi)) W_\eps(\xi)$, for 
$|\xi|\leq 2\eps^{-\sigma_1}$. In region $\mathrm{II}_\eps$ we have $\nabla W_\eps=0$ 
whereas in region $\mathrm{III}_\eps$ we know that $|\nabla W_\eps|\lesssim 
|\xi|^{\gamma-1} \rho_\eps W_\eps$. Overall, we obtain 
\begin{equation}\label{bd:NL2}
\begin{split}
  &\bigl|\jap{\left\{\varphi-\cT_\eps \varphi,W_\eps\right\}\ww,\ww }_{L^2}\bigr|
  \,\lesssim\, \|\nabla \varphi\|_{L^\infty}\bigl(\|\rho_\eps \ww\|_{\cX_\eps}
  + \|\ww\|_{\cX_\eps}\bigr)\|\ww\|_{\cX_\eps}\\
  &\qquad \,\lesssim\, \bigl(\|\nabla\ww\|_{\cX_\eps}^\frac12+\|\ww\|^\frac12_{\cX_\eps}\bigr)
  \bigl(\|\rho_\eps \ww\|_{\cX_\eps} + \|\ww\|_{\cX_\eps}\bigr)\|\ww\|^\frac32_{\cX_\eps}
  \,\lesssim\, \sqrt{E_\eps[\ww]}D_{\eps}[\ww]\,,
\end{split}
\end{equation}
where in the last inequality we used Proposition \ref{prop:Diffusion}.

Finally, for the remaining term in $\mathrm{NL}$, we use Lemma~\ref{lem:withmom} again
together with the estimate on $\zeta$ in Lemma~\ref{lem:zeta2}. We thus obtain
\begin{align}
  \notag \frac{\eps |\zeta|}{2\pi \delta}\, \bigl|\langle &\{\xi_1,\ww\},W_\eps \ww
  +\varphi-\cT_\eps\varphi\rangle_{L^2}\bigr| \,\lesssim\, \frac{\eps}{\delta}\,|\zeta|\,
  \|\nabla\ww\|_{\cX_\eps}\|\ww\|_{\cX_\eps}\\
  \label{bd:NL3}&\lesssim \bigl(\delta^{-1}\eps^{M+1}+\delta\eps^2+\eps \|\ww\|_{\cX_\eps}
  +\delta \eps^3\|\ww\|_{\cX_\eps}^2\bigr)\|\nabla\ww\|_{\cX_\eps}\|\ww\|_{\cX_\eps}\\
  \notag &\lesssim \bigl(\delta^{-1}\eps^{M+1}+\delta\eps^2+\eps \sqrt{E_{\eps}[\ww]}
  +\delta \eps^3E_{\eps}[\ww]\bigr)D_{\eps}[\ww]\,.
\end{align}
Under the assumptions of Proposition~\ref{prop:key}, we have $\delta^{-1}\eps^{M+1}
+\delta\eps^2 \le \delta^{s_*}$ for some $s_* > 0$, see Remark~\ref{rem:scales}. 
Therefore, combining the the bounds \eqref{bd:NL1}, \eqref{bd:NL2} and \eqref{bd:NL3},
we see that $\mathrm{NL}$ defined by \eqref{errNL} satisfies the estimate
\eqref{bd:NL} for some constant $C_* > 0$.
\end{proof}

\subsection{Conclusion of the proof}\label{sec:continuity}

We are now in position to conclude the proof Proposition~\ref{prop:key}, hence also 
of Theorem~\ref{th:mainNL}. Under the assumptions of Proposition~\ref{prop:key}, we
consider the solution $\ww$ of \eqref{linear0}--\eqref{eq:zeromoments} with initial data
$\ww|_{t=0}=0$. Thanks to Proposition~\ref{prop:coercivity}, we can use the
energy $E_\eps[\ww]$ defined in \eqref{def:Eeps} to control the size of $\ww$ in
the function space $\cX_\eps$. The energy evolves in time according to \eqref{EnergyIdentity},
where the quantities $\mathrm{A}$, $\mathrm{F}$, $\mathrm{NL}$ in the right-hand side 
satisfy the estimates \eqref{bd:Aprop}, \eqref{bd:Fprop}, \eqref{bd:NLprop} for
some constants $s_* > 0$, $\kappa_* > 0$, and $C_* \ge 1$. Without loss of
generality, we assume that $\kappa_* \le 1/8$  and we take $\delta > 0$ small enough
so that $\delta^{s_*} \le 1/16$. 

As long as the energy satisfies $E_\eps[\ww] \le \cE_0 := \min\bigl\{1,(16C_*)^{-2}\bigr\}$,
we have
\begin{align*}
  t\partial_t E_\eps[\ww] + D_\eps[\ww] \,&\le\, 2\delta^{s_*} D_{\eps}[\ww] +
  2\kappa_* D_{\eps}[\ww] + 2 C_* \sqrt{E_{\eps}[\ww]} D_\eps[w] + C_*\bigl(\delta^{-4}
  \eps^{2(M+1)}+\eps^4\bigr)\\
  \,&\le\, \frac18\,D_{\eps}[\ww] + \frac14\,D_{\eps}[\ww] + \frac18\,D_{\eps}[\ww] + 
  C_*\bigl(\delta^{-4}\eps^{2(M+1)}+\eps^4\bigr)\,,
\end{align*}
so that $t\partial_t E_\eps[\ww] + D_\eps[\ww]/2 \le C_*\bigl(\delta^{-4}\eps^{2(M+1)}+\eps^4\bigr)$.  
Recalling that $\eps = \eps(t) = \sqrt{\nu t}/d$, we can integrating this differential 
inequality on the time interval $(0,t)$ to obtain the bound
\[
  E_{\eps}[\ww(t)] + \frac12\int_0^t\frac{D_{\eps}[\ww(\tau)]}{\tau}\,\dd\tau
  \,\le\, \frac{C_*}{M+1}\,\delta^{-4}\eps^{2(M+1)} + \frac{C_*}{2}\,\eps^4
  \,\le\, C_*\bigl(\delta^{-4}\eps^{2(M+1)}+\eps^4\bigr)\,.
\]
According to \eqref{eq:tTadv}, the right-hand side is smaller than a fractional
power of $\delta$ as long as $t \in (0,T_\adv\delta^{-\sigma})$, so we can make it
smaller than the fixed number $\cE_0$ by taking $\delta > 0$ small enough.
In that case, the argument above holds for all $t \in (0,T_\adv\delta^{-\sigma})$,
which concludes the proof of \eqref{bd:key}.

With Proposition~\ref{prop:key} at hand, it is straightforward to conclude the proof
of Theorem~\ref{th:mainNL}. Indeed, as already noticed, the estimate \eqref{bd:wmain}
follows directly from \eqref{bd:key} in view of \eqref{bd:coercivity}, and the formula
\eqref{eq:ZpdefNL} for the vertical speed is a consequence of the decomposition
\eqref{def:perZ2}, the expression \eqref{zetaapp} of $\zeta_\app$, and the bound on 
the correction $\zeta$ in Lemma~\ref{lem:zeta2}.\qed


\appendix

\section{}\label{sec:app}

\subsection{Homogeneous polynomials}\label{ssecA1}

For any integer $n \in \NN$ we denote by $Q_n^c(x)$ and $Q_n^s(x)$ the
$n$-homogeneous polynomials on $\RR^2$ defined by
\begin{equation}\label{def:Qncs}
  Q_n^c(x) \,=\, \Re \bigl(x_1 + i x_2)^n\,, \qquad
  Q_n^s(x) \,=\, \Im \bigl(x_1 + i x_2)^n\,. 
\end{equation}
Note that $Q_n^c(\cos\theta,\sin\theta) = \cos(n\theta)$ and $Q_n^s(\cos\theta,\sin\theta) =
\sin(n\theta)$ by De Moivre's formula. For the first values of $n$ we have
\begin{align*}
  Q_0^c(x) \,&=\, 1\,, \quad Q_1^c(x) \,=\, x_1\,, \quad Q_2^c(x) \,=\, x_1^2 - x_2^2\,,
  \quad Q_3^c(x) \,=\, x_1^3 - 3x_1 x_2^2\,, \\
  Q_0^s(x) \,&=\, 0\,, \quad Q_1^s(x) \,=\, x_2\,, \quad Q_2^s(x) \,=\, 2 x_1 x_2\,,
  \hspace{18pt} Q_3^s(x) \,=\, 3x_1^2 x_2 - x_2^3\,.
\end{align*}
Assume that $x \in \RR^2$ satisfies $|x|^2 = x_1^2 + x_2^2 < 1$. Denoting $z =
x_1 + ix_2$ we find
\begin{equation}\label{eq:Qnlog2}
  \sum_{n=1}^\infty \frac{(-1)^{n-1}}{n}\,Q_n^c(x) \,=\, \Re\,\sum_{n=1}^\infty \frac{(-1)^{n-1}}{n}\,z^n
  \,=\, \Re\,\log(1+z) \,=\, \frac12 \log(1 + 2x_1 + |x|^2)\,,
\end{equation}
which is \eqref{eq:Qnlog}. Since $\partial_1 Q_n^c(x) = n\,Q_{n-1}^c(x)$ and $\partial_2 Q_n^c(x)
= -n\,Q_{n-1}^s(x)$, we can differentiate both sides of \eqref{eq:Qnlog2} to arrive at the
formulas
\begin{equation}\label{eq:Qnid}
  \sum_{n=0}^\infty (-1)^n Q_n^c(x) \,=\, \frac{1+x_1}{1+2x_1+|x|^2}\,, \qquad
  \sum_{n=1}^\infty (-1)^{n-1} Q_n^s(x) \,=\, \frac{x_2}{1+2x_1+|x|^2}\,,
\end{equation}
which were used several times in the previous sections.

\subsection{Inverting the diffusion operator}\label{ssecA2}

This section is devoted to the proof of Lemma~\ref{lem:Linvert}. Given $\kappa > 0$
and $f \in \cY$, we have the Laplace formula
\begin{equation}\label{eq:kcLinv}
  (\kappa - \cL)^{-1}f \,=\, \int_0^\infty \e^{-\kappa\tau} S(\tau) f\,\dd\tau\,,
\end{equation}
where $S(\tau) = \exp(\tau\cL)$ is the analytic semigroup in $\cY$ generated by the
selfadjoint operator $\cL$, see \cite[Appendix~A]{GalWay2002}. It is well known
that
\begin{equation}\label{eq:Srep}
  \bigl(S(\tau)f\bigr)(\xi) \,=\, \frac{1}{4\pi a(\tau)} \int_{\RR^2} \e^{-|\xi-\eta\,\e^{-\tau/2}|^2
  /4a(\tau)} f(\eta)\,\dd\eta\,, \qquad \tau > 0\,, \quad\xi \in \RR^2\,,
\end{equation}
where $a(\tau) = 1-\e^{-\tau}$. Multiplying both sides by $\e^{|\xi|^2/4}$ and 
using the identity
\[
  \frac{|\xi|^2}{4} - \frac{1}{4 a(\tau)}\,\bigl|\xi-\eta\,\e^{-\tau/2}\bigr|^2
  \,=\, \frac{|\eta|^2}{4} -\frac{1}{4 a(\tau)}\,\bigl|\eta-\xi\,\e^{-\tau/2}\bigr|^2\,,
\]
we obtain the equivalent formula
\[
  G(\xi,\tau) \,:=\, \e^{|\xi|^2/4} \bigl(S(\tau)f\bigr)(\xi) \,=\, \frac{1}{4\pi a(\tau)}
  \int_{\RR^2}\e^{-|\eta - \xi\,\e^{-\tau/2}|^2/4a(\tau)} g(\eta)\,\dd\eta\,,
\]
where $g(\eta) = \e^{|\eta|^2/4}f(\eta)$. 

We assume henceforth that $f \in \cZ$, hence $g \in \cS_*(\RR^2)$. In particular
$|g(\eta)| \le C(1+|\eta|)^N$ for some $C > 0$ and some $N \in \NN$, so that
\[
  |G(\xi,\tau)| \,\le\, \frac{C}{a(\tau)} \int_{\RR^2}(1+|\eta|)^N \e^{-|\eta - \xi\,\e^{-\tau/2}|^2/
  4a(\tau)}\,\dd\eta \,\le\, C(1+|\xi|)^N\,.
\]
Returning to \eqref{eq:kcLinv} we thus obtain
\begin{equation}\label{eq:Sest1}
  \e^{|\xi|^2/4}\bigl|\bigl((\kappa - \cL)^{-1}f\bigr)(\xi)\bigr| \,\le\, \int_0^\infty
  \e^{-\kappa\tau}|G(\xi,\tau)|\,\dd\tau \,\le\, C\kappa^{-1}(1+|\xi|)^N\,.
\end{equation}

The derivative $\partial^\alpha (\kappa - \cL)^{-1}f$ can be estimated in the same way,
for any multi-index $\alpha \in \NN^2$. If $0 < \tau < 1$ we differentiate \eqref{eq:Srep}
and integrate by parts to obtain
\[
  \partial_\xi^\alpha\bigl(S(\tau)f\bigr)(\xi) \,=\, \frac{\e^{|\alpha|\tau/2}}{4\pi a(\tau)}
  \int_{\RR^2} \e^{-|\xi-\eta\,\e^{-\tau/2}|^2/4a(\tau)} \partial_\eta^\alpha f(\eta)\,\dd\eta\,,
  \qquad 0 < \tau < 1\,.
\]
Since $f \in \cZ$ we have $\e^{|\eta|^2/4}|\partial_\eta^\alpha f(\eta)| \le C(1+|\eta|)^{N'}$ for
some $N' \in \NN$ (depending on $\alpha$), and proceeding as before we deduce that 
$\e^{|\xi|^2/4}|\partial_\xi^\alpha\bigl(S(\tau)f\bigr)(\xi)| \le C(1+|\xi|)^{N'}$. If $\tau \ge 1$,
we observe that $1 - \e^{-1} \le a(\tau) \le 1$ and we differentiate \eqref{eq:Srep}
to obtain
\[
  \bigl|\partial_\xi^\alpha\bigl(S(\tau)f\bigr)(\xi)\bigr| \,\le\, C\int_{\RR^2}
  \bigl(1 + |\xi-\eta\,\e^{-\tau/2}|\bigr)^{|\alpha|}\,\e^{-|\xi-\eta\,\e^{-\tau/2}|^2/4a(\tau)}
  |f(\eta)|\,\dd\eta\,, 
\]
which gives $\e^{|\xi|^2/4}|\partial_\xi^\alpha\bigl(S(\tau)f\bigr)(\xi)| \le
C(1+|\xi|)^{N+|\alpha|}$. Thus using \eqref{eq:kcLinv} we find, as in \eqref{eq:Sest1},
\begin{equation}\label{eq:Sest2}
  \e^{|\xi|^2/4}\bigl|\partial^\alpha\bigl((\kappa - \cL)^{-1}f\bigr)(\xi)\bigr| \,\le\,
  \int_0^\infty \e^{-\kappa\tau} \e^{|\xi|^2/4}\bigl|\partial_\xi^\alpha \bigl(S(\tau)f\bigr)(\xi)\bigr| 
  \,\dd\tau \,\le\, C\kappa^{-1}(1+|\xi|)^{N''}\,,
\end{equation}
for some $N'' \in \NN$. This shows that $(\kappa - \cL)^{-1}f \in \cZ$. \qed

\subsection{Inverting the advection operator}\label{ssecA3}

In this section, for the reader's convenience, we recall some known results
about the (partial) inverse of the operator $\Lambda$ introduced in
\eqref{def:Lambda}. In particular, we prove the second half of
Proposition~\ref{prop:Lambda}. We work in the function space $\cY$ defined by
\eqref{def:cY}, and we recall that $\Lambda$ leaves invariant the direct sum
decomposition \eqref{Ydecomp}, so that it is sufficient to consider the
restriction of $\Lambda$ to each subspace $\cY_n$. To do that, we use polar
coordinates $\xi = (r\cos\theta,r\sin\theta)$ in $\RR^2$, and we define the
radially symmetric functions
\begin{equation}\label{def:v0gh}
  v_0(r) \,=\, \frac{1}{2\pi r^2}\bigl(1-\e^{-r^2/4}\bigr)\,,\qquad
  g(r) \,=\, \frac{1}{8\pi}\,\e^{-r^2/4}\,, \qquad h(r) \,=\,
  \frac{g(r)}{v_0(r)} \,=\, \frac{r^2/4}{\e^{r^2/4}-1}\,.
\end{equation}
We observe that $\partial_r\Omega_0 = -rg$ and $\partial_r\Psi_0 = rv_0$,
where $\Omega_0,\Psi_0$ are given by \eqref{def:G}, \eqref{def:UG}. 

Since $\Lambda$ vanishes on the subspace $\cY_0$ of radially symmetric
functions, we assume henceforth that $n \ge 1$. If $\Omega \in \cY_n$
takes the form $\Omega = w(r)\cos(n\theta)$ for some function $w : \RR_+
\to \RR$, the associated stream function is $\Psi = \varphi(r)\cos(n\theta)$,
where $\varphi$ denotes the unique solution of the ordinary differential equation
\begin{equation}\label{eq:varphidef}
  \varphi'' (r) + \frac{1}{r}\,\varphi'(r) - \frac{n^2}{r^2}\varphi(r)
  \,=\, w(r)\,, \qquad r > 0\,,
\end{equation}
satisfying the boundary conditions $\varphi(r) = \cO(r^n)$ as $r \to 0$ and
$\varphi(r) = \cO(r^{-n})$ as $r \to +\infty$. A direct calculation shows that
\begin{equation}\label{eq:Lamn}
  \Lambda \Omega \,=\, \bigl\{\Psi_0\,,\,\Omega\bigr\} + \bigl\{\Psi\,,\,\Omega_0\}
  \,=\, -n\bigl(v_0 w + \varphi g\bigr)\sin(n\theta)\,.
\end{equation}
Similarly, if $\Omega = w(r)\sin(n\theta)$, then $\Psi = \varphi(r)\sin(n\theta)$
and $\Lambda \Omega = n\bigl(v_0 w + \varphi g\bigr)\cos(n\theta)$.

Suppose now that $f = b(r)\sin(n\theta) \in \cY_n$. If the inhomogeneous
differential equation
\begin{equation}\label{eq:varphidef2}
  -\varphi'' (r) - \frac{1}{r}\,\varphi'(r) + \Bigl(\frac{n^2}{r^2} - h(r)\Bigr)
  \varphi(r) \,=\, \frac{b(r)}{nv_0(r)}\,, \qquad r > 0\,,
\end{equation}
has a solution satisfying the boundary conditions, we can define
$\Omega = w(r)\cos(n\theta)$ with 
\begin{equation}\label{def:wsol}
  w(r) \,=\, -\varphi(r)h(r) - \frac{b(r)}{nv_0(r)}\,.
\end{equation}
Then \eqref{eq:varphidef} is satisfied, and it follows from \eqref{eq:Lamn} that
$\Lambda\Omega = f$. The same conclusion holds if $f = b(r)\cos(n\theta)$ and
$\Omega = -w(r)\sin(n\theta)$. 

So the invertibility of the operator $\Lambda$ in the subspace $\cY_n$ is
reduced to the solvability of the ODE \eqref{eq:varphidef2}. If $n \ge 2$, the
coefficient $n^2/r^2 - h(r)$ is positive, which ensures that equation
\eqref{eq:varphidef2} always has a unique solution satisfying the boundary
conditions. This leads to the following statement, where $\cZ$ denotes the
function space introduced in \eqref{def:calZ}.

\begin{lemma}\label{lem:Lambda} {\bf \cite{Gallay2011}}
If $n\ge 2$ and $f\in \cY_n\cap \cZ$, there exists a unique $\Omega\in \cY_n
\cap \cZ$ such that $\Lambda\Omega =f$. Moreover, if $f=b(r)\sin(n\theta)$
(respectively, $f = b(r)\cos(n\theta)$) then $\Omega = w(r)\cos(n\theta)$
(respectively, $\Omega = -w(r)\sin(n\theta)$) where $w$ is defined
by \eqref{def:wsol} with $\varphi$ given by \eqref{eq:varphidef2}. 
\end{lemma}

If $n = 1$, the homogeneous ODE \eqref{eq:varphidef2} with $b = 0$ has a
nontrivial solution $\varphi = rv_0$ satisfying the boundary conditions.
As a consequence, the inhomogeneous equation can be solved only if
the source term satisfies $\int_0^\infty b(r)r^2\,\dd r = 0$, and the solution
is not unique. This solvability condition ensures that the function
$f=b(r)\sin\theta$ (or $f=b(r)\cos\theta$) belongs to the subspace $\cY_1'$
defined by
\begin{equation}\label{eq:Y1Kerperp}
  \cY_1' \,=\, \cY_1\cap \Ker(\Lambda)^\perp \,=\, \Bigl\{f\in \cY_1 \, : \,
  \int_{\RR^2}\xi_1f(\xi)\dd\xi = \int_{\RR^2}\xi_2f(\xi)\dd\xi = 0\Bigr\}\,,
\end{equation}
see also Remark~\ref{rem:Lambda}. This leads to the following
result, which complements Lemma~\ref{lem:Lambda}. 

\begin{lemma}\label{lem:Lambda2} {\bf \cite{GaSring}}
If $n = 1$ and $f\in \cY_1'\cap \cZ$, there exists a unique $\Omega\in \cY_1'
\cap \cZ$ such that $\Lambda\Omega =f$. Moreover, if $f=b(r)\sin\theta$
(respectively, $f = b(r)\cos\theta$) then $\Omega = w(r)\cos\theta$
(respectively, $\Omega = -w(r)\sin\theta$) where $w$ is defined
by \eqref{def:wsol} with $\varphi$ given by \eqref{eq:varphidef2}. 
\end{lemma}

\begin{remark}\label{rem:sym}
In the lemmas above, the assumption that $f \in \cY_n$ (respectively,
$f \in \cY_1'$) implies that $f \in \Ker(\Lambda)^\perp$, but does not
ensure that $f \in \Ran(\Lambda)$ because $f(r)$ may not decay to
zero sufficiently rapidly as $r \to +\infty$. This problem disappears
if we assume in addition that $f \in \cZ$, in which case $f \in \Ran(\Lambda)$
and the unique preimage $\Omega = \Lambda^{-1}f$ in $\cY_n$ (respectively,
in $\cY_1'$) still belongs to $\cZ$. Note also that, if $f$ is an
odd (respectively, even) function of $\xi_2$, then $\Omega$ is
an even (respectively, odd) function of $\xi_2$. 
\end{remark}

\subsection{Estimates for the stream function}\label{ssecA4}

In this final section we give a short proof of
Lemmas~\ref{lem:velocityLq}--\ref{lem:withmom} and of Corollary~\ref{cor:TepsXeps}. 

\begin{proof}[Proof of Lemma~\ref{lem:velocityLq}]
In view of \eqref{UPsirel} we have, for all $\xi \in \RR^2$, 
\begin{equation}\label{bd:phinabphi}
  |\varphi(\xi)| \,\lesssim\, \int_{\RR^2} \bigl|\log|\xi-\eta|\bigr|\,|\ww(\eta)|\,\dd\eta\,, \qquad
  |\nabla\varphi(\xi)| \,\lesssim\, \int_{\RR^2} \frac{1}{|\xi-\eta|}\,|\ww(\eta)|
  \,\dd\eta\,.
\end{equation}
To prove the logarithmic upper bound on $|\varphi|$, we use the first inequality in
\eqref{bd:phinabphi} together with the crude estimates 
\[
  \bigl|\log|\xi-\eta|\bigr| \,\lesssim\, \begin{cases} |\xi-\eta|^{-1/2} & 
  \hbox{ if } \,|\xi-\eta| \le 1\,,\\ \log(2+|\xi|)\log(2+|\eta|) & 
  \hbox{ if } \,|\xi-\eta| \ge 1\,.\end{cases}
\]
We thus find
\[
  |\varphi(\xi)| \,\lesssim\, \int_{\{|\xi-\eta| \le 1\}}\frac{1}{|\xi-\eta|^{1/2}}|w(\eta)|\,\dd \eta
  \,+\, \log(2+|\xi|)\int_{\{|\xi-\eta| \ge 1\}}\log(2+|\eta|)|w(\eta)|\,\dd \eta\,.
\]
The first integral is bounded by $C\|\ww\|_{L^2} \le C\|\ww\|_{\cX_\eps}$, and the second
one by $C\|w\|_{\cX_\eps}\log(2+|\xi|)$ in view of \eqref{bd:Weps} and H\"older's inequality. 
Altogether we see that $|\varphi(\xi)| \le C\|w\|_{\cX_\eps}\log(2+|\xi|)$, which implies
in particular that $\|(1+|\cdot|)^{-1}\varphi\|_{L^q} \le C\|w\|_{\cX_\eps}$ for any $q > 2$.

To conclude the proof of \eqref{bd:velocityLq}, we estimate $|\nabla\varphi|$ in
\eqref{bd:phinabphi} using the classical Hardy-Littlewood-Sobolev inequality,
see \cite[Lemma 2.1]{GalWay2002}.  Given $1<p<2<q<\infty$ with $1/p = 1/q+1/2$,
the HLS inequality shows that $\norm{\nabla \varphi}_{L^q} \lesssim \norm{\ww}_{L^p}
\lesssim \norm{\ww}_{\cX_\eps}$, where in the last step we used the fact that
$\cX_\eps\hookrightarrow L^r(\RR^2)$ for any $r \in [1,2]$.

It remains to establish \eqref{bd:velocityLinfty}. Starting from \eqref{bd:phinabphi}
we observe that
\[
  (1+|\xi|)|\nabla\varphi(\xi)| \,\lesssim\, \int_{\RR^2}\Bigl(1 + \frac{1+|\eta|}{|\xi-\eta|}
  \Bigr)\,|w(\eta)|\,\dd \eta \,\lesssim\, \|\ww\|_{\cX_\eps} + \int_{\{|\xi-\eta| \le 1\}}
  \frac{1+|\eta|}{|\xi-\eta|}\,|w(\eta)| \,\dd \eta\,.
\]
The last integrand is equal to $|\xi-\eta|^{-1}\bigl((1+|\eta|)|\ww(\eta)|^{1/2}\bigr)
|\ww(\eta)|^{1/2}$, so we can apply the trilinear H\"older inequality with exponents
$8/5$, $4$, $8$ to obtain
\[
  \int_{\{|\xi-\eta| \le 1\}}\frac{1+|\eta|}{|\xi-\eta|}\,|w(\eta)| \,\dd \eta
  \,\lesssim\, \bigl\|(1+|\cdot|)^2\ww \bigr\|_{L^2}^{1/2}\|\ww\|_{L^4}^{1/2}
  \,\lesssim\, \norm{\ww}_{\cX_\eps}^{1/2} \bigl(\norm{\nabla\ww}_{\cX_\eps}^{1/2} +
  \norm{\ww}_{\cX_\eps}^{1/2}\bigr)\,,
\]
where in the last step we used the embedding $H^1(\RR^2)\hookrightarrow L^4(\RR^2)$.
\end{proof}

\begin{proof}[Proof of Lemma~\ref{lem:withmom}]
In view of Lemma~\ref{lem:velocityLq}, it is sufficient here to estimate $\varphi(\xi)$
and $\nabla\varphi(\xi)$ for $|\xi| \ge 1$. Using \eqref{UPsirel} and the assumptions
on the moments of $\ww$, we find
\[
  \varphi(\xi) \,=\, \frac{1}{4\pi} \int_{\RR^2}\biggl\{\log\Bigl(1 - \frac{2\xi\cdot\eta}{
  |\xi|^2} + \frac{|\eta|^2}{|\xi|^2}\Bigr) + \frac{2\xi\cdot\eta}{|\xi|^2}\biggr\}
  w(\eta)\dd\eta \,=:\, \varphi_1(\xi) + \varphi_2(\xi)\,.
\]
Here the decomposition $\varphi = \varphi_1 + \varphi_2$ is obtained by splitting
the integration domain $\RR^2$ according to whether $|\eta| \le |\xi|/4$ (for $\varphi_1)$
or $|\eta| > |\xi|/4$ (for $\varphi_2$). A simple Taylor expansion shows that
\[
  \biggl|\log\Bigl(1 - \frac{2\xi\cdot\eta}{|\xi|^2} + \frac{|\eta|^2}{|\xi|^2}\Bigr)
  + \frac{2\xi\cdot\eta}{|\xi|^2}\biggr| \,\le\, C\,\frac{|\eta|^2}{|\xi|^2}\,,
  \qquad \hbox{when}~\,|\eta| \,\le\, \frac{|\xi|}{4}\,,
\]
and this readily implies that $|\xi|^2|\varphi_1(\xi)| \le C\|\ww\|_{\cX_\eps}$.
On the other hand, we observe that
\[
  \varphi_2(\xi) \,=\, \frac{1}{4\pi}\int_{\{|\eta| > |\xi|/4\}} \frac{1}{|\eta|^N}
  \Bigl(\log|\xi-\eta|^2 - \log|\xi|^2 + \frac{2\xi\cdot\eta}{|\xi|^2}\Bigr)
  |\eta|^N w(\eta) \dd \eta\,,
\]
where the integer $N \in \NN$ is arbitrary. Proceeding as in the proof of
Lemma~\ref{lem:velocityLq} and using the fact that $|\eta|^N w(\eta)$ decays
rapidy as $|\eta| \to \infty$, we deduce that $|\xi|^N|\varphi_2(\xi)| \le
C\|w\|_{\cX_\eps}\log(2+|\xi|)$. Taking $N > 2$, we conclude that 
$(1+|\xi|)^2|\varphi(\xi)| \le C\|w\|_{\cX_\eps}$ as asserted in \eqref{bd:streammom}. 

We proceed similarly to estimate $\nabla\varphi(\xi)$ for $|\xi| \ge 1$. 
Our starting point is the formula
\[
  \nabla\varphi(\xi) \,=\, \frac{1}{2\pi} \int_{\RR^2}\biggl\{(\xi-\eta)
  \Bigl(\frac{1}{|\xi-\eta|^2} - \frac{1}{|\xi|^2}\Bigr)- \frac{2\xi\cdot\eta}{|\xi|^4}\,
  \xi\biggr\}w(\eta)\dd\eta \,=:\, \psi_1(\xi) + \psi_2(\xi)\,, 
\]
where $\psi_1(\xi), \psi_2(\xi)$ correspond to integrating over the regions
$|\eta| \le |\xi|/4$ and $|\eta| > |\xi|/4$, respectively. Again, a
Taylor expansion shows that
\[
  \biggl|(\xi-\eta)\Bigl(\frac{1}{|\xi-\eta|^2} - \frac{1}{|\xi|^2}\Bigr)-
  \frac{2\xi\cdot\eta}{|\xi|^4}\,\xi\biggr| \,\le\, C\,\frac{|\eta|^2}{|\xi|^3}\,,
  \qquad \hbox{when}~\,|\eta| \,\le\, \frac{|\xi|}{4}\,,
\]
so that $|\xi|^3 |\psi_1(\xi)| \le C\|\ww\|_{\cX_\eps}$. On the other hand,
given any $N \in \NN$, we have
\[
  |\psi_2(\xi)| \,\le\, C \int_{\{|\eta| > |\xi|/4\}} \frac{1}{|\eta|^N}
  \Bigl(\frac{1}{|\xi-\eta|} + \frac{|\eta|}{|\xi|^2}\Bigr)|\eta|^N w(\eta) \dd \eta\,.
\]
Proceeding as in the proof of Lemma~\ref{lem:velocityLq}, we deduce that
\[
  \|\,|\cdot|^N \psi_2\|_{L^q} \,\le\, C\|\ww\|_{\cX_\eps}\,, \quad \hbox{and} \quad
  \|\,|\cdot|^N \psi_2\|_{L^\infty} \,\le\, C\bigl(\|\nabla\ww\|_{\cX_\eps}^{1/2} +
  \|\ww\|_{\cX_\eps}^{1/2}\bigr)\|\ww\|_{\cX_\eps}^{1/2}\,.
\]
When $N = 2$, the first estimate also holds for $\psi_1$, hence for $\nabla\varphi =
\psi_1 + \psi_2$, which concludes the proof of \eqref{bd:streammom}. Similarly, the
second estimate also holds for $\psi_1$ when $N = 3$, which implies that $\nabla\varphi$
satisfies \eqref{bd:velocitymom}. 
\end{proof}

\begin{proof}[Proof of Corollary~\ref{cor:TepsXeps}]
Denoting $\eta = \widetilde{\xi} - \eps^{-1}e_1$, we have $\bigl(\cT_\eps\varphi\bigr)(\xi)
= \varphi(\eta)$ by \eqref{def:Teps}. We assume here that $\xi \in \mrI_\eps$, so that
$|\xi| \le 2\eps^{-\sigma_1} \ll \eps^{-1}$ if $\eps$ is sufficiently small. This in
turn implies that $(2\eps)^{-1} \le |\eta| \le 2\eps^{-1}$. In particular, we have
\[
  \bigl|\mathbbm{1}_{\mrI_\eps}(\xi) \bigl(\cT_\eps\varphi)(\xi)\bigr| \,=\,
  \Bigl|\mathbbm{1}_{\mrI_\eps}(\xi)\frac{1}{(1+|\eta|)^2}\,(1+|\eta|)^2\varphi(\eta)
  \Bigr| \,\le\, C\eps^2 \bigl\|(1+|\cdot|)^2\varphi\bigr\|_{L^\infty}
  \,\le\, C\eps^2 \|\ww \|_{\cX_\eps}\,,
\]
where the last step follows from \eqref{bd:streammom}. This proves the first inequality
in \eqref{bd:UTesp}. On the other hand, the proof of Lemma~\ref{lem:withmom} shows
that $\nabla\varphi = \psi_1 + \psi_2$ where
\[
  \bigl\|(1+|\cdot|)^3\psi_1\bigr\|_{L^\infty} + \bigl\|(1+|\cdot|)^3\psi_2\bigr\|_{L^q}
  \,\le\, C\|\ww\|_{\cX_\eps}\,,
\]
for any $q > 2$. As a consequence, we have
\begin{align*}
  \bigl\|\mathbbm{1}_{\mrI_\eps}(\xi)(1+|\xi|)^{-1}\psi_1(\eta)\bigl\|_{L^q} \,&\le\,
  \bigl\|\mathbbm{1}_{\mrI_\eps}(\xi)\frac{1}{(1+|\eta|)^3}\,(1+|\eta|)^3
  \psi_1(\eta)\bigl\|_{L^\infty} \,\le\, C \eps^3\|\ww\|_{\cX_\eps}\,, \\
  \bigl\|\mathbbm{1}_{\mrI_\eps}(\xi)(1+|\xi|)^{-1}\psi_2(\eta)\bigl\|_{L^q} \,&\le\,
  \bigl\|\mathbbm{1}_{\mrI_\eps}(\xi)\frac{1}{(1+|\eta|)^3}\,(1+|\eta|)^3
  \psi_2(\eta)\bigl\|_{L^q} \,\le\, C \eps^3\|\ww\|_{\cX_\eps}\,.
\end{align*}
Since $\nabla_\xi \bigl(\cT_\eps\varphi\bigr)(\xi) = \widetilde{\nabla}_\eta \varphi(\eta)
= \widetilde{\psi_1}(\eta) + \widetilde{\psi_2}(\eta)$, this implies the second
inequality in \eqref{bd:UTesp}.
\end{proof}

\medskip\noindent{\bf Acknowledgements.} The research of MD was supported by the
Swiss State Secretariat for Education, Research and lnnovation (SERI) under
contract number MB22.00034 through the project TENSE and by the Swiss National
Science Foundation (SNF Ambizione grant PZ00P2\_223294). The research of TG is
partially supported by the grant BOURGEONS ANR-23-CE40-0014-01 of the French
National Research Agency.

\bibliographystyle{siam}

\begin{thebibliography}{10}

\bibitem{Arnold66}
{\sc V.~Arnold}, {\em Sur la g\'{e}om\'{e}trie diff\'{e}rentielle des groupes
  de {L}ie de dimension infinie et ses applications \`a l'hydrodynamique des
  fluides parfaits}, Ann. Inst. Fourier (Grenoble), 16 (1966), pp.~319--361.

\bibitem{BedMas2014}
{\sc J.~Bedrossian and N.~Masmoudi}, {\em Existence, uniqueness and {Lipschitz}
  dependence for {Patlak}-{Keller}-{Segel} and {Navier}-{Stokes} in {{\(\mathbb
  R^2\)}} with measure-valued initial data}, Arch. Ration. Mech. Anal., 214
  (2014), pp.~717--801.

\bibitem{Bur88}
{\sc G.~R. Burton}, {\em Steady symmetric vortex pairs and rearrangements},
  Proc. R. Soc. Edinb., Sect. A, Math., 108 (1988), pp.~269--290.

\bibitem{BNLL13}
{\sc G.~R. Burton, H.~J. Nussenzveig~Lopes, and M.~C. Lopes~Filho}, {\em
  Nonlinear stability for steady vortex pairs}, Commun. Math. Phys., 324
  (2013), pp.~445--463.

\bibitem{CLZ21}
{\sc D.~Cao, S.~Lai, and W.~Zhan}, {\em Traveling vortex pairs for 2d
  incompressible {Euler} equations}, Calc. Var. Partial Differ. Equ., 60
  (2021), p.~16.
\newblock Id/No 190.

\bibitem{CeSeis24}
{\sc S.~Ceci and C.~Seis}, {\em On the dynamics of vortices in viscous 2d
  flows}, Math. Ann., 388 (2024), pp.~1937--1967.

\bibitem{CeSe18}
{\sc D.~Cetrone and G.~Serafini}, {\em Long time evolution of fluids with
  concentrated vorticity and convergence to the point-vortex model}, Rend. Mat.
  Appl., VII. Ser., 39 (2018), pp.~29--78.

\bibitem{DdPMP23}
{\sc J.~D{\'a}vila, M.~del Pino, M.~Musso, and S.~Parmeshwar}, {\em {Global in
  time vortex configurations for the $2$D Euler equations}}, preprint
  arXiv:2310.07238,  (2023).

\bibitem{DelRos2009}
{\sc I.~Delbende and M.~Rossi}, {\em The dynamics of a viscous vortex dipole},
  Phys. Fluids, 21 (2009), p.~15.
\newblock Id/No 073605.

\bibitem{DD25}
{\sc M.~Dolce and M.~Donati}, {\em Viscous vortex crystals}.
\newblock In preparation, 2025.

\bibitem{FoxDav2010}
{\sc S.~Fox and P.~A. Davidson}, {\em Freely decaying two-dimensional
  turbulence}, J. Fluid Mech., 659 (2010), pp.~351--364.

\bibitem{GaGa2005}
{\sc I.~Gallagher and T.~Gallay}, {\em Uniqueness for the two-dimensional
  {Navier}-{Stokes} equation with a measure as initial vorticity}, Math. Ann.,
  332 (2005), pp.~287--327.

\bibitem{Gallay2011}
{\sc T.~Gallay}, {\em Interaction of vortices in weakly viscous planar flows},
  Arch. Ration. Mech. Anal., 200 (2011), pp.~445--490.

\bibitem{GaSarnold}
{\sc T.~Gallay and V.~{\v{S}}ver{\'a}k}, {\em Arnold's variational principle
  and its application to the stability of planar vortices}, Anal. PDE, 17
  (2024), pp.~681--722.

\bibitem{GaSring}
\leavevmode\vrule height 2pt depth -1.6pt width 23pt, {\em Vanishing viscosity
  limit for axisymmetric vortex rings}, Invent. Math., 237 (2024),
  pp.~275--348.

\bibitem{GalWay2002}
{\sc T.~Gallay and C.~E. Wayne}, {\em Invariant manifolds and the long-time
  asymptotics of the {Navier}-{Stokes} and vorticity equations on ${\RR^2}$},
  Arch. Ration. Mech. Anal., 163 (2002), pp.~209--258.

\bibitem{GalWay2005}
\leavevmode\vrule height 2pt depth -1.6pt width 23pt, {\em Global stability of
  vortex solutions of the two-dimensional {Navier}-{Stokes} equation}, Commun.
  Math. Phys., 255 (2005), pp.~97--129.

\bibitem{Gr77}
{\sc W.~Gr{\"o}bli}, {\em Specielle {Probleme} {\"u}ber die {Bewegung}
  gradliniger paralleler {Wirbelf{\"a}den}}.
\newblock PhD Thesis, Georg-August-Universit{ä}t G{ö}ttingen, Z{ü}rcher und
  Furrer, Zurich, 1877.

\bibitem{HaNaFu2018}
{\sc U.~Habibah, H.~Nakagawa, and Y.~Fukumoto}, {\em Finite-thickness effect on
  speed of a counter-rotating vortex pair at high {R}eynolds number}, Fluid
  Dyn. Res., 50 (2018), p.~27.
\newblock Id/No 031401.

\bibitem{HmMa17}
{\sc T.~Hmidi and J.~Mateu}, {\em Existence of corotating and counter-rotating
  vortex pairs for active scalar equations}, Commun. Math. Phys., 350 (2017),
  pp.~699--747.

\bibitem{DiVer2002}
{\sc S.~Le~Diz\`es and A.~Verga}, {\em Viscous interactions of two co-rotating
  vortices before merging}, J. Fluid Mech., 467 (2002), pp.~389--410.

\bibitem{Maekawa2011}
{\sc Y.~Maekawa}, {\em Spectral properties of the linearization at the
  {Burgers} vortex in the high rotation limit}, J. Math. Fluid Mech., 13
  (2011), pp.~515--532.

\bibitem{Mar98}
{\sc C.~Marchioro}, {\em On the inviscid limit for a fluid with a concentrated
  vorticity}, Commun. Math. Phys., 196 (1998), pp.~53--65.

\bibitem{MarPul94}
{\sc C.~Marchioro and M.~Pulvirenti}, {\em Mathematical theory of
  incompressible nonviscous fluids}, vol.~96 of Appl. Math. Sci., New York, NY:
  Springer-Verlag, 1994.

\bibitem{McW84}
{\sc J.~C. Mcwilliams}, {\em The emergence of isolated coherent vortices in
  turbulent flow}, J. Fluid Mech., 146 (1984), pp.~21--43.

\bibitem{Meunier2005}
{\sc P.~Meunier, S.~L. Diz\`es, and T.~Leweke}, {\em Physics of vortex
  merging}, C. R. Phys., 6 (2005), pp.~431--450.

\bibitem{ZZ25}
{\sc P.~Zhang and Y.~Zhang}, {\em Long time evolution of a pair of 2d viscous
  point vortices}.
\newblock Preprint, {arXiv}:2510.03991, 2025.

\end{thebibliography}

\end{document}